\documentclass[11pt,a4paper]{article}
\usepackage{a4wide,amsfonts,amsmath,latexsym,amsthm,amssymb,euscript,eufrak,graphicx,units,mathrsfs,setspace}

\usepackage[pdftex]{color}
\definecolor{darkblue}{rgb}{0.3,0.3,0.7}

\usepackage{float}
\newfloat{figure}{H}{lof}
\floatname{figure}{\figurename}

\usepackage[french,english]{babel}
\usepackage[T1]{fontenc}
\usepackage[utf8]{inputenc}
\usepackage{graphicx}
\usepackage{upgreek}
\usepackage{marginnote}

\DeclareMathAlphabet{\eufrak}{U}{}{}{}  
\SetMathAlphabet\eufrak{normal}{U}{euf}{m}{n}
\SetMathAlphabet\eufrak{bold}{U}{euf}{b}{n}

\newtheorem{prop}{Proposition}[section]

\newtheorem{theorem}[prop]{Theorem}
\newtheorem{lemma}[prop]{Lemma}
\newtheorem{corollary}[prop]{Corollary}
\newtheorem{proposition}[prop]{Proposition}

\newtheorem{assumption}[prop]{Assumption}

\theoremstyle{definition}

\newtheorem{remark}[prop]{Remark}
\newtheorem{definition}[prop]{Definition}

\numberwithin{equation}{section}

\providecommand{\pg}{\geqslant}      
\providecommand{\pp}{\leqslant}      
\renewcommand{\textless}{\langle}  
\renewcommand{\textgreater}{\rangle}  
\newcommand{\f}[1]{\EuScript{F}_{#1}}   
\newcommand{\g}[1]{\EuScript{G}_{#1}} 
\newcommand{\dom}[1]{\mathrm{dom \ }{#1}} 
\newcommand{\var}{\mathrm{var}} 
\newcommand{\id}[2]{\{{#1}:{#2}\}} 

\def\D{\Delta} 
\def\F{\EuScript{F}} 
\def\Gm{\mathrm G} 
\def\Cm{\mathrm C} 
\def\Wm{\mathrm W} 
\def\Pm{\mathrm P} 
\def\Xm{\mathrm X} 
\def\Dm{\mathrm D} 
\def\wDm{\widetilde{\mathrm D}} 
\def\Um{\mathrm U} 
\def\Hm{\mathrm H} 
\def\J{\mathrm J} 
\def\Jt{\mathrm J}  
\def\Y{\mathrm Y} 
\def\Zm{\mathrm Z} 
\def\Fm{\mathrm F} 
\def\Gm{\mathrm G} 
\def\Km{\mathrm K} 
\def\Rm{\mathrm R} 
\def\Zt{\overline{\mathrm Z}} 
\def\Ztm{\widetilde{\mathrm Z}} 
\def\Yt{\overline{\mathrm Y}} 
\def\Xm{\mathrm X} 
\def\Nm{\mathrm N} 
\def\V{\mathrm V} 
\def\B{\mathrm B} 
\def\Xm{\mathrm X} 
\def\Tm{\mathrm T} 
\def\DD{\mathbf D} 
\def\P{\mathbf P} 
\def\gV{\mathbf V} 
\def\Q{\mathbf Q} 
\def\hk{\mathfrak h} 
\def\l{\ell} 
\def\d{\mathrm d} 
\def\St{\overline{\mathrm S}} 
\def\Vt{\overline{\mathrm V}} 
\def\S{\mathrm S} 
\def\Xb{\mathbb X} 
\def\L{\mathrm L} 
\def\E{\mathrm E} 
\def\A{\mathrm A} 
\def\B{\mathrm B} 
\def\vp{\varphi} 
\def\rem{\mathrm{rem}} 
\def\tk{\tau} 
\def\fk{\mathfrak f} 
\def\gk{\mathfrak g} 
\def\rk{\mathfrak r} 
\def\uk{\mathfrak u} 
\def\vk{\mathfrak v} 
\def\car{\mathbf 1} 
\def\esp#1{\mathbf E\left[#1\right]} 
\def\esph#1{\widehat{\mathbf E}\left[#1\right]} 

\def\cyl{\mathcal S}
\def\I{\mathcal{J}}
\def\R{\mathbf{R}}
\def\N{\mathbf{N}}
\def\Z{\mathbf{Z}}
\def\O{\Omega} 

\def\T{\mathbf T}
\def\1{\textbf{1}}

\newcommand{\be}{\begin{equation}}
\newcommand{\ee}{\end{equation}}
\newcommand{\bde}{\begin{displaymath}}
\newcommand{\ede}{\end{displaymath}}
\newcommand{\beq}{\begin{eqnarray*}}
\newcommand{\eeq}{\end{eqnarray*}}

\newcommand{\beqa}{\begin{eqnarray}}
\newcommand{\eeqa}{\end{eqnarray}}
\newcommand{\bel }{\left\{\begin{array}{ll}}
\newcommand{\eel}{\cr \end{array} \right.}

\newcommand{\bex}{\begin{ex} \rm }
\newcommand{\eex}{\end{ex}}

\definecolor{ying}{rgb}{0.8, 0.0, 0.04}

\date{}

\begin{document}

\title{\textbf{Malliavin calculus for marked binomial processes: portfolio optimisation in the trinomial model and compound Poisson approximation}}
\author{Hélène Halconruy\footnote{Mathematics Research Unit, Luxembourg University, Luxembourg. \textbf{e}-mail: helene.halconruy$\at$uni.lu}}
\maketitle

\begin{abstract}
\noindent
In this paper we develop a stochastic analysis for marked binomial processes, that can be viewed as the discrete analogues of marked Poisson processes. The starting point is the statement of a chaotic expansion for square-integrable (marked binomial) functionals, prior to the elaboration of a Markov-Malliavin structure within this framework. We take advantage of the new formalism to deal with two main applications. First,  we revisit the Chen-Stein method for the (compound) Poisson approximation which we perform in the paradigm of the built Markov-Malliavin structure, before studying in the second one the problem of portfolio optimisation in the trinomial model.
\end{abstract}


 \allowdisplaybreaks


\section{Introduction}

This paper is motivated by two applications: the compound Poisson approximation by a revisited Chen-Stein method and a problem of portfolio optimisation in the trinomial model. Apparently unrelated, they are both possible by-products of Malliavin calculus. The eponymous theory designed by Paul Malliavin in the 70's was initially elaborated to provide an infinite dimensional differential calculus for the Wiener space, and further extended to other settings. Thus one can find in the literature stochastic variational calculus for Gaussian processes in general (see Nualart \cite{nualart.book}, Nourdin and Peccati \cite{Nourdin:2012fk}), Poisson processes (see Bichteler \textit{et al.}  \cite{bichteler1987malliavin} for a variational approach, Nualart and Vives \cite{nualart88_1} or Privault \cite{privault93} for a chaotic approach),
Lévy processes (see Nualart and Schoutens \cite{NualartSchoutens}), Rademacher processes (see Privault \cite{Privault_stochastica}), and more recently for independent random variables (developed independently by Duerinckx, Gloria and Otto in \cite{DuerinckxGloriaOtto2020}, Decreusefond and Halconruy in \cite{DecreusefondHalconruy}).
Even if the multiplicity of approaches and the variety of canonical spaces on which Malliavin calculus operates appeared at first glance to be an obstacle to a complete unifying theory, one can find nevertheless a common terminology for all these formalisms around the notions of Malliavin operators (gradient $\Dm$, divergence $\delta$, Ornstein-Uhlenbeck operator $\L$ and semi-group $(\Pm_t)_{t\in\R_+}$) and the fundamental relationship between the gradient operator and divergence (defined as the gradient adjoint): the integration by parts formula. In fact, this is part of a deeper structure in which the $\L$ operator would be, by virtue of its dual role, the cornerstone: as the Ornstein-Uhlenbeck operator, it generates an underlying Markovian structure, and, as the Laplacian operator, it satisfies $\L=-\delta \Dm$. This raises, via the integration by parts formula, a Dirichlet form $\mathcal E$, so that any Malliavin equipment comes with a Dirichlet structure (see Decreusefond \cite{DecreusefondSDM}). Besides, the form $\mathcal E$ appears as the energy function associated to the \textit{carré du champ} operator $\Gamma$. The naturally emerging Gamma-Malliavin structure provides an ideal framework for future applications. Indeed the operator $\Gamma$ is one efficient to deduce quantitative limits by means of... \\

... \textit{Stein's method}, that frames the first application we have investigated. Initially designed to quantify the errors in the normal approximation by sums of random variables
having a stationary dependence structure, Stein's method stood out as one (not to say the one) efficient way to derive distance bounds between two probability measures (referred to as initial distribution and target distribution) with respect to a certain probabilistic metrics. It seems to be split into
two stages; the first is to take advantage of the target law characterisation to convert the original problem into that of bounding some functional of the initial space, whereas the second one aims at developing tools to tackle with this new expression. This latter can be performed by using exchangeable pairs (see Barbour and Chen \cite{barbour_introduction}) or other forms of couplings such as zero- or size-biased couplings. In a path-breaking work, Nourdin and Peccati (see
\cite{NourdinPeccati2009},\cite{Nourdin:2012fk}) showed that the transformation step can be advantageously made simple using integration by parts in the sense of Malliavin calculus, and by the same gave an intersection to the two theories. \\This approach is efficient provided there exists a Malliavin gradient on the initial space. This was handled to derive bounds for normal approximation by functionals acting on spaces equipped with a Malliavin structure such as Wiener chaoses (see Nourdin and Peccati \cite{Nourdin_2008}, with Reinert \cite{NourdinPeccatiReinert2009}, \cite{Nualart2014finitechaos}, V{\'\i}quez \cite{Viquez2018normal}), Poisson functionals (Lachièze-Rey and Peccati \cite{lachiezereypeccati},  Peccati \textit{et al.} \cite{peccati2010}, Schulte \cite{Schulte2016Kolomogorov}), functionals of Rademacher (see Reinert \textit{et al.} \cite{NourdinPeccatiReinert}, Zheng \cite{Zheng2017NormalAA}). On the other hand, Stein's method was developed to other target distributions; its most famous declination is undoubtedly the \textit{Chen-Stein} method that deals with (compound) Poisson approximation (see Chen \cite{Chen1974}). This was first performed using location arguments by the introduction of "neighbourhood of dependence" sets to deal with Poisson approximation (see Arratia \textit{et al} \cite{ArratiaGoldsteinGordon1990} and \cite{ArratiaGoldsteinGordon1989Twomoments}, Barbour \textit{et al} \cite{BarbourLarsJanson1992}) or compound Poisson one (see Barbour \textit{et al} \cite{BarbourChenLoh1992}, \cite{BarbourUtev1998}). Instigated first by Peccati in \cite{Peccati2011chen}, the Chen-Stein method was also combined to Malliavin calculus to provide Poisson approximation bounds for point processes (see Torrisi \cite{Torrisi2017Poisson}), Rademacher functionals (Krokowski \cite{Krokowski2015}, Privault and Torrisi \cite{PrivaultTorrisi2015}), or multiple integrals (see Bourguin and Peccati \cite{BourguinPeccati2016}). Last, as mentioned above, the Stein-Malliavin method was recently improved by exploiting the underlying Markovian structure (and in particular by using the carré du champ operator) to get quantitative limits by overcoming possible combinatorial difficulties that may arise from the use of multiplication formulae on configuration spaces. Finding its theoretical roots in the innovative works of Azmoodeh, Campese and Poly \cite{AzmoodehCampesePoly}, Azmoodeh \textit{et al.} \cite{AzmoodehMalicetMijoulePoly2016}, Ledoux \cite{Ledoux2012}, this approach successfully succeeds exploiting  the powerful techniques of Markov generators within Gamma-calculus. \\
Our first application takes place in this landscape. We propose a new point of view to address the problem of Poisson (respectively compound Poisson) approximation for the longest perfect head run in a coin tossing experiment (respectively for the occurrence of a rare word in a DNA sequence) within the Chen-Stein method. Indeed, these were treated so far by means of identified "dependent neighbours"; here, we handle these problems through a Stein-Markov-Malliavin method based on the keystone operator $\L$ (we cannot really take advantage of the operator carré du champ within our framework) for marked binomial processes.
\vspace{3pt}\\

The second motivation to elaborate a stochastic variational calculus for marked binomial processes, comes from one famous scope of Malliavin calculus: finance.
One of the most successful area of financial mathematics deals with option pricing and more generally valuation of contingent contracts. A claim is a non-negative random variable $\Fm$ often assumed to be square-integrable that models the \textit{payoff} (the value of the option at expiry) at a fixed-term maturity $T$ of some asset. The simple example is given by a European call (resp. put) option based on the asset $\S$ with expiration date $T$ and strike price $K$, defined by $\Fm=(\S_T-K)_{+}$ (respectively $\Fm=(K-\S_T)_{+}$). The claim is said to be \textit{attainable} or \textit{duplicable} if there exists a self-financing portfolio of value $\Fm$ at expiry, and \textit{redundant} if this replication is only based on the existing assets. In \textit{complete} markets, all claims are reachable. The Cox-Ross-Rubinstein (CRR for short) model is the simplest discrete example of complete market (see Cox, Ross and Rubinstein \cite{cox1979option}, Rendleman and Bartter \cite{RendlemanBartter1979}). The so-called Fundamental Asset Pricing Theorem (FAPT) asserts that an arbitrage-free market is complete if and only if there exists a unique probability measure, equivalent to the initial under which the discounted price process is a martingale (in a discrete time setting see Jacod \cite{jacod1998local}, Schachermayer \cite{Schachermayer1992hilbert}). The \textit{prime} or \textit{fair price}, equal to the initial value of the replicating portfolio and more generally to its value at any time can be written as the (conditional) expectation with respect to this unique \textit{martingale measure} (for the seminal paper see Harrison and Kreps \cite{HarrisonKreps1979}). Besides, an explicit formula of the replicating strategy in terms of Malliavin derivative can be provided by the application of Clark-Ocone formula in the CRR model (see chapter one in Privault \cite{Privault_stochastica} and \cite{Privault2013market}). Several approaches were developed to address the problem of hedging in an incomplete market, as the trinomial model we are focusing on in this paper. To name but a few, one consists in "completing" the market by introducing new securities in an equilibrium approach (see Hakansson \cite{Hakansson1979}, Boyle and Tan \cite{BoyleTan}). Within the no-arbitrage framework, another approach consists in exhibiting the so-called \textit{minimal martingale measure} from the set of equivalent martingale measures. It is related to risk-minimizing strategies (Föllmer and Sondermann \cite{FollmerSondermann}, Schweizer \cite{Schweizer1991semimartingales}, \cite{Schweizer1995}) or portfolio expected utility maximization under constraints (Delbaen and Schachermayer \cite{Delbaen}, Frittelli \cite{Frittelli2000}, Runggaldier \cite{runggaldier2006}). 
In the application we choose to develop, we are less concerned by stating a valuation formula than to determine the optimized portfolio composition of minimal risk (in a sense to be defined) for an a priori non-attainable claim. When the claim is reachable, we get then a hedging formula, that has not been done - to our knowledge - in the trinomial frame. \\
Our initial aim was to transpose the criteria stated by Föllmer and Sondermann in \cite{FollmerSondermann} into the frame of the trinomial model (underpinned by sequences of $\{-1,0,1\}$-valued independent random variables) in order to determine the less risky approximating portfolio and derive the explicit expression of the corresponding optimizing strategy from the Clark formula stated for independent random variables (see Decreusefond and Halconruy \cite{DecreusefondHalconruy} Theorem 3.3). The lack of a martingale representation theorem in that latter framework, made it impossible to derive a hedging formula from Clark's (as done in binomial or Black-Scholes model); this led us to replace the trinomial model with what we called \textit{a ternary model}, also composed of two assets (a riskless and a risky asset) and  subtended by a marked binomial process where the mark space consists of two elements that indicate whether the price sequence (of the risky asset) rises or falls. Besides, if the probabilities of the occurrence of jumps and marks are properly defined, the ternary model is in fact equivalent in law to the trinomial model, so that all results "on expectation" hold in this latter. 
\vspace{3pt}\\

The interest of this work is therefore twofold; from a theoretical point of view, is built up there a variational calculus for binomial processes within a unifying Markov-Malliavin structure including pre-existing theories (in the Wiener or Poisson space). From an application angle, this new formalism is prone to offer an alternative point of view on Chen-Stein method and to produce an explicit portfolio optimisation formula in the trinomial model that had - to the best of our knowledge - not been done (in that way) so far for multiple periods.
\vspace{3pt}\\

The paper is structured as follows. In section \ref{Binomial_marked_analysis_sec} we give some elements of stochastic analysis for marked binomial processes and state a chaotic expansion result for any square-integrable binomial functional. From the successive development of a variational calculus in $\L^1$ and $\L^2$-contexts, we formalize a Markov-Malliavin structure for marked binomial processes in Section \ref{Gamma-Malliavin structure for marked binomial processes} whereas the Section \ref{Functional_identities_sec} is devoted to the statement of a Girsanov theorem and a Clark formula within this framework.The Section \ref{Applications_sec} is dedicated to the two main applications of our formalism: the (compound) Poisson approximation by Chen-Stein method and the portfolio optimisation in the trinomial model. Almost all proofs are postponed to Section \ref{Proofs_sec}.

\section{Stochastic analysis for marked binomial processes, part I}\label{Binomial_marked_analysis_sec}

\noindent
The first part of this present section is devoted to the introduction of notation and the main object of interest, the \textit{marked binomial process}. Throughout, $(\O,\mathcal A,\P)$ will be an abstract probability space assumed to be wide enough to support all random objects in question. 

\subsection{Framework}
Consider the measurable space $(\Xb,\mathcal X)$ where $\Xb=\N\times \E$, and $(\E,\mathcal B(\E))$, the \textit{mark space}, is a Borel space. Without any other indication, $\E$ will designate a countable (possibly finite) subset of $\Z$. Nevertheless, our construction may be extended to any subset of $\R$ and we provide later additional elements to formalize it. Denote by $\mathfrak N_{\Xb}$ (respectively $\widehat{\mathfrak N}_{\Xb}$) the space of simple, integer-valued, $\sigma$-finite (respectively finite) measures on $\Xb$. 
\noindent
Let $\EuScript N^{\Xb}$ be the smallest $\sigma$-field of subsets of $\mathfrak N_\Xb$ such that the mapping $\chi\in\mathfrak N_\Xb\mapsto \chi(\A)$ is measurable for all $\A\in\mathcal X$. \\
\noindent 
A \textit{point process} (respectively finite point process)  - or random counting measure -  is a random element $\eta$ in $\mathfrak N_\Xb$ (respectively in $\widehat{\mathfrak N}_{\Xb}$) that satisfies $\eta(\A)\in\Z_+\cup\{\infty\}$ for all $\A\in\mathcal X$. In its very definition $(\E,\mathcal B(\E))$ is a very simple Polish space endowed with its Borel $\sigma$-field so that we may and will assume that any element $\eta$ of $\mathfrak N_\Xb$  is \textit{proper}, i.e. can be $\P$-a.s. written as
\begin{equation}\label{Proper_process_def_eq}
\eta=\sum_{n=1}^{\eta(\Xb)}\delta_{\Xm_n},
\end{equation}
where $\{\Xm_n, \, n\pg 1\}$ denotes a countable collection of $\Xb$-valued random elements, and for $x\in\Xb$, $\delta_x$ is the Dirac measure at $x$. For a complete exposé on the subject of point processes, the reader can refer to the monograph of Last and Penrose (\cite{LastPenrose2017}, section 6.1) or Last \cite{Last2016} from that our presentation is largely inspired.
A \textit{binomial marked process} is a particular point process $\eta$ defined as follows; let $\lambda\in(0,1)$, and consider a Bernoulli process (see for instance Decreusefond and Moyal \cite{DecreusefondMoyal2012}, definition 6.6) of parameter $\lambda$, described by a sequence of jump times $(\Tm_t)_{t\in\Z_+}$, such that for any $t\in\Z_+$, the $t$-th arrival time $\Tm_t$ is defined by $\Tm_0=0$ and $\Tm_t=\sum_{s=1}^t\xi_s$, and where the inter-arrival variables $\{\xi_t,\,t\in\N\}$ are independent and identically distributed by a geometric law of parameter $\lambda$. In analogy with marked Poisson processes (see Last and Penrose \cite{LastPenrose2017}, chapter 7), we can set that $\eta$ is $\P$-a.s. represented as
\begin{equation}\label{Proper_process_bin_def_eq}
\eta=\sum_{t=1}^{\infty}\car_{\{\Tm_t<\infty\}}\delta_{(\Tm_t,\V_t)},
\end{equation}
where $\{\V_t,\, t\in\N\}$ is a collection of $\E$-valued random elements such that almost surely $\eta(\Tm_t,\V_t)=1$, for $\Tm_t<\infty$, and that are independent of the underlying jump process $\Nm=(\Nm_t)_{t\in\Z_+}$ defined by $\Nm_0=0$ and
$\Nm_t=\sum_{s \in\N} \car_{\{\Tm_s\pp t\}}$. By a slight abuse of notation, we shall write $(t,k)\in\eta$ in order to indicate that the point $(t,k)\in\Xb$ is charged by the random measure $\eta$. Note that for any $t\in\N$, $\Nm_t$ is a binomial random variable of mean $\lambda t$. \\
\noindent 
We may and will assume that $\mathcal A=\sigma(\eta)=:\f{}$ where  $\f{}=(\f{t})_{t\in\N}$ is the canonical filtration defined from $\eta$ by
\begin{equation*}
\f{0}:=\{\emptyset,\Omega\} \quad \text{and} \quad \f{t}:=\sigma\left\{\sum_{(s,k)}\eta(s,k),\, s\pp t,\, k\in\mathcal \E\right\}.
\end{equation*}
\noindent
Let $\Q$ be the common distribution of the $\V_t$ and $\P_\eta=\P\circ\eta^{-1}$ be the image measure of $\P$ under $\eta$ on the space $(\mathfrak N_\Xb,\EuScript N^{\Xb})$ i.e. the distribution of $\eta$; its compensator - the intensity of $\eta$ - is the measure $\nu$ defined on $\mathcal X$ by 
\begin{equation*}
\nu(\A)=\sum_{(t,k)\in\A}\sum_{s\in\N}\Big(\lambda\delta_{s}(\{t\})\otimes \sum_{\l\in\E}\Q(\{\l\})\delta_{\l}(\{k\})\Big)\; ; \; \A\in\mathcal X.
\end{equation*}
Throughout, we denote by $\R(\mathfrak N_\Xb)$ the class of real-valued measurable functions $\fk$ on $\mathfrak N_\Xb$ and by $\mathcal L^0(\O):=\mathcal L^0(\O,\mathcal A)$ the class of real-valued measurable functions $\Fm$ on $\O$.  Since $\mathcal A=\sigma(\eta)$, for any $\Fm\in\mathcal L^0(\O)$, there exists a function $\mathfrak f\in\R(\mathfrak N_\Xb)$ such that $\Fm=\mathfrak f(\eta)$. The function $\mathfrak f$ is called a \textit{representative} of $\Fm$ and is $\P\otimes\eta^{-1}$-a.s. uniquely defined. By default, the representative of a random variable $\Fm\in\mathcal L^0(\O)$ will be noted by the corresponding gothic lowercase letter, $\fk$. Last, for $p\in\N$, we define $\L^p(\P):=\L^p(\O,\mathcal A,\P)$ the set of $p$-integrable functions on $\O$ with respect to $\P$.
\begin{remark}
The marked binomial process $\eta$ can be equivalently written as 
$\sum_{t\in\N}\delta_{(\D\Nm_t,\Wm_t)},$
where for any $t\in\N$ the random variables $\Delta \Nm_t$ and $\Wm_t$ are defined by
\begin{equation*}
\Delta \Nm_t=\Nm_{t}-\Nm_{t-1}=\displaystyle\sum_{k\in \E}\car_{\{(t,k)\in\eta\}} \quad \text{and} \quad \Wm_t= \sum_{k\in\E}k\car_{\{(t,k)\in\eta\}}.
\end{equation*}
The variables $\Delta \Nm_t$ and $\Wm_t$ thus defined play a major part and we will often refer to them. Indeed, $\Delta \Nm_t$ indicates whether there is a jump at time $t$, and, if so, the variable  $\Wm_t$ gives its corresponding mark $k$. 
\end{remark}
\noindent
If $(\V_t)_{t\in\N}$ is a sequence of independent $\R$-valued random variables with common distribution $\Q$ and that are independent of the process $\Nm$, we can define the \textit{compound binomial process} $\Y=(\Y_t)_{t\in\N}$ of intensity $\nu$ by
\begin{equation}\label{Compound_pp_def_eq}
\Y_t:=\sum_{s=1}^{\Nm_t}\V_s.\\
\end{equation}
The corresponding  compensated process denoted $\overline \Y=(\overline\Y_t)_{t\in\N}$ defined by
\begin{equation}\label{Compensated_compound_pp_def_eq}
\overline\Y_t:=\Big(\sum_{s=1}^{\Nm_t}\V_s\Big)-\lambda \Q(\{k\}) t \; ; \; t\in\N,
\end{equation}
is a $\F$-martingale. In their very definitions, $\eta$, $\Y$ and $\Yt$ are the discrete analogues of the marked, compound Poisson and compensated compound Poisson processes.

\subsection{Integration with respect to a binomial marked process}
\noindent A process $u=(u_{(t,k)})_{(t,k)\in\Xb}$ is a measurable random variable defined on $(\mathfrak N(\Xb)\times\Xb,\F\otimes \mathcal X)$ that can be written $u=\sum_{(t,k)\in\Xb}\uk(\eta,(t,k))\car_{(t,k)}$, where $\{\uk(\eta,(t,k)),\,(t,k)\in\Xb\}$ is a family of measurable functions from $\mathfrak N_{\Xb}\times\Xb$ to $\R$ and $\uk$ is called the \textit{representative} of $u$. As for random variables, the representative of a process will be noted by a Gothic letter. For instance, considering a process $r$, its representative will be denoted by $\rk$.
The following assumption holds throughout this subsection. 
\begin{assumption}\label{assumptionR} There exists a discrete-time process $\Rm=(\Rm_{(t,k)},\,t\in\Z_+,k\in\E)$ where $\Rm_{(t,k)}=\rk(\eta,(t,k))$ is defined on $(\O,\F,\P)$ and that satisfies the following hypotheses:
\begin{enumerate}
\item The process $\Big(\sum_{k\in\E} \Rm_{(t,k)}\Big)_{t\in\Z_+}$ is a $\F$-martingale,
\item The family $\mathcal R=\{\D\Rm_{(t,k)}, \, t\in\N,\, k\in\E\}$ is orthogonal for the scalar product  $(\Xm,\Y)\in \L^2(\P)\mapsto \esp{\Xm \Y}$ and $\D\Rm_{(t,k)}$ and $\D\Rm_{(s,k)}$ are identically distributed for all $t,s\in\N$, $k\in\E$. We denote $\esp{(\D\Rm_{(t,k)})^2}=:\kappa_k$ for any $(t,k)\in\Xb$.
\end{enumerate}
\end{assumption}

\subsubsection{Stochastic integrals}

\noindent
Throughout the paper, we adopt the following set notations; we denote $\id{a}{b}:=\{a,\dots,b\}$ for any $a,b\in\Z$ such that $a<b$, and $\Xb_t:=\id{1}{t}\times\E$ for any $t\in\N$. By convention, $\id{1}{0}=\emptyset$.
Any $n$-tuple of $\Xb^n$ can be denoted by bold letters; for instance, $(\mathbf t_n,\mathbf k_n)=\big((t_1,k_1),\cdots,(t_n,k_n)\big)$. For any $\A\in\mathcal X$, we denote $\A^{n,<}=\{(\mathbf t_n,\mathbf k_n)\in\Xb^{n}\,: \,t_1<\cdots<t_n\},$ the corresponding time-ordered set, and
$\A^{n,\neq}=\{(\mathbf t_n,\mathbf k_n)\in\A^{n}\,: \,\forall i\neq j,\,  t_i\neq t_j\}$, the set with pairwise distinct (in time) elements.\\

\noindent
We denote by $\L^2(\P\otimes\nu)$ the Hilbert space of processes that are square-integrable with respect to the measure $\P\otimes\nu$, for which we define the corresponding inner product and norm by
\begin{equation*}
 \textless u,v\textgreater_{\L^2(\P\otimes\nu)}=\mathbf E\Big[\int_{\Xb}\uk(\eta,(t,k))\vk(\eta,(t,k))\,\d\nu(t,k)\Big] \;\text{and}\;\|u\|_{\L^2(\P\otimes\nu)}^2=\mathbf E\Big[\int_{\Xb}\uk(\eta,(t,k))^2\,\d\nu(t,k)\Big].
\end{equation*}

\begin{definition}
The set of \textit{simple processes}, denoted by $\mathcal U$ is the set of random variables of the form
\begin{equation}\label{simpleprocess}
u=\sum_{(t,k)\in\Xb_T}\uk(\eta,(t,k))\car_{(t,k)},
\end{equation}
where $T\in\N$,  and $\uk$ is the representative of $u$.
Let $\mathcal P$ denote the subspace of $\mathcal U$ made of simple predictable processes  i.e. of the form \eqref{simpleprocess} where $\uk(\eta,(t,\cdot))$ is $\f{t-1}$-measurable for any $t\in\id{1}{T}$.
\end{definition}

\begin{proposition}\label{Isometry_integral_prop}
Any process $u\in\mathcal U$ of representative $\uk$ is integrable with respect to the process $\Rm$ by the formula
\begin{equation*}
\Jt_1(u\,;\mathcal R)=\sum_{(t,k)\in\Xb}\uk(\eta,(t,k))\D\Rm_{(t,k)}.
\end{equation*}
The so-called $\mathcal R$-stochastic integral  $\Jt_1(u\,;\,\mathcal R)$ of $u$ extends to square-integrable predictable processes via the (conditional) isometry formula
\begin{equation}\label{isometry_proc_eq}
\esp{\big|\Jt_1\big(\car_{[t,\infty)}u\,;\mathcal R\big)\big|^2\,\big|\,\f{t-1}}=\esp{\big\|\car_{[t,\infty)}u\big\|_{\L^2(\Xb,\tilde\nu)}^2\,\big|\,\f{t-1}},
\end{equation}
where $\tilde \nu$ is the measure on $\Xb$ defined by $\tilde \nu(\{(t,k)\})=\kappa_k\nu(\{(t,k)\}))$, for any $(t,k)\in\Xb$.
\end{proposition}
\subsubsection{Multiple integrals}

\noindent
In order to define (multiple) stochastic integrals, we work in a space of  symmetrical functions. Our construction follows closely that depicted by Privault (see \cite{Privault_stochastica}, chapter 6); in a certain sense we transpose it into our context.  The space $\L^2(\Xb,\nu)^{\circ 0}$ is by convention identified to $\R$; let thus for any $f\in\L^2(\Xb,\nu)^{\circ 0}$, $\Jt_0(f_0)=f_0$. 
\begin{definition}
For $n\in\N$, let $\L^2(\Xb,\nu)^{\circ n}$ denote the subspace of $\L^2(\Xb,\nu)^{\otimes n}=\L^2(\Xb,\nu)^n$ composed of the functions $f_n\in\R(\Xb^{n})$ symmetric in their $n$ variables, i.e. such that for any permutation $\tau$ of $\{1,\dots,n\}$, $f_n\big((t_{\tau(1)},k_{\tau(1)}),\cdots,(t_{\tau(n)},k_{\tau(n)})\big)=f_n\big((t_1,k_1),\cdots,(t_n,k_n)\big)$,
for all $(t_1,k_1),\dots,(t_n,k_n)\in \Xb$. The space $\L^2(\Xb,\nu)^{\circ n}$ is endowed by the scalar product
\begin{align*}
\textless f_n,g_n\textgreater_{\L^2(\Xb,\nu)^{\circ n}}
&=n!\int_{\Xb^{n,<}}\,f_n(\mathbf t_n,\mathbf k_n)\,g_n(\mathbf t_n,\mathbf k_n)\,\d\nu^{\otimes n}(\mathbf t_n,\mathbf k_n),
\end{align*}
where  the tensor measure $\nu^{\otimes n}$ is defined on $ \Xb^{n,\neq}$ by $\nu^{\otimes n}=\bigotimes_{i=1}^n\,\nu$.
\end{definition}

\noindent
The multiple stochastic integral can be defined on $\mathcal C_c(\Xb^n,\R)$, the set of continuous functions with compact support on $\Xb^n$ and extended by isometry to $\L^2(\Xb,\nu)^{\circ n}$.
\begin{proposition}\label{Isometry_mul_prop}
The $\mathcal R$-stochastic integral of order $n$ is the application defined on $\mathcal C_c(\Xb^n,\R)$ by
\begin{equation}\label{Jn_mult_int_rec_eq}
\Jt_n(f_n\,;\mathcal R)=n\displaystyle\sum_{(t,k) \in \Xb}\Jt_{n-1}(f_n(\star,(t,k)))\,\D\Rm_{(t,k)},
\end{equation}
where $"\star"$ denotes the first $n-1$ variables of $f_n((t_1,k_1),\dots,(t_n,k_n))$.
It can equivalently be written as
\begin{equation}\label{Jn_mult_int_eq}
\Jt_n(f_n\,;\mathcal R)
=n!\sum_{(\mathbf t_n,\mathbf k_n)\in \Xb^{n}}\,f_n(\mathbf t_n,\mathbf k_n)\,\prod_{i=1}^{n}\D\Rm_{(t_i,k_i)}.
\end{equation}
Besides, it satisfies the isometry formula: for any $f_n\in \L^2(\Xb,\nu)^{\circ n}$, $g_m\in \L^2(\Xb,\nu)^{\circ  m}$
\begin{equation}\label{isometry_eq}
\esp{\Jt_n(f_n\,;\,\mathcal R)\Jt_m(g_m\,;\,\mathcal R)}=\car_{\{n\}}(m)n!\textless f_n,g_n \textgreater_{\L^2(\Xb,\tilde\nu)^{\circ n}},
\end{equation}
so that its domain can be extended to $\L^2(\Xb,\mathcal X,\tilde\nu)^{\circ n}\simeq \L^2(\Xb,\mathcal X,\nu)^{\circ n}$.
\end{proposition}

\noindent 
Up to now, if no need to specify, $\L^2(\Xb)^{\circ n}$ could indifferently designate $\L^2(\Xb,\mathcal X,\nu)^{\circ n}$ or $\L^2(\Xb,\mathcal X,\tilde\nu)^{\circ n}$. This subsection ends up with two Lemmas that will be useful to state the chaotic expansion theorem.

\begin{lemma}\label{Int_tensor_lem}
For any $g\in \L^2(\Xb)$ and $f_n\in \L^2(\Xb)^{\circ n}$,
\begin{multline*}
\Jt_{n+1}(g\circ f_n\,;\mathcal R)
=n\sum_{(t,k)\in\Xb}\Jt_n\big(f_n(\star,(t,k))\circ g(\cdot)\car_{\id{1}{t-1}^n}(\star,\cdot)\,;\mathcal R\big)\D\Rm_{(t,k)}\\
+ \sum_{(t,k)\in\Xb}g(t,k)\Jt_n(f_n\car_{\id{1}{t-1}^n}\,;\mathcal R)\D\Rm_{(t,k)},
\end{multline*}
where $\circ$ designates the symmetric tensor product and satisfies for $(\mathbf t_n,\mathbf k_n)\in\Xb^n$,
\begin{equation*}
g\circ f_n(\mathbf t_{n+1},\mathbf k_{n+1})=\frac{1}{n+1}\sum_{i=1}^{n+1}g(t_i,k_{i})f_n^{\neg i}(\mathbf t_n,\mathbf k_n),
\end{equation*}
with for $i\in\id{1}{n},\,$ $f_n^{\neg i}(\mathbf t_n,\mathbf k_n)=f_n\big((t_1,k_1),\cdots,(t_{i-1},k_{i-1}),(t_{i+1},k_{i+1}),(t_n,k_n)\big).$
\end{lemma}

\begin{lemma}\label{Cond_int_lem}
For any $(t,n)\in\N^2$, $f_n\in\L^2(\Xb)^{\circ n}$,
\begin{equation*}\label{condinteq}
\esp{\Jt_n(f_n\,;\mathcal R)\,|\,\f{t}}=\Jt_n(f_n\car_{\id{1}{t}}\,;\mathcal R).
\end{equation*}
\end{lemma}

\subsection{Chaotic decomposition}

\noindent
This subsection is devoted to the statement of a chaotic decomposition for any square-integrable \textit{marked binomial functional}, that are random variables of the form
\begin{equation}\label{functional}
\Fm=f_0\car_{\{\eta(\mathbb X)=0\}}+\sum_{n \in\N}\sum_{(\mathbf t_n,\mathbf k_n)\in\Xb^n}\,\car_{\{\eta(\mathbb X)=n\}}f_n(\mathbf t_n,\mathbf k_n)\prod_{i=1}^{n}\car_{\{(t_i,k_i)\in\eta\}},
\end{equation}
where any function $f_n$ is an element of $\L^1(\nu)^{\circ n}$, that is the subspace of $\L^1(\nu)^{\otimes n}:=\L^1(\Xb,\mathcal X,\nu)^{\otimes n}\\=\L^1(\Xb,\mathcal X,\nu)^{n}$ composed of the functions symmetric in their $n$ variables. We introduce the space of cylindrical functions, which is dense in $\L^2(\P)$.
\begin{definition}
A functional $\Fm$ is \textit{cylindrical} if there exists $T\in\N$ such that
\begin{equation}\label{Cylindrical_def_eq}
\Fm=f_0\car_{\{\eta(\mathbb X)=0\}}+\sum_{n \in\N}\sum_{(\mathbf t_n,\mathbf k_n)\in\Xb_T^n}\,\car_{\{\eta(\mathbb X)=n\}}f_n(\mathbf t_n,\mathbf k_n)\prod_{i=1}^{n}\car_{\{(t_i,k_i)\in\eta\}},
\end{equation}
where $\Xb_T=\{(t,k)\in\Xb\,:\,t\pp T\}$.
\end{definition}
\noindent

\noindent
Within Assumption \ref{assumptionR}, let $\mathcal H_0:=\R$ and for any $n\in\N$, $\mathcal H_n$ be the subspace of $\L^2(\P)$ made of integrals of order $n\pg 1$:
\begin{equation*}
\mathcal H_n:=\left\{\Jt_n(f_n) \; ; \; f_n\in\L^2(\Xb)^{\circ n}  \right\},
\end{equation*}
where $\Jt_n(f_n):=\Jt_n(f_n\,;\mathcal R)$, and  called \textit{chaos of order $n$}. In what follows $\mathcal L^0(\P,\f{t})$ denotes the set of $\f{t}$-measurable random variables.
\begin{lemma}\label{chaos_L0_lem}
For any $t\in\N$,
\begin{equation}\label{chaos_L0_lem_eq}
\mathcal L^0(\P,\f{t})=(\mathcal H_0\oplus\cdots\oplus\mathcal H_t)\bigcap \mathcal L^0(\P,\f{t}).
\end{equation}
\end{lemma}

\noindent
As a direct consequence, any random variable $\Fm\in \mathcal L^0(\P,\f{t})$ can be expressed as
\begin{equation*}
\Fm=\esp{\Fm}+\sum_{n=1}^{t} \Jt_n\big(f_n\car_{\id{1}{t}^n}\big).
\end{equation*}
This also means that the space of cylindrical functions coincides with the linear space spanned by multiple stochastic integrals i.e.
\begin{align*}
\cyl
&=\mathrm{Span}\bigg\{\bigcup_{n\in\Z_+}\mathcal H_n \bigg\}.
\end{align*}

\noindent
The completion of $\cyl$ in $\L^2(\P)$ is denoted by the sum $\displaystyle\bigoplus_{n\in\Z_+}\,\mathcal H_n.$
We can state the main theorem of this section and provide a chaotic decomposition for any square-integrable random variable. 
\begin{theorem}\label{Chaos_R_th}
The space of square-integrable marked binomial functionals is provided with a chaotic decomposition
\begin{equation}\label{Chaos_L2_prop}
\L^2(\P)=\bigoplus_{n\in\Z_+}\,\mathcal H_n.
\end{equation}
In other terms, any random variable $\Fm\in \L^2(\P)$ can be expanded in a unique way as
\begin{equation}\label{Chaos_R_eq}
\Fm=\esp{\Fm}+\sum_{n\in\N}\Jt_n(f_n).
\end{equation}
\end{theorem}

\begin{proof}
The proof follows closely that of Proposition 1.5.3 in Privault \cite{Privault_stochastica} by combining Lemma \ref{chaos_L0_lem} and the density of $\mathcal S$ in $\L^2(\P)$.
\end{proof}

\begin{corollary}
For any $\Fm,\Gm\in\L^2(\P)$,
\begin{equation*}
\mathrm{cov}(\Fm,\Gm)=\sum_{n\in\N}n!\textless f_n,g_n\textgreater_{\L^2(\Xb,\tilde\nu)^{\otimes n}}.
\end{equation*}
\end{corollary}
\begin{proof}
Immediate using \eqref{Chaos_R_eq}  together with Proposition \ref{Isometry_integral_prop}.
\end{proof}
\begin{remark}
The chaotic decomposition for marked binomial processes is not generated - as in the framework of normal martingales (including Brownian motion, Poisson and Rademacher processes) - from the increments of the compensated underlying process $\overline\Y$ \eqref{Compensated_compound_pp_def_eq} itself but in terms of stochastic integrals with respect to an auxiliary process $\Rm$ satisfying both a martingale and an orthogonality properties (Assumption \ref{assumptionR}). This can be explained by the absence of normal martingales associated to the compound binomial process $\overline\Y$. Indeed, by transposing the remark p. 95 in Privault \cite{Privault_stochastica} into our framework, the quadratic variation of the compensated compound $\F$-martingale  $\Yt$ satisfies
\begin{equation*}
[\Yt,\Yt]_t=\frac{1}{{\sqrt{\lambda\var [\V_1]}}}\sum_{s=1}^{\Nm_t}|\V_s|^2=\frac{1}{{\sqrt{\lambda\var [\V_1]}}}\sum_{s=1}^t\,|\V_{\Nm_s}|\,\D \Yt_s+\frac{\esp{\V_1}}{{\sqrt{\var [\V_1]}}}\sum_{s=1}^t\,|\V_{\Nm_s}|,
\end{equation*}
does not allow to find a square-integrable $\F$-adapted process $(\phi_t)_{t\in\R_+}$ solution of  the \textit{structure equation}
\begin{equation*}
[\Yt,\Yt]_t=t+\sum_{s=1}^t \phi_s\,\D\Yt_s,
\end{equation*}
when $\V$ is not deterministic. This structural reason explains the lack of usual chaotic decomposition with respect to the increments of the compensated initial process. 
\end{remark}
\noindent
Despite previous remark, we can nevertheless provide a \textit{pseudo}-chaotic (not orthogonal) decomposition related to the process $\Y$.
In order to do that, we introduce the process $\Zm=(\Zm_{(t,k)} \,;\, (t,k)\in\Xb)$ which increments are defined by the family $\mathcal Z=\{\Delta \Zm_{(t,k)}\,;\, (t,k)\in\Xb\}$ with
\begin{equation}\label{Defintion_DZ_eq}
\Delta \Zm_{(t,k)}=\car_{\{(\Delta \Nm_t,\V_{\Nm_t})=(1,k)\}}-\lambda \Q(\{k\})=\car_{\{(\Delta \Nm_t,\Wm_{t})=(1,k)\}}-\lambda\Q(\{k\})\; ; \; (t,k)\in\Xb. 
\end{equation}
\noindent
The definition of $\mathcal Z$ is quite natural since the process $\Yt=(\Yt_{t})_{t\in\N}$ can be equivalently written
\begin{equation}\label{compoundgeom}
\Yt_t=\sum_{s\pp t}\sum_{k\in\E}k\Delta \Zm_{(s,k)}.
\end{equation}
For any \label{key}$T\in\N$, define $\mathcal Z_T=\{\Delta \Zm_{(t,k)}\,;\, (t,k)\in\Xb_T\}$. This family is not orthogonal; however, the finite dimension of the related spanned space, being equal to
\begin{equation*}
1+\sum_{s=1}^T\,|\E|^s\times\binom{T}{s}=(|\E|+1)^T=:\overline{\mathfrak m},
\end{equation*}
enables to derive from $\mathcal Z_T$ an orthogonal family, $\mathcal R_T=\{\D\Rm_{(t,k)}\,;\, (t,k)\in\Xb_T\}$. Assume that $\E=\{k^1,\cdots,k^{\overline{\mathfrak m}}$); then, the Gram-Schmidt process provides
\begin{equation}\label{family_R_eq}
\D\Rm_{0}=1,\quad \D\Rm_{(t,k^1)}=\D\Zm_{(t,k^1)} \quad \text{and}\qquad   \D\Rm_{(t,k^n)}=\D \Zm_{(t,k^n)}-\sum_{j=1}^{n-1}\frac{\esp{\D\Zm_{(1,k^n)}\D\Rm_{(1,k^j)}}}{\esp{(\D\Rm_{(1,k^j)})^2}}\D\Rm_{(t,k^j)},
\end{equation}
for $n\in\id{1}{\overline{\mathfrak m}}$, by noting that the random variables $\D\Rm_{(t,k)}$ (respectively $\D\Zm_{(t,k)}$) and $\D\Rm_{(1,k)}$ (respectively $\D\Zm_{(1,k)}$) are identically distributed and that for any $s\in\id{1}{t-1}$,
$\esp{\D\Rm_{(s,k)}\,\Delta \Zm_{(t,1)}}=\esp{\D\Rm_{(s,k)}\esp{\D \Zm_{(t,1)}|\f{s}}}=0$.
In fact, for any $t\in\id{1}{T}$, $(\D\Zm_{(t,k)}, \, k\in\E)$ is the image of $(\D\Rm_{(t,k)}, \, \, k\in\E)$ by the linear transformation of associated to the $\overline{\mathfrak m}\times \overline{\mathfrak m}$ triangular matrix 
\begin{equation}\label{matrix_basis_eq}
\mathfrak M=(\mathfrak m_{ij})_{i,j\in\id{1}{\overline{\mathfrak m}}}=
\begin{pmatrix}
1 & 0 & \cdots & 0 \\
\gamma_{21} & 1 & \cdots & 0 \\
\vdots  & \vdots  & \ddots & \vdots  \\
\gamma_{n1}  &\gamma_{n2}  & \cdots & 1
\end{pmatrix},
\end{equation}
where  $\gamma_{ij}:=\esp{\D\Zm_{(1,k^i)}\D\Rm_{(1,k^j)}}/\esp{(\D\Rm_{(1,k^j)})^2}$. As the matrix $\mathfrak M$ is invertible, for any $t\in\id{1}{T}$, $(\D\Rm_{(t,k)}, \, k\in\E)$ is obtained through the product of $\mathfrak M^{-1}$ by the vector $(\D\Zm_{(t,k)}, \, \, k\in\E)$. Moreover, since the linear transformation it stands for is then bijective, the family $\mathcal R$ can be constructed in a similar fashion when $\E$ is countable not finite. Thus, the process $\Rm$ which increments are defined by the family $\mathcal R$ satisfies Assumption \ref{assumptionR}.
\begin{remark}
It seems to be possible to construct such a family $\mathcal R$ even if  $\E$ is not countable (take for instance $\E=\R$), by drawing inspiration from the design of the \textit{orthogonal power jump process} for Lévy processes, in Di Nunno, Oksendal and Proske \cite{DiNunnoOksendalProske2004}. Transposing it into our framework, that would give: define for any $n\in\N$, 
\begin{equation*}
\D\Zm_{t}^{(n)}=\Xm_t^{(n)}-\mathbf E\big[\Xm_t^{(n)}\big]:=\sum_{s\in\id{1}{t}}(\D\Y_s)^{n}-\mathbf E\bigg[\sum_{s\in\id{1}{t}}(\D\Y_s)^{n}\bigg],
\end{equation*}
and the family $\mathcal R$ by $\D\Rm_{0}=1$, and
\begin{equation*}
\D\Rm_{t}^{(n)}=\Xm_t^{(n)}+\sum_{j=1}^{n-1}\gamma_{nj}\Xm_t^{(j)},
\end{equation*}
where the $\gamma_{nj}$ are real numbers such that the processes of  the collection $\{(\D\Rm_{t}^{(n)})_{t\in\N},\; n\in\N\}$ are \textit{strongly orthogonal martingales}, i.e. for any $t\in\N$, the product $\D\Rm^{(n)}\D\Rm^{(m)}$ is a uniformly integrable martingale for all $(n,m)\in\N^2$, $m\neq n$. 
\end{remark}
\begin{remark}
Let the $\mathcal Z$-stochastic integral of order $n\in\N$ be the application on $\L^2(\Xb)^{\circ n}$ such that for any $f_n\in\L^2(\Xb)^{\circ n}$,
\begin{equation*}
\Jt_n(f_n\,;\mathcal Z)
:=\sum_{(\mathbf t_n,\mathbf k_n)\in \Xb^{n}}\,f_n(\mathbf t_n,\mathbf k_n)\,\prod_{i=1}^{n}\D\Zm_{(t_i,k_i)}.
\end{equation*}
Considering the application $\car_{(\mathbf t_n,\mathbf k_n)}^{<}\,:\,(\mathbf s_n,\mathbf l_n)\in\Xb^{n,<}\mapsto\car_{(\mathbf t_n,\mathbf k_n)}(\mathbf s_n,\mathbf l_n)$, we retrieve the remarkable and usual identity
\begin{equation*}
\Jt_n\big(\car_{(\mathbf t_n,\mathbf k_n)}^{<}\,;\mathcal Z\big)=\prod_{i=1}^{n}\Delta \Zm_{(t_i,k_i)} \; ; \; n\in\N.
\end{equation*}
This is of key importance; it basically means that we can reconstruct the signal $\Y$ by means of the stochastic integral of elementary functions defined on $\Xb^n$. In particular for $n=1$, this gives $\Jt_1(\car_{(t,k)}\,;\mathcal Z)=\D\Zm_{(t,k)}$, that appears as a reminiscence of the Wiener integral.
\end{remark}
\noindent
We can thus state the following result.
\begin{proposition}\label{Chaos_Z_prop}
Any random variable $\Fm\in \L^2(\P)$ can be expressed as
\begin{equation}\label{Chaos_Z_eq}
\Fm=\esp{\Fm}+\sum_{n\in\N}\Jt_n(g_n,;\mathcal Z).
\end{equation}
In particular if $|\E|=\overline{\mathfrak m}$ such that $\E=\{k^1,\dots,k^{\overline{\mathfrak m}}\}$, any function $g_n$ is explicitely given by
\begin{equation*}
g_n(\mathbf{t}_n,\mathbf k_n)=\sum_{i_1=p_1}^{\overline{\mathfrak m}}\cdots\sum_{i_n=p_n}^{\overline{\mathfrak m}}\Big(\prod_{j=1}^n\mathfrak m_{k^{i_j}k^{p_j}}^{-1}\Big) f_n((t_1,k_1^{i_1}),\dots,(t_n,k_n^{i_n})),
\end{equation*}
where for any $j\in\{1,\dots,n\}$, $p_j$ denotes the element of $\{1,\dots,m\}$ such that $k_j=k^{p_j}\in\E$, and for notation purposes, $\mathfrak m_{k^i,k^j}^{-1}$ designate the $(i,j)$-th entry of matrix $\mathfrak M^{-1}$, the inverse of matrix $\mathfrak M$ defined by \eqref{matrix_basis_eq}.
\end{proposition}

\subsection{Doléans exponentials}
Define for any $h\in\L^2(\Xb)$ the exponential vector by
\begin{equation}\label{Doleans_def_eq}
\xi(h)=\esp{\xi(h)}+\sum_{n\in\N}\frac{1}{n!}\Jt_n(h^{\otimes n}).
\end{equation}
\noindent
The family $(\xi_t(h))_{t\in\N}$ defined by $\xi_t(h)=\xi(h\car_{\id{1}{t}})$ can be viewed as a discrete Doléans exponential solution of the equation in differences
\begin{equation*}
\xi_t(h)-\xi_{t-1}(h)=\xi_{t-1}(h)\sum_{k\in\E}g(t,k)\D\Zm_{(t,k)}, \; t\in\N,
\end{equation*}
where $g\in\L^2(\Xb)$ is given in the following theorem.
\begin{proposition}\label{Doleans_prop}
For any $h\in\L^2(\Xb)$, the discrete Doléans exponential defined by \eqref{Doleans_def_eq} can be written as
\begin{equation}\label{Doleans_eq}
\xi(h)
=\esp{\xi(h)}\prod_{t\in\N}\Big(1+\sum_{k\in\E}g(t,k)\big(\car_{(t,k)}-\lambda\Q(\{k\})\big)\Big),
\end{equation}
where $g$ is the element of $\L^2(\Xb)$ such that
$\Jt_1(g\,;\mathcal Z)=\Jt_1(h).$
\end{proposition}

\section{Stochastic analysis for marked binomial processes, part II}\label{Gamma-Malliavin structure for marked binomial processes}

The section  is organised as follows; the first subsection is dedicated to the development of a $\L^1$-theory for binomial marked processes which starting point is a Mecke-type formula. In the just following part, are provided some elements of Malliavin calculus whereas in the third subsection, the tools of $\L^1$ and $\L^2$ theories are gathered to formalize a unified Markov-Malliavin structure.

\subsection{$\L^1$-theory:  the Mecke and Mehler's formulas}
\subsubsection{The Mecke formula and difference operators on $\L^1$}

\noindent
The following Lemma is the analogue of the Mecke formula for marked binomial processes. 
\begin{lemma}\label{Mecke_formula_th}
Let $\eta$ be a marked binomial process on $\Xb$ with intensity measure $\nu$. Then for any real-valued, non-negative, $\Xb\times\mathfrak N_\Xb$-measurable function $\uk$,
\begin{equation}\label{Mecke_formula_eq}
\mathbf E\Bigg[\sum_{(t,k)\in\eta}\uk(\eta,(t,k))\Bigg]=\esp{\int_{\Xb}\uk\big(\pi_{t}(\eta)+\delta_{(t,k)},(t,k)\big)\d\nu(t,k)}.
\end{equation}
where the application $\pi_t\,:\,\mathfrak N_\Xb\rightarrow\mathfrak N_\Xb$ is the restriction of $\eta$ to $\g{t}:=\sigma\big\{\sum_{(s,k)}\eta(s,k), s\neq t, k\in\E\big\}$ i.e.
\begin{equation}
\pi_t(\eta)=\displaystyle\sum_{s\neq t}\sum_{k\in\E}\,\eta(s,k).
\end{equation}
\end{lemma}

\noindent
\begin{remark}
Clearly, the formula \eqref{Mecke_formula_eq} continues to hold provided the process $u$ of representative $\uk$ belongs to $\L^1(\P\otimes\nu)$. Furthermore, we can state
\begin{equation}\label{Mecke_formula_eq2}
\mathbf E\bigg[\sum_{(t,k)\in\eta}\uk\big(\eta-\delta_{(t,k)},(t,k)\big)\bigg]=\esp{\int_{\Xb_T}\uk\big(\pi_t(\eta),(t,k)\big)\d\nu(t,k)}.
\end{equation} 
\end{remark}
\noindent
The applications  defined on $\mathfrak N_\Xb\times \Xb$, and expressed for any $\big(\eta,(t,k)\big)\in\mathfrak N_\Xb\times \Xb$ by
\begin{equation}\label{jumpapplicationseq}
\eta\mapsto\pi_{t}(\eta)+ \delta_{(t,k)} \quad\text{and}\quad \eta\mapsto\pi_{t}(\eta),
\end{equation}
can be interpreted as the applications  acting on $\eta$ respectively by forcing a jump of height $k$ at time $t$ or forbidding any jump at time $t$. 
\noindent
As a  reminiscence of Poisson space theory, define $\Dm^+$ the \textit{add-one cost operator} for any $\Fm\in\mathcal L^0(\O)$ by
\begin{equation}\label{Def_D+_eq}
\Dm_{(t,k)}^+\Fm:=\mathfrak f(\pi_{t}(\eta)+\delta_{(t,k)} )-\mathfrak f(\pi_{t}(\eta)).
\end{equation}
The difference operator $\Dm^+$ measures the effect of adding a point $(t,k)\in\Xb$ to $\eta$ compared to the process truncated at time $t$. 
The product formula can be easily deduced from this expression of $\Dm$ and is strongly reminiscent to that existing in the Poisson setting (see for instance Privault \cite{Privault_stochastica}, Proposition 6.4.8).
For $\Fm,\Gm\in\L^1(\P)$ of respective representatives $\fk$ and $\gk$,
\begin{equation}\label{Product_D_formula_eq}
\Dm_{(t,k)}^+(\Fm\Gm)=\fk(\pi_{t}(\eta))(\Dm_{(t,k)}^+\Gm)+\gk(\pi_{t}(\eta))(\Dm_{(t,k)}^+\Fm)+(\Dm_{(t,k)}^+\Fm)(\Dm_{(t,k)}^+\Gm).
\end{equation}
\begin{remark}
By definition, given $k\in\E$, for any $t\in\N$ the random variables $\pi_{t}(\eta)+\delta_{(t,k)}$ and $\pi_{t}(\eta)$ are $\g{t}$-measurable; so does $\Dm_{(t,k)}^+\Fm$. 
\end{remark}

\noindent
In a similar way the operator $\Dm^{-}$ is defined for any $\Fm\in\mathcal L^{0}(\O)$ by
\begin{equation}\label{D^-def}
\Dm_{(t,k)}^-\Fm:=\fk(\eta)-\fk(\eta-\delta_{(t,k)}),
\end{equation}  
if $(t,k)\in\eta$, and is equal to zero otherwise.
The operator $\Dm^-$ may be interpreted as a \textit{remove-one cost operator}: if the point $(t,k)$ was charged by $\eta$, this is removed by the action of $\Dm_{(t,k)}^-$. 
The operator $\Dm^-$ satisfies the product formula: for $\Fm,\Gm\in\mathcal L^0(\O)$,
\begin{equation}\label{productformulaD^-eq}
\Dm_{(t,k)}^-(\Fm\Gm)=\Fm(\Dm_{(t,k)}^-\Gm)+\Gm(\Dm_{(t,k)}^-\Fm)-(\Dm_{(t,k)}^-\Fm)(\Dm_{(t,k)}^-\Gm).
\end{equation} 
\noindent 
Define on $\L^1(\P\otimes\nu)$ the operator $\tilde\delta$ such that for any process $u\in\L^1(\P\otimes\nu)$ of representative $\uk$,
\begin{equation}\label{Deltatilde_def_eq}
\tilde\delta(u):=\sum_{(t,k)\in\eta}\uk(\eta,(t,k))-\int_\Xb\uk(\eta,(t,k))\d\nu(t,k).
\end{equation}
As $\pi_t(\eta)+\delta_{(t,k)}=\eta$  if $(t,k)\in\eta$, we can additionally introduce the operator $\widetilde\L$ on $\mathcal L^0(\O)$ such that
\begin{multline}\label{DeltaDF_eq}
\widetilde \L\Fm:=-\tilde\delta(\Dm^+\Fm)=-\sum_{(t,k)\in\eta}\big[\fk(\pi_t(\eta)+\delta_{(t,k)})-\fk( \eta)\big]+\int_\Xb\big[\fk(\pi_t(\eta)+\delta_{(t,k)})-\fk(\pi_t(\eta))\big]\d\nu(t,k)\\
=-\sum_{(t,k)\in\eta}\big[\fk(\eta)-\fk(\eta-\delta_{(t,k)})\big]+\int_\Xb\big[\Dm_{(t,k)}^+\Fm\big]\,\d\nu(t,k)\\=\sum_{(t,k)\in\eta}\big[\Dm_{(t,k)}^-\Fm\big]+\int_\Xb\big[\Dm_{(t,k)}^+\Fm\big]\,\d\nu(t,k),
\end{multline}
for any $\Fm\in\mathcal L^0(\O)$. The Mecke equation \eqref{Mecke_formula_eq} ensures that this definition does not depend $\P$-a.s. on the choice of the representative. We get the following "almost"-$\L^1$-integration by parts formula.
\begin{proposition}\label{IPP_L1_prop}
For any predictable process $u\in\mathcal L^0(\O\times\N)$ and $\Fm\in\mathcal L^0(\O)$,
\begin{equation*}
\esp{\int_\Xb\Dm^+\Fm\,\uk(\eta,(t,k))\,\d\nu(t,k)}=\mathbf E[\Fm\widetilde\delta(u)]+\esp{\int_\Xb\overline\Dm_{t}\Fm\,\uk(\eta,(t,k))\,\d\nu(t,k)}.
\end{equation*}
where the operator $\overline{\Dm}$ is defined on $\mathcal L^0(\O)$ by
\begin{equation*}
\overline{\Dm}_t(\Fm)=\fk(\eta)-\fk(\pi_t(\eta)) \; ; \; t\in\N.
\end{equation*}
\end{proposition}
\begin{remark}
This latter can be rewritten as 
\begin{equation}\label{IPP_tildeD_eq}
\mathbf E\big[\textless\widetilde \Dm\Fm,u\textgreater_{\L^2(\Xb,\nu)}\big]=\mathbf E\big[\Fm\widetilde\delta(u)\big],
\end{equation}
where $\widetilde \Dm_{(t,k)}\Fm=\Dm_{(t,k)}^+\Fm-\overline{\Dm}_t\Fm=\fk(\pi_t(\eta)+\delta_{(t,k)})-\fk(\eta)$.
The operator $\widetilde\Dm$ is the exact discrete analogue of the usual gradient on Poisson space. In that latter case, provided the intensity measure of the Poisson point process is diffuse, $\Dm^+$ and $\widetilde\Dm$ are equal $\P\otimes\nu$ almost surely. That does not hold here, and we justify our choice to define the add-one-cost operator by $\Dm^+$ and not via $\widetilde \Dm$ in the perspective to combine $\L^1$ and $\L^2$ later on through \eqref{Grad_difference_eq}. This remark is crucial to understand why Gamma calculus can not perform within our framework. Indeed, let us introduce the operator $\widetilde\Gamma$  defined for any random variables $\Fm,\Gm\in\mathcal L^0(\O)$ such that
$(\Dm^+\Fm)(\Dm^+\Gm)\in\L^1(\P\otimes\nu)$, 
by $\widetilde\Gamma(\Fm,\Gm)=1/2\big[\widetilde \L(\Fm\Gm)-\Fm(\widetilde \L\Gm)-\Gm(\widetilde \L\Fm)\big]$. By combining the product rules \eqref{Product_D_formula_eq} and \eqref{productformulaD^-eq}, we obtain 
\begin{align}
\widetilde\Gamma(\Fm,\Gm)
&=\frac{1}{2}\bigg[\int_{\Xb}(\Dm_{(t,k)}^+\Fm)(\Dm_{(t,k)}^+\Gm)\,\d\nu(t,k)+\int_{\Xb}(\Dm_{(t,k)}^-\Fm)(\Dm_{(t,k)}^-\Gm)\,\d\eta(t,k)\nonumber\\
&\qquad-\int_{\Xb}(\Dm_{(t,k)}^+\Fm)(\overline{\Dm}_t\Gm)\,\d\nu(t,k)-\int_{\Xb}(\Dm_{(t,k)}^+\Gm)(\overline{\Dm}_t\Fm)\,\d\nu(t,k)\bigg],\label{Gamma_0_def}
\end{align}	
whereas, as a consequence of the Mecke formula we can prove that for any $\Fm,\Gm\in\mathcal L^0(\O)$ of respective representatives $\fk$ and $\gk$ such that
$\Fm\gk(\pi_t(\eta)+\delta_{\cdot}),\,\fk(\pi_t(\eta)+\delta_{\cdot})\Gm,\,\fk(\pi_t(\eta)+\delta_{\cdot})\gk(\pi_t(\eta)+\delta_{\cdot})\in\L^1(\P\otimes\nu)$, we have
\begin{multline*}\label{L1_IPP_eq}
-\mathbf E\big[\widetilde\Gamma(\Fm,\Gm)\big]
=\frac{1}{2}\Big[\mathbf E[\Fm(\widetilde \L\Gm)]+\mathbf E[\Gm(\widetilde \L\Fm)]\Big]
\\=\frac{1}{2}\esp{\left(\int_{\Xb}(\Dm_{(t,k)}^+\Gm)(\overline\Dm_{t}\Fm)\d\nu(t,k)+\int_{\Xb}(\Dm_{(t,k)}^+\Fm)(\overline\Dm_{t}\Gm)\d\nu(t,k)\right)},
\end{multline*}
from which it is not possible to draw an $\L^1$-integration by parts formula since $\mathbf E[\Fm(\widetilde \L\Gm)]\neq\mathbf E[\Gm(\widetilde \L\Fm)]$. As a result, the possibility to combine $\L^1$ and $\L^2$ theories in regards of the carré du champ operator $\Gamma$ seems compromised; their connection can at best come at the level of the operator $\L$.
\end{remark}

\subsection{$\L^2$-theory: Malliavin operators}

\noindent
From the chaotic decomposition that equips the space $\L^2(\P)$, we define the Malliavin operators, \textit{gradient}, \textit{divergence}, \textit{number operator}, and the \textit{Ornstein-Uhlenbeck semi-goup}. 

\subsubsection{Gradient}

\noindent
As one way to develop it, we  introduce the Malliavin derivative as the \textit{annihilation operator} acting on the space $\L^2(\P)$ seen in terms of its chaotic expansion \eqref{Chaos_L2_prop}.
\begin{definition}
Let $\DD_0$ be the set of random variables $\Fm\in\L^2(\P)$ whose decomposition \eqref{Chaos_R_eq} satisfies 
\begin{equation}\label{Dom_D_condition_chaos_eq}
\sum_{n=1}^{\infty}nn!\|f_n\|_{\L^2(\Xb)^{\otimes n}}^2<\infty.
\end{equation}
Let the linear, unbounded, closable operator $\Dm\,:\, \DD_0\rightarrow \L^2(\P\otimes\nu)$ be defined for any element $\Jt_n(f_n)$ of $\mathcal H_n$ by
\begin{equation}\label{Gradient_def_eq}
\Dm_{(t,k)}\Jt_n(f_n)=n\,\Jt_{n-1}\big(f_n(\star,(t,k))\car_{\id{1}{n-1}^{n,<}}\big).
\end{equation}
\end{definition}

%
%

\subsubsection{Divergence}

\noindent
Let $\mathcal U$ be the space
\begin{equation}\label{cylindrical_2_eq}
\mathcal U=\bigg\{\sum_{n\in\id{0}{T}}\Jt_n(f_{n+1}(\star,\cdot))\,;\, f_{n+1}\in\L^2(\Xb)^{\circ n}\otimes\L^2(\Xb),\, n\in\id{0}{T},\, T\in\N\bigg\}.
\end{equation}
The operator divergence is introduced as the \textit{creation operator} acting on $\L^2(\P)$, that can be, thanks to Theorem \ref{Chaos_R_th}, understood as a Fock space.

\begin{definition}\label{Divergence_creation_def}
Let the linear, unbounded, closable operator $\delta\,:\, \dom\delta\rightarrow \L^2(\P)$ whose domain $\dom\delta$ (that will be described later) contains the set of processes which expansion is of the form $\sum_{n\in\Z_+}\Jt_n(f_n(\star,\cdot))$ and satisfies
\begin{equation*}
\sum_{n\in\Z_+}(n+1)!\|\bar{f}_{n+1}\|_{\L^2(\Xb)^{n+1}}<\infty,
\end{equation*}
and that is defined for any element $\Jt_n(f_{n+1}(\star,\cdot))$ of $\mathcal U$ by
\begin{equation}\label{Divergence_creation_eq}
\delta\big(\Jt_n(f_{n+1}(\star,\cdot))\big):=\Jt_{n+1}(\bar f_{n+1}),
\end{equation}
where
\begin{equation*}
\bar f_{n+1}=\frac{1}{n+1}\sum_{i=1}^{n+1}f_{n+1}\big((t_1,k_1),\cdots,(t_{i-1},k_{i-1}),(t_{i+1},k_{i+1}),\cdots,(t_{n+1},k_{n+1}),(t_i,k_i)\big).
\end{equation*}
\end{definition}
\noindent 
In the setting of classical Malliavin calculus, the divergence of adapted processes coincides with the Itô integral. We get the analogue in our context, where the role of the Itô integral is played by the $\mathcal R$-integral. Indeed, the equality $\delta(\car_\A)=\eta(\A)-\nu(\A)$ holds  for any $\A\in\mathcal X$ and leads for any $u\in\mathcal U$ to
\begin{equation}\label{delta_u_adapted_eq}
\delta (u)=\Jt_1(u)=\sum_{(t,k)\in\Xb}\uk(\eta,(t,k))\D\Rm_{(t,k)}.
\end{equation}
This property holds for any $\P\otimes\nu$-square integrable process $u$. Let $u=\Jt_{n-1}(f_{n}(\star,\cdot))$ for some $f_n\in\L^2(\Xb)^{n}$; the adaptedness of $u$ implies that $f_n(\star,(t,k))=g_n(\star,(t,k))\car_{\id{1}{t-1}^n}$ for some $g_n\in\L^2(\Xb)^{n}$. The result follows by writing 
\begin{equation*}
\delta(u)=\Jt_n(\bar f_{n+1})=n\sum_{(t,k)\in\Xb}\Jt_{n+1}\big(\bar g_n(\star,(t,k))\car_{\id{1}{t-1}^n}\big)\,\D\Rm_{(t,k)}=\sum_{(t,k)\in\Xb}\uk(\eta,(t,k))\D\Rm_{(t,k)}.
\end{equation*}
\noindent 
To state the property of closability of the gradient we need an integration by parts formula, appearing as a duality relation between $\Dm$ and $\delta$. Here is its version restricted to cylindrical functionals and simple processes.
\begin{proposition}[Integration by parts formula on $\mathcal S\times\mathcal U$]\label{IPP_S_prop}
For any $(\Fm,u)\in\cyl\times\mathcal U$,
\begin{equation}\label{IPP_S_eq}
\esp{\Fm\delta u}=\esp{\textless\Dm\Fm, u\textgreater_{\L^2(\Xb,\tilde\nu)}}.
\end{equation}
\end{proposition}
\begin{corollary}[Closability]\label{Closability_corollary}
The operator $\Dm$ is closable from $\L^2(\P)$ to $\L^2(\P\otimes\nu)$.
\end{corollary}
\noindent
By adjunction the operator $\delta$ is also closable from $\L^2(\P\otimes\nu)$ to $\L^2(\P)$. Thus the domain $\DD$ of $\Dm$ is the closure of $\cyl$ with respect to the norm
\begin{equation*}
\|\Fm\|_{\DD}:=\left(\|\Fm\|_{\L^2(\P)}^2+\|\Dm\Fm\|_{\L^2(\P\otimes\tilde\nu)}^2\right)^{1/2},
\end{equation*}
whereas the domain of $\delta$ is given by
\begin{equation*}
\dom{\delta}=\left\{u\in\L^2(\P\otimes\nu):\exists c>0,\forall\Fm\in\DD,\,|\textless \Dm\Fm,u \textgreater|_{\L^2(\P\otimes\nu)}\pp c\|\Fm\|_{\L^2(\P)}\right\}.
\end{equation*}
The integration by parts formula can be thus extended to the respective domains of $\Dm$ and $\delta$ to get what appears as a generalised commutation property: for any $\Fm\in\DD$, $u\in\dom\delta$,
\begin{equation}\label{IPP_gen_eq}
\esp{\Fm\,\delta u}=\esp{\textless\Dm\Fm, u\textgreater_{\L^2(\Xb,\tilde\nu)}}.
\end{equation}

\subsubsection{The Ornstein-Uhlenbeck structure}

\noindent 
This section is devoted to the construction of an Ornstein-Uhlenbeck structure around the eponymous semi-group $(\Pm_\tk)_{\tk\in\R_+}$ and its generator $\L$. Define the \textit{Ornstein-Uhlenbeck semi-group} by its action on the chaotic decomposition: for any $\Fm\in\L^2(\P)$ decomposed as \eqref{Chaos_R_eq},
\begin{equation*}
\Pm_\tk\Fm=\sum_{n\in\Z_+}e^{-n\tk}\Jt_n(f_n).
\end{equation*}

\begin{proposition}
The domain of the Ornstein-Uhlenbeck operator $\L$  (also called \textit{number operator}) is the set of random variables $\Fm\in\L^2(\P)$ which chaotic decomposition satisfies  \eqref{Dom_D_condition_chaos_eq} (in particular $\dom\L\subset\DD_0$). For any $\Fm\in\dom\L$ of expansion \eqref{Chaos_R_eq},
\begin{equation*}
\L\Fm=-\sum_{n\in\N}n\Jt_n(f_n).
\end{equation*}
It satisfies the remarkable identity: $\Fm\in\dom\L$ if and only if $\Fm\in\DD$ and $\Dm\Fm\in\dom\delta$ and, in this case, $\L\Fm=-\delta\Dm\Fm$. 
\end{proposition}
\begin{proof}
The identity $\L\Fm=-\delta\Dm\Fm$ can be stated first for $\Fm=\Jt_n(f_n)$ with $f_n\in\L^2(\Xb)^n$, using \eqref{delta_u_adapted_eq} and then extended to $\DD$ by closability of the operator $\Dm$. 
\end{proof}

\noindent
The inverse of the number operator, denoted $\L^{-1}$ is defined on the subspace of $\L^2(\P)$ made of random variables with null expectation, and that is given, for any $\Fm$ written as \eqref{Chaos_R_th}, by
\begin{equation}\label{inverseL}
\L^{-1}\Fm=-\sum_{n\in\N}\frac{1}{n}\Jt_n(f_n).
\end{equation}
\subsection{From combination of $\L^1$ and $\L^2$ theories to a unified Markov-Malliavin structure}
As seen in the previous subsections, the operators $\Dm^+/\Dm,\,\widetilde\delta/\delta$ and $\widetilde\L/\L,$ have a meaning either in $\L^1$ or $\L^2$ context. In this section we combine $\L^1$ and $\L^2$ theories to formalize a unified Markov-Malliavin structure.
\subsubsection{Operators $\Dm^+$ and $\Dm$: Stroock's formula}
\noindent
Within additional hypotheses, the operators $\Dm^+$ and $\Dm$ coincide. The very definition of the domain of the operator $\Dm$ and the chaotic decomposition ensure that if $\Fm\in\DD$, then $\Dm\Fm\in\L^2(\P)$. The following lemma provides the reciprocal, as well as a more tractable expression of the Malliavin derivative, in terms of a difference operator acting on $\L^2(\P)$. 
\begin{proposition}\label{Grad_difference_prop}
Let $\Fm\in\L^2(\P)$. If  $\Dm^+\Fm\in\L^2(\P\otimes\nu)$, then $\Fm\in\DD$. Moreover, 
\begin{equation}\label{Grad_difference_eq}
\Dm\Fm=\Dm^+\Fm\;;\; \; \P\otimes\nu\mathrm{-a.s.}
\end{equation}
\end{proposition}
\begin{remark}\label{Grad_J1_rem}
We can retrieve one of the specific identities existing in the Gaussian and Poisson spaces: for any process $u\in\mathcal U$, $\Dm_{(t,k)}\Jt_1(h)=h(t,k).$
Let $\rk_{(t,k)}$ be the representative of the process $\Rm$.
Applying \eqref{Gradient_def_eq} to $\Fm=\Jt_1(h)$ gives
\begin{equation*}
\Dm_{(t,k)}\Jt_1(h)=\sum_{s\in\N}\sum_{\ell\in\E}h(s,\l)\,\big[\rk(\pi_{t}(\eta)+\delta_{(t,k)},(s,\l))-\rk(\eta,(s,\l))\big]=h(t,k),
\end{equation*}
by noting that $\rk(\pi_{t}(\eta)+\delta_{(t,k)},(s,\l))-\rk(\pi_{t}(\eta),(s,\l))=\car_{\{(t,k)\}}((s,\l))$.
\end{remark}

\noindent 
The integrands of multiple integrals appearing in the chaotic decomposition of $\Fm$ can be expressed in terms of iterative Malliavin derivatives of $\Fm$. This entails the useful following lemmas. In fact, the operator $\Dm^+$ can be canonically iterated by letting $\Dm^{(1)}=\Dm^+$ and defining the $n$-th ($n\in\N$) \textit{difference operator} by the recursion formula $\Dm^{(n)}=\Dm^+(\Dm^{(n-1)})$. We get explicitly for any $\Fm\in\mathcal L^0(\O)$,
\begin{equation}\label{Mult_D_eq}
\Dm_{(\mathbf t_n,\mathbf k_n)}^{(n)}\Fm
\nonumber=\Dm_{(t_1,k_1)}^+\big(\Dm_{(t_2,k_2),\dots,(t_n,k_n)}^{(n-1)}\Fm\big)=\sum_{\J\subset [n]}(-1)^{n-|\J|}\Fm\Big(\pi^{[n]}(\cdot)+\sum_{{j\in\J}}\delta_{(t_j,k_j)}\Big),
\end{equation}
where $\pi^{[n]}:=\bigcirc_{t=1}^n\pi_t$.
This satisfies the remarkable identities that lead to the expression of the functions $f_n$ in \eqref{Chaos_R_eq} in terms of  the $n$-th difference operator, called \textit{Stroock's formula}.

\begin{lemma}\label{Mult_D_esp_lem}
For any $\Fm\in\L^2(\P)$,
\begin{equation*}
\esp{\Dm_{(\mathbf t_n,\mathbf k_n)}^{(n)}\Fm}=\esp{\Fm\prod_{i=1}^n\frac{\D\Rm_{(t_i,k_i)}}{\kappa_i}}.
\end{equation*}
\end{lemma}

\begin{lemma}\label{Stroock_formula_esp_lem}
For any $\Fm,\Gm\in\L^2(\P)$,
\begin{equation*}\label{Stroock_formula_esp_eq}
\esp{\Fm\Gm}=\esp{\Fm}\esp{\Gm}+\sum_{n\in\N}\frac{1}{n!}\textless\mathbf E[\Dm^{(n)}\Fm],\mathbf E[\Dm^{(n)}\Gm]\textgreater_{\L^2(\Xb,\nu)^{\otimes n}}.
\end{equation*}
\end{lemma}

\begin{proposition}[Stroock's formula]\label{Stroock_formula_lem}
Let $\Fm\in\L^2(\P)$. Then, $\Dm^{(n)}\Fm\in\L^2(\P\otimes\nu^{\otimes n})$ for any $n\in\N$, and $\Fm$ admits a chaotic decomposition of the form \eqref{Chaos_R_th} with $f_0=\esp{\Fm}$ and 
\begin{equation}\label{Stroock_formula_eq}
f_n((\mathbf t_n,\mathbf k_n))=\frac{1}{n!}\esp{\Dm^{(n)}_{(\mathbf t_n,\mathbf k_n)}\Fm}\; ; \; \forall (\mathbf t_n,\mathbf k_n)\in\Xb^n.
\end{equation}
\end{proposition}

\subsubsection{$\widetilde\L$ and Ornstein-Uhlenbeck operators: Mehler's formula}

\noindent
In this part, we provide an integral representation of $(\Pm_\tk)_{\tk\in\R_+}$ in $\L^1(\P)$, called \textit{Mehler's formula}. We proceed in a similar fashion as done in the Poisson setting by Last, Peccati and Schulte  \cite{LastPeccatiSchulte2016}. Let $\eta\in\mathfrak N_\Xb$.
Consider the binomial process $\Nm$ associated to $\eta$ and split it into two processes $\Nm^{(\gamma)}$ according to independent random draws of a Bernoulli random variable of mean $\gamma$. This means that any point charged by $\eta$ belongs to $\eta^{(\gamma)}$ with probability $\gamma$ and to $\eta^{(1-\gamma)}$ with probability $1-\gamma$. Important point: since the measure $\nu$ is not diffuse, we need to ensure that a point $(t,k)\in\eta$ cannot be simultaneously charged by $\eta^{(\gamma)}$ and $\eta^{(1-\gamma)}$. In other words, $(t,k)$ is either in the support of $\eta^{(\gamma)}$ or in that of $\eta^{(1-\gamma)}$. Considering $\eta$ as a proper process via Definition \eqref{Proper_process_def_eq}, let $\eta_{\mathrm K^\tk}$ be the $\mathrm K^\tk$-\textit{marking} (see Last and Penrose \cite{LastPenrose2017}, definition 5.3) of $\eta$ defined by
\begin{equation*}
\eta_{\mathrm K^\tk}=\sum_{t=1}^{\Nm}\delta_{((\Tm_t,\V_t),\varepsilon_t^\tk)},
\end{equation*}
where $(\varepsilon_t^\tk)_{t\in\N}$ is a sequence of variables which conditional distribution given $\{\Nm=n\}$ (for $n\in\N$) \textbf{and} $\{(\Tm_t,\V_t),\, t\in\id{1}{n}\}$, is that of independent random variables defined by $\varepsilon_t^{\tk}=\car_{\{\theta_t\leqslant \tk\}}$,
and $(\theta_t)_{t\in\N}$ is a sequence of independent exponential random variables of mean $1$. We can prove that $\eta_{\mathrm K^\tk}$ is a binomial process on $\Xb\times\{0,1\}$ of intensity measure $\nu\otimes\mathrm K^\tk$. Denote also
\begin{equation}\label{wtk_def_eq}
\eta^{\tk,0}:=\eta_{\mathrm K^\tk}(\cdot\times \{0\}) \quad \text{and} \quad \eta^{\tk,1}:=\eta_{\mathrm K^\tk}(\cdot\times \{1\}),
\end{equation}
that are (not independent) binomial processes with respective intensities $e^{-\tk}\nu$ and $(1-e^{-\tk})\nu$. To see it, one can use the Laplace characterisation of binomial processes, that can be found in Last and Penrose (see \cite{LastPenrose2017}, exercise 3.5): the Laplace transform of a mixed binomial process (which definition is given by \eqref{Mixing_bin_def_eq}) with mixing measure $\mathbf K$ and sampling distribution $\Q$ is the function defined on $\R_+(\Xb)$, the set of measurable functions from $\Xb$ to $\R_+$, by 
\begin{equation*}
\EuScript L_{\eta}(f)=\mathcal G_{\mathbf K}\Big(\int e^{-f}\,\d\Q\Big) \; ; \; f\in\R_+(\Xb),
\end{equation*}
where $\mathcal G_{\mathbf K}(x):=\sum_{n\in\Z_+}\mathbf K(\{n\})x^n,$ for $ x\in[0,1]$. Then, for any $f\in\R_+(\Xb)$,
\begin{equation*}
\EuScript L_{\eta^{\tk,0}}(f)=\sum_{n\in\Z_+}\mathbf K(\{n\})\bigg(\sum_{k\in\E}e^{-f(k,0)}\,e^{-\tau}\Q(\{k\})\bigg)^n=\mathcal G_{\mathbf K}\Big(\int_{\E\times\{0,1\}} e^{-f}\,(e^{-\tau}\d\delta_0\otimes\d\Q)\Big),
\end{equation*}
so that $\eta^{\tk,0}$ is a binomial process with  intensity $e^{-\tk}\nu$. The computation of the Laplace transform of $\eta^{\tk,0}+\eta^{\tk,1}$ suffices to see that the two processes are not independent. Nevertheless, we have $\eta^{\tk,0}+\eta^{\tk,1}=\eta$.  The formula below is very similar to the one existing in the Poisson case that can be found in the work of Last, Peccati and Schulte \cite{LastPeccatiSchulte2016} or in its original formulation in Privault (see \cite{Privault_stochastica}, Lemma 6.8.1). The main difference lies in the presence here of the random variable $\varepsilon$. Implicitly defined in the thinning appearing in Mehler's formula for Poisson processes, it is explicitly required here to guarantee that a same point can not be weighted simultaneously by $\eta^{\tk,0}$ and $\tilde\eta$.
\begin{proposition}\label{Mehler_formula_prop}
Let $\eta\in\widehat{\mathfrak N}_\Xb$ and $\Fm\in\L^1(\P)$ of representative $\fk$. For any $\tk\in\R_+$,
\begin{equation}\label{Mehler_formula_prop_eq}
\Pm_\tk\Fm=\Pm_\tk\fk(\eta^{\tk,0}+\eta^{\tk,1})=\int\esp{\fk(\eta^{\tk,0}+\varepsilon^\tk\tilde\eta)\big|\eta}\Pi_{\nu}(\d\tilde\eta)\; ; \;  \P\mathrm{-a.s.},
\end{equation}
where $\Pi_{\nu}$ denotes the distribution of a marked binomial process of intensity measure $\nu$ and $\tilde\eta$ is a point process which distribution given $\eta$ follows the rule:
\begin{equation}\label{Law_tilde_eta}
\P((t,k)\in\tilde\eta\,|\,(t,k)\notin\eta)=\lambda\Q(\{k\})\quad\text{and}\quad \P((t,k)\notin\tilde\eta\,|\,(t,k)\in\eta)=1-\lambda\Q(\{k\}).
\end{equation}
\end{proposition}

\noindent
The first equality in \eqref{Mehler_formula_prop_eq} ensures that for any $\Fm\in\L^1(\P)$, $\tk\in\R_+$,
\begin{equation*}
\esp{\Pm_\tk\Fm}=\esp{\Fm},
\end{equation*}
while Jensen's inequality together with \eqref{Mehler_formula_prop_eq} imply the contractivity property of the semi-group: for any $p\in\N$,
\begin{equation}\label{Contractivity_PtF_eq}
\esp{|\Pm_\tk\Fm|^p}\pp\esp{|\Fm|^p}.
\end{equation}
\noindent
The semi-group $(\Pm_\tk)_{\tk\in\R_+}$ satisfies the usual commutation property:
\begin{proposition}\label{Commutation_PtF_prop}
For any $\Fm\in\L^2(\P)$, and $\tk\in\R_+$,
\begin{equation}\label{Commutation_PtF_eq}
\Dm\Pm_\tk\Fm=e^{-\tk}\Pm_\tk\Dm\Fm.
\end{equation}
\end{proposition}
\noindent
The commutation property induces several and useful corollary results gathered in the following statement. Unfortunately, even if $\Fm\in\dom\L$ such that $\Dm\Fm\in\L^1(\P)$, we can not state $\L\Fm=\widetilde\L\Fm$ $\P$-almost surely. Indeed, on the one hand, $\delta(\Dm\Fm)=\sum_{(t,k)\in\Xb}(\Dm_{(t,k)}\Fm)\D\Rm_{(t,k)}$ whereas if  $\Dm\Fm\in\L^1(\P)$, $\tilde\delta(\Dm\Fm)=\sum_{(t,k)\in\Xb}(\Dm_{(t,k)}\Fm)\D\Zm_{(t,k)}$. Nevertheless, follows from \eqref{matrix_basis_eq} that
\begin{equation}\label{L_widetildeL_eq}
\L\Fm=-\sum_{(t,k)\in\Xb}(\Dm_{(t,k)}\Fm)\sum_{\l\in\E}\mathfrak m_{k\l}^{-1}\D\Zm_{(t,\l)}=\sum_{(t,\l)\in\Xb}\sum_{k\in\E}\mathfrak m_{k\l}^{-1}(\Dm_{(t,k)}\Fm)\D\Zm_{(t,\l)}=:\widetilde\L\widetilde\Fm,
\end{equation}
where $\widetilde\Fm$ is a square-integrable random variable such that $\Dm_{(t,\l)}\widetilde\Fm=\sum_{k\in\E}\mathfrak m_{k\l}^{-1}(\Dm_{(t,k)}^+\Fm)$ for any $(t,\l)\in\Xb$. This is well and uniquely defined provided $\mathbf E[\widetilde\Fm]$ is given; indeed as a consequence of Clark formula (see forthcoming section 4), the knowledge of $(\Dm_{(t,k)}\Fm,\, (t,k)\in\Xb)$ and $\esp{\Fm}$ provides the expression of $\Fm$ $\P$-almost surely. 
\begin{corollary}\label{Inverse_L_cor}
For any $\Fm\in\L^2(\P)$ such that $\esp{\Fm}=0$,
\begin{equation}\label{Inverse_L_eq}
\L^{-1}\Fm=-\int_{0}^{\infty}\Pm_{\tk}\Fm\,\d\tk,\qquad \P\otimes\nu-\mathrm{a.e.}
\end{equation}
Moreover,
\begin{equation}\label{Inverse_L2_eq}
-\Dm\L^{-1}\Fm=\int_{0}^{\infty}e^{-\tk}\Pm_{\tk}\Dm\Fm\,\d \tk, \qquad \P\otimes\nu-\mathrm{a.e.}
\end{equation}
\end{corollary}
\begin{remark}
The combination of Corollary \ref{Inverse_L_cor} with the contraction property of $(\Pm_\tk)_{t\in\R_+}$ enables to bound $\Dm\L^{-1}\Fm$  with respect to the norm of $\Dm\Fm$: 
$\;\|\Dm\L^{-1}\Fm\|_{\L^2(\P\otimes\nu)}\pp  \|\Dm\Fm\|_{\L^2(\P\otimes\nu)}$.
\end{remark}

\begin{remark}\label{E_singleton_remark}
In that case where $\E$ is a singleton, i.e. $\eta$ is a simple binomial process, we have $\D\Rm_t=\D\Zm_t$ for any $t\in\N$, so that $\widetilde\L=\L$ and as a result $\widetilde\Gamma=\Gamma$ $\P$-almost surely (by letting $\Gamma(\Fm,\Gm)=1/2[\L(\Fm\Gm)-\Fm(\L\Gm)-\Gm(\L\Fm)]$, for $\Fm,\Gm\in\dom\L$). We retrieve thus the coincidence of $\L^1$ operators and Malliavin's ones within the remarkable association of $\L^1$ and $\L^2$ theories, as well as the natural link between Malliavin and Gamma calculus on the paradigm of $\L^2$ framework; both are highlighted in the Poisson case by D\"obler and Peccati in \cite{DoblerPeccati2018}.
\end{remark}
\noindent
As a conclusion, the combination of $\L^1$ and $\L^2$ theories is embodied by the existence of a $\L^1$-correspondence $(\Dm^+,\tilde\delta,\widetilde\L,(\Pm_\tk)_\tk)$ to the tuple of random objects $(\Dm,\delta,\L,(\Pm_\tk)_\tk)$ that equips the space $(\mathfrak N_\Xb,\F,\P)$. Moreover, each element of $(\Dm,\delta,\L,(\Pm_\tk)_\tk)$  can be linked to one another of the tuple through one or a combination of the following identities and properties: the generalised integration by parts formula $\eqref{IPP_gen_eq}$, the identity $\L=-\delta\Dm$,  the generation of the semi-group $(\Pm_\tk)_{\tk}$ by $\L$, making it a Markov-Malliavin unified structure.
\subsection{Comparison with pre-existing theories in the Poisson and Rademacher settings}
It seems reasonable to ask: if we let the sequence $\V$ is deterministic constant equal to $1$ (respectively $\lambda=1$ and $\E=\{-1,1\}$), can we retrieve some element of stochastic analysis for Poisson processes on the real line (respectively Rademacher processes)?\\
\noindent  
Considering first the case where the sequence $\V$ be deterministic constant equal to $1$ leads to define the orthogonal family  $\mathcal Z^{\Pm}=\{\Delta \Zm_{t}^{\Pm}\,;\, t\in\N\}$ by
$\Delta \Zm_{t}^{\Pm}=\car_{\{\Delta \Nm_t=1\}}-\lambda,$ and the stochastic integral defined as the application $\J^\Pm_1 \,:\, f\in\L^2(\N,\lambda\#)\mapsto \J^\Pm_1(f)$ ($\#$ is the counting measure). The process $\eta$ is then the discrete analogue of the standard Poisson process on the real line. The gradient reads for $\Fm\in\dom\Dm^\Pm$ of representative $\fk$,
\begin{equation}\label{Gradient_Poisson_analogue_eq}
\Dm_{t}^{\Pm}\Fm= \fk(\pi_t(\eta )+\delta_{t})-\fk(\pi_t(\eta )),
\end{equation}
which is - up to a constant - a reminiscent of the gradient used by Decreusefond and Flint \cite{Decreusefondflint} on the Poisson space that is written (with corresponding notations): $\fk(\eta+\delta_{t})-\fk(\eta-\delta_{t}).$ Nevertheless the operator $\Dm^\Pm$ is different from the usual one for Poisson processes on the real line:  $\nabla_{t}\Fm=\fk(\eta+\delta_{t})-\fk(\eta)$. This definition is not suitable in the present context. Indeed, as stated by N. Privault (see \cite{Privault_stochastica}, proof of the proposition 6.4.7),
\begin{equation}\label{Annihilation_difference_eq}
\nabla_{t}\Jt_n(f_n)=\car_{\{t\notin\eta\}}\Jt_{n-1}(f_{n-1}(\star,t)),
\end{equation}
which is $\P$-almost surely equal to $\Jt_{n-1}(f_{n-1}(\star,t))$ since the intensity measure is diffuse in the Poisson case. This does not hold in our framework; the definition of the gradient \eqref{Gradient_Poisson_analogue_eq} is thus justified to guarantee \eqref{Annihilation_difference_eq}, in order to make the difference ($\Dm^+$) and annihilation ($\Dm$) operators coincide and thereby combine $\L^1$ and $\L^2$ theories.\\

\noindent 
Let now $\lambda=1$ and $\E=\{-1,1\}$. Basically, that means that the underlying binomial process jumps every time step. A Rademacher process $(\Xm_t)_{t\in\N}$ can be defined by letting $\Xm_t=\V_t\in\{-1,1\}$ and  $\Y_t:=(2pq)^{-1/2}(\D\Zm_{(t,1)}+\D\Zm_{(t,-1)})=(2pq)^{-1/2}(\Xm_t-p+q)$ where the $\D\Zm_{t,\cdot}$ are defined as usual by \eqref{Defintion_DZ_eq} and $p:=\P(\Xm_t=1)=1-q$. Thus $(\Y_t)_{t\in\N}$ is a $\F$-(normal) martingale. By properly defining the function $g$ on $\L^2(\P)$ such that $\Fm:=\Fm(\Xm_1,\dots,\Xm_T)=g(\sqrt{2pq}\Y_T)$ and $\overline\Dm_t\Fm:=\Dm_{(t,1)}g(\sqrt{2pq}\Y_t)-\Dm_{(t,-1)}g(\sqrt{2pq}\Y_t)$, we get
\begin{equation*}
\overline\Dm_t\Y_s=\frac{2}{\sqrt{2pq}}\car_{\{t\}}(s)=\sqrt{\frac{2}{pq}}\widehat\Dm\Y_s,
\end{equation*}
that is - up to a constant - the expression of the gradient $\widehat\Dm$ defined on the Rademacher space (see for instance Privault \cite{Privault_stochastica}, Proposition 1.6.2). All identities and formulas, such as the Clark formula and the predictable representation (see Privault \cite{Privault_stochastica}, chapter 1), are inherited by construction.

\section{Functional identities}\label{Functional_identities_sec}

\noindent
In this section we derive some functional identities from our formalism. In fact, we can get the analogues of almost all identities existing in the Wiener or Poisson spaces that use the similar "Markov-Malliavin" structure. In a perspective of the forthcoming applications in the trinomial model we have chosen to focus on and present only two of them: Girsanov theorem and the Clark formula (and some corollaries). 

\subsection{Girsanov theorem}

\noindent
We provide our construction with the analogue of Girsanov theorem, which is reminiscent of  that stated for compound Poisson processes (see Privault \cite{Privault2018jump}, Theorem 15.11).

\begin{theorem}[Girsanov theorem]\label{Girsanov_th}
Let $T\in\N$ and $\widetilde\P$ be a probability measure equivalent to $\P$ on $\f{T}$. Then, there exist $\tilde\lambda\in(0,1)$ and a measure $\widetilde\Q$ on $\E$ such that $\widetilde\P$ is of compensator $\nu_{\widetilde\P}:=\tilde\lambda\#\otimes\widetilde\Q$. Moreover, for any $t\in\id{1}{T}$,
\begin{equation*}
\frac{\d\widetilde\P}{\d\P}\Big|_{\f{t}}=\xi_t(h),
\end{equation*}
where, if $\E=\{k^i,\, i\in\Z\}$, $h$  is the element of $\L^2(\Xb)$ such that 
$\Jt_1(h)=\Jt_1(g\,;\mathcal Z)$ with
\begin{equation}\label{Girsanov_drift_eq}
g{(t,k^i)}=\Big(\frac{\tilde\lambda\widetilde\Q(\{k^i\})}{\lambda \Q(\{k^i\})}-\frac{1-\tilde\lambda}{1-\lambda}\Big),
\end{equation}
for all $(t,i)\in\id{1}{T}\times\Z$.
\end{theorem}

\begin{corollary}\label{Girsanov_cor}
Let $\tilde\lambda\in(0,1)$ and $\widetilde\Q$ be a measure  on $\E$. Let $\varphi$ be the function defined on $\E$ 
\begin{equation}\label{Girsanov_varphi_eq}
\varphi=\frac{\tilde\lambda(1-\lambda)}{\lambda(1-\tilde\lambda)}\frac{\sum_{k\in\E}\widetilde \Q(\{k\})\,\car_{\{k\}}}{ \sum_{k\in\E} \Q(\{k\})\,\car_{\{k\}}}-1,
\end{equation}
Then, under the probability measure $\widetilde\P$ such that
\begin{equation*}
\d\widetilde\P=\left(\frac{1-\tilde\lambda}{1-\lambda}\right)^T\prod_{s=1}^{\Nm_T}(1+\varphi(\V_s))\,\d\P,
\end{equation*}
the process $\Y$ defined by \eqref{Compound_pp_def_eq}
is a compound binomial process on $\Xb_T$ of intensity measure $\nu_{\widetilde\P}:=\tilde\lambda\#\otimes\widetilde\Q$.
\end{corollary}

\begin{remark}
 The perturbations described by the \textit{shift space} in Gaussian analysis (Cameron-Martin space for the Brownian motion in particular) act here on what characterizes the jumps: their occurrence and their height, respectively parametrised by $\lambda$ and $\Q$. The similar phenomenon is observed in the Poisson space.
\end{remark}

\subsection{Clark formula and corollaries}
The Brownian martingale representation theorem says that a martingale adapted to the filtration of a Brownian motion is in fact a stochastic integral. The Clark formula gives the expression of the integrand of this stochastic integral in terms of the Malliavin gradient of the terminal value of the martingale. We here have the analogue formula.
\begin{proposition}[Clark formula]\label{Clark_formula_prop} For any $\Fm\in\L^2(\P)$,
\begin{equation}\label{Clark_formula_eq}
\Fm=\esp{\Fm}+\sum_{(t,k)\in\Xb}\mathbf E\big[\Dm_{(t,k)} \Fm\,|\,\f{t-1}\big]\,\D\Rm_{(t,k)}.
\end{equation}
\end{proposition}
\begin{remark}\label{Boundness_D_rem}
The operator $\Fm\in\L^2(\P)\mapsto \big(\mathbf E[\Dm_{(t,k)}\Fm\,|\,\f{t-1}],\, (t,k)\in\Xb\big)$ is bounded with norm equal to $1$. Indeed, from \eqref{Clark_formula_eq} together with the isometry property \eqref{isometry_eq},
\begin{equation*}
\big\|\mathbf E[\Dm_{\cdot}\Fm\,|\,\f{\cdot-1}]\big\|_{\L^2(\P\otimes\tilde\nu)}=\big\|\Fm-\mathbf E[\Fm]\big\|_{\L^2(\P)}\pp \big\|\Fm-\mathbf E[\Fm]\big\|_{\L^2(\P)}^2+ \big(\mathbf E[\Fm]\big)^2=\|\Fm\|_{\L^2(\P)}^2,
\end{equation*}
with equality in case $\Fm=\Jt_1(f_1)$ for some $f_1\in\L^2(\Xb)$.
\end{remark}
\noindent
As a direct consequence of Lemma \ref{Cond_int_lem} and Clark formula \eqref{Clark_formula_eq} we get the following corollary. 
\begin{corollary}\label{Clark_conditionnal_lem}
For any $t\in\N$ and $\Fm\in\L^2(\P)$,
\begin{equation}\label{Clark_conditionnal_eq}
\Fm=\mathbf E\big[\Fm|\f{t}\big]+\sum_{s\pg t+1}\sum_{k\in\E}\mathbf E\big[\Dm_{(s,k)}\Fm\,|\,\f{s-1}\big]\D\Rm_{(s,k)}
\end{equation}
\end{corollary}
\noindent
We can state the analogue of the so-called Chernoff-Nash-Poincaré inequality of Gaussian analysis (see Chernoff \cite{Chernoff1981note}, Nash \cite{Nash1958continuity}). Our result is clearly a reminiscence of its counterpart in the Poisson space (see Last and Penrose \cite{LastPenrose2011}, Wu \cite{Wu:2000lr}) or for independent random variables
(see Decreusefond and Halconruy \cite{DecreusefondHalconruy}).
\begin{corollary}[Poincaré inequality]\label{Poincare}
For any $\Fm\in\L^2(\P)$,
\begin{equation*}
\mathrm{var}(\Fm)\pp\esp{\int_{\Xb}|\Dm_{(t,k)}\Fm|^2\,\d\tilde\nu(t,k)}.
\end{equation*}
\end{corollary}

\begin{remark}\label{Clark_Z_rem}
Assume $\E=\{k^1,\cdots,k^n\}$. 
The transposition of the Clark formula in terms of $\mathcal Z$-integrals can be easily deduced from \eqref{Clark_conditionnal_eq} together with \eqref{family_R_eq}; for any $\Fm\in\L^2(\P)$,
\begin{multline*}
\Fm=\mathbf E[\Fm]+\sum_{s\pg t+1}\sum_{j=1}^n\sum_{i=1}^j\mathfrak m_{k^jk^i}^{-1}\mathbf E\big[\Dm_{(s,k^j)}\Fm\,|\,\f{s-1}\big]\D \Zm_{(s,k^i)}\\
=\mathbf E[\Fm]+\sum_{s\pg t+1}\sum_{i=1}^n\mathbf E\big[\Dm_{(s,k^i)}^{\mathcal Z}\Fm\,|\,\f{s-1}\big]\D \Zm_{(s,k^i)},
\end{multline*}
where $\Dm_{(s,k^i)}^{\mathcal Z}\Fm=\sum_{j=i}^{n}\mathfrak m_{k^jk^i}^{-1}\Dm_{(s,k^j)}\Fm$.
\end{remark}
\section{Applications}\label{Applications_sec}
\subsection{Malliavin-Stein method for compound Poisson approximation}
\noindent
The Stein method, initially developed to quantify the rate of convergence in the Central Limit Theorem (see Stein \cite{stein1972}) and then for Poisson convergence
(see for instance Barbour, Lars and Janson \cite{BarbourLarsJanson1992}), has become a very popular not to say  the most famous procedure to assess distances between two probability measures of the form
\begin{equation}\label{probametric}
\mathrm{dist}_{\mathcal T}(\P,\Q)=\underset{h\in\mathcal T}\sup\Big|\int_{\mathfrak{F}} h\,\d\P-\int_{\mathfrak{F}} h\,\d\Q\Big|,
\end{equation}
where $\mathcal T$ is a class of real-valued test functions. The class $\mathcal T$ is furthermore separating, in the sense that if 
$\int_{\mathfrak{F}} h\,\d\P=\int_{\mathfrak{F}} h\,\d\Q$ for all $h\in\mathcal T$ if and only if $\Q=\P.$ In particular, if $\mathcal T=\{\car_\A,\, \A\in\mathcal X\}$ coincides with the total \textit{total-variation distance} and will be denoted $\mathrm{dist}_{\mathrm{TV}}$. 
The first one consists in converting  the difficult initial problem \eqref{probametric} by the more tractable expression
\begin{equation*}
\sup_{\varphi\in \mathcal K}\Big|\esp{ \L\varphi(\Y)}\Big|=\sup_{\varphi\in \mathcal K}\Big|\esp{ \L_1\varphi(\Y)+\L_2\varphi(\Y)}\Big|,
\end{equation*}
where $\Y$ is a random variable of law $\Q$, $\L$ and $\mathcal K$ are respectively the Stein operator and the Stein class associated to the target measure $\P$. The aim of the second step is to developp tools in order to transform $\L_1\varphi(\Xm)$ into $-\L_2\varphi(\Xm)+\text{remainder}$.
This remainder is what gives the bound of the distance and, in a problem of
convergence, provides its rate. Besides, Nourdin and Peccati showed in \cite{Nourdin_2008} that this transformation step can be  advantageously performed using integration by parts in the sense of Malliavin calculus.
In this section, we make use of our formalism to provide an analogue of the Stein-Malliavin criterion for the Poisson (respectively compound Poisson) approximation by binomial (respectively marked binomial) functionals with respect to the total variation distance. This is defined  for two $\Z_+$-random variables $\Xm$ and $\Y$ (the case of interest here), by
\begin{equation*}
\mathrm{dist}_{\mathrm{TV}}(\P_{\Xm},\P_{\Y})=\underset{\A\subset\Z_+}\sup\big|\P(\Xm\in\A)-\P(\Y\in\A)\big|.
\end{equation*}
First, we can state a result for the Poisson approximation, in the same spirit as Peccati \cite{Peccati2011chen}. Within the same framework, let $\EuScript P(\lambda_0)$ the Poisson law with parameter $\lambda_0$. Consider for a given function $\varphi\,:\,\Z_+\rightarrow\R$, $\nabla \varphi$ the \textit{forward difference} $\nabla \varphi:=\varphi(\cdot+1)-\varphi$, and $\nabla^2\varphi$ its second iteration $\nabla^2:=\nabla(\nabla\varphi)$, that satisfies the useful (as proved in particular in Peccati \cite{Peccati2011chen}, proof of Theorem 3.3) inequality: for all $a,k\in\Z_+$,
\begin{equation}\label{Peccati_ineq}
\big|\varphi(k)-\varphi(a)-\nabla\varphi(a)(k-a)\big|\pp\frac{\|\nabla^2\varphi\|}{2}|(k-a)(k-a-1)|.
\end{equation}
For any $\A\subset\Z_+$, we denote by $\varphi_{\A}\,:\,\Z_+\rightarrow\R$ the unique solution to the Chen-Stein equation
\begin{equation}\label{Stein_Poisson_eq}
\P(\EuScript P(\lambda_0))-\car_{\A}(k)=k\varphi_{\A}(k)-\lambda_0\varphi_{\A}(k+1)\; ; \; k\in\Z_+,
\end{equation}
satisfying the boundary condition $\nabla^2\varphi_{\A}(0)=0$. The function class $\mathcal K=\{\varphi_\A,\; \A\subset\Z_+\}$ fulfils the estimates (be also found for instance in Peccati \cite{Peccati2011chen}),
\begin{equation*}
\|\varphi\|_{\infty}\pp\min\Big(1,\sqrt{\frac{2}{e\lambda_0}}\Big),\; \|\nabla\varphi\|_{\infty}\pp\frac{1-e^{-\lambda_0}}{\lambda_0},\;\text{and}\;\|\nabla^2\varphi\|_{\infty}\pp\frac{2-2e^{-\lambda_0}}{\lambda_0^2},
\end{equation*}
where we have denoted $\|\varphi\|_{\infty}=\max_{\A\subset\Z_+}\|\nabla\varphi_\A\|_{\infty}$, $\|\nabla\varphi\|_{\infty}=\max_{\A\subset\Z_+}\|\nabla\varphi_\A\|_{\infty}$ and $\|\nabla^2\varphi\|_{\infty}=\max_{\A\subset\Z_+}\|\nabla^2\varphi_\A\|_{\infty}$.
\begin{theorem}\label{Approximation_Poisson_th} Consider $\lambda_0\in\R_+^*$ and let $\Fm$ be a square-integrable $\Z_+$-valued random variable such that  $\esp{\Fm}=\lambda_0$. Then,
\begin{align}\label{Approximation_Poisson_eq}
\nonumber
\mathrm{dist}_{\mathrm{TV}}\big(\P_{\Fm},\EuScript P(\lambda_0)\big)
&\pp\frac{1-e^{-\lambda_0}}{\lambda_0}\esp{\big|\lambda_0-\textless\widetilde\Dm\Fm,-\Dm\L^{-1}(\Fm-\esp{\Fm})\textgreater_{\L^2(\Xb,\nu)}\big)\big|}\\
&+\frac{1-e^{-\lambda_0}}{\lambda_0^2}\esp{\int_{\N}\big|(\widetilde\Dm_{t}\Fm)(\widetilde\Dm_{t}\Fm-1)\big|\big|\Dm_{t}\L^{-1}(\Fm-\esp{\Fm})\big|\,\nu(\d t)}.
\end{align}	
\end{theorem}
\noindent
The aim is now to provide such a bound for the compound Poisson approximation.
Let $\EuScript P\EuScript C(\lambda_0,\gV)$ denote the law of a compound Poisson variable of  parameters $(\lambda_0,\gV)$, that means it can be written as the distribution of the variable
\begin{equation*}
\sum_{i=1}^{\Nm^\Pm}\V_i,
\end{equation*}
where $\Nm^\Pm$ is a Poisson random variable of mean $\lambda_0$ and $\{\V_i,\,i\in\N\}$ is a family of independent non-negative random variables of distribution $\gV$. For any $\A\subset\Z_+$, denote  $\psi_\A$ the unique solution of the Chen-Stein equation
\begin{equation}\label{Stein_Compound_Poisson_eq}
\car_{\A}(\l)-\P(\EuScript P\EuScript C(\lambda_0,\gV))=\l\psi_{\A}(\l)-\int_{\Xb}k\psi_{\A}(\l+k)\,\d\nu(t,k)\; ; \; \l\in\Z_+.
\end{equation}
The function class $\mathcal K'=\{\varphi_\A,\; \A\subset\Z_+\}$ satisfies the following estimate (see Erhardsson \cite{Erhardsson2005}, Theorem 3.5) 
\begin{equation}\label{Stein_class_bound_psi}
\mathfrak d_{\EuScript P\EuScript C}:=\max_{\A\subset\Z_+}\|\psi_\A\|_{\infty}\,\vee\,\max_{\A\subset\Z_+}\|\nabla\psi_\A\|_{\infty}\pp\min\Big(1,\frac{1}{\lambda_0\gV(\{1\})}\Big)e^{\lambda_0}. 
\end{equation}
\begin{proposition}\label{Approximation_compound_Poisson_th} Consider $\lambda_0\in\R_+^*$ and $\gV$ a probability distribution on $\N$. Let $\V_1$ be a random variable of law $\gV$ and $\Fm$ a square-integrable $\Z_+$-valued random variable such that  $\esp{\Fm}=\lambda_0\esp{\V_1}$.
\begin{align}\label{Approximation_compound_Poisson_eq}
\nonumber
\mathrm{dist}_{\mathrm{TV}}\big(\P_{\Fm}\,,\,&\EuScript P\EuScript C(\lambda_0,\gV)\big)\\
\nonumber
&\pp\Big|\int_{\Xb}\big[\Dm^+\widetilde\L^{-1}(\fk(\eta)-\esp{\fk(\eta)})\psi_\A(\fk(\pi_t(\eta)+\delta_{(t,k)}))-k\psi_\A(\fk(\eta)+k)\big]\d\nu(t,k)\Big|\\
&\quad+\mathfrak d_{\EuScript P\EuScript C}\Big|\int_{\Xb}\big[\Dm^+\widetilde\L^{-1}(\fk(\eta)-\esp{\fk(\eta)})-k\big]\,\d\nu(t,k)\Big|
.
\end{align}	
\end{proposition}
\begin{remark}
This result is only of interest in the case of the variable $\Fm$ is a marked binomial functional in the first chaos i.e. $\Fm=\J_1(f)$ for some $f\in\L^2(\Xb)$ (which corresponds to the framework of the application of subsection 5.3) and is no relevant for more complicated functionals. In this latter, we can provide a bound by means of a Taylor expansion and in terms of the iterated operator $\nabla^2$. That turns out to be sub-optimal in the first chaos case we will be interested in later, which justifies our choice not to present it.
\end{remark}

\subsection{Head run problems}
Consider a large number of  independent throws of a coin of  success (falling on face) probability $p\in(0,1)$. Whatever the value of $p$, there will be sequences where the coin will fall on face each time; this sequence is called a \textit{head run} and we aim at computing the probability that $\Um$, the lenght of the longest run of heads beginning in the first $n$ tosses, will be less to a test length $m\in\N$. The critical fact is that head runs occur in \textit{clumps}; indeed, if  there is the head of a run of length $m$ at position $i$, then with probability $p$, there will also be a run of length $m$  at position $i+1$. We need then to "declump" the sequences in order to count only the first count. To do that, let $(\Cm_i)_{i\in\N}$ be a sequence of independent and identically distributed Bernoulli variables of parameter $p$. Let $m$ be a fixed positive integer "test" value and consider the random variable
\begin{equation*}
\Um=\prod_{i=1}^m\Cm_i+\sum_{i=2}^n(1-\Cm_{i-1})\Cm_i\Cm_{i+1}\cdots\Cm_{i+m-1}
\end{equation*}
that gives the total number of clumps of runs of length $m$ or more. Note that $\esp{\Um}=p^m((n-1)(1-p)+1)=:\lambda_0$. Let $\Nm$ be a binomial process of intensity $p$. The random variable $\Um$ can be rewritten as
\begin{equation}\label{F_head_run}
\Um=\prod_{i=1}^m\D\Nm_i+\sum_{i=1}^{n-1}(1-\D\Nm_{i})\D\Nm_{i+1}\D\Nm_{i+2}\cdots\D\Nm_{i+m}=:\sum_{i=0}^{n-1}\Um_i,
\end{equation} 
so that it appears as a binomial functional with mean $p$ since the sequence $\V$ is here deterministic and constant equal to $1$. Its chaotic decomposition reads
\begin{equation*}
\Um=\esp{\Um}+\sum_{j=1}^{m+1}\Jt_j\big(f_j\car_{[n+m-1]}\big),
\end{equation*}
where, in particular, for $j\in\id{1}{m+1}$, $f_j(t_1,t_2,\cdots,t_j)=0,$ as soon as $\prod_{i=1}^{j-1}\car_{\{1\}}(t_{i+1}-t_i)=0$. Since $\Um\in\DD$, $\Dm\Um=\Dm^+\Um$ $\P\otimes\nu$-almost surely, and
\begin{equation*}
\Dm_t\Um=\prod_{i=1,i\neq t}^m\D\Nm_i+\sum_{i=1}^{n-1}\Big(\car_{[i+1,i+m]}(t)(1-\D\Nm_{i})\prod_{\l=1,i+\l\neq t}^{m}\D\Nm_{i+\l}+\car_{\{i\}}(t)\prod_{\l=1}^m\D\Nm_{i+\l}\Big).
\end{equation*}
\begin{theorem}\label{longest_head_run_th}
Let $\lambda_0=p^m((n-1)(1-p)+1)$. Then,
\begin{equation*}
\mathrm{dist}_{\mathrm{TV}}(\P_\Um,\EuScript P(\lambda_0))\pp p^{2m}[2(m-1)q^2+2mq+1]+(n-m+1)(1-p)p^{2m+1}\|\nabla\varphi\|_{\infty}.
\end{equation*}
\end{theorem}
\begin{remark}
The previous result gives an insight into the distribution of $\Tm_n$, the length of the longest head run. 
As explained in Arratia, Goldstein and Gordon \cite{ArratiaGoldsteinGordon1990}, as a consequence of the previous theorem, the distribution of $\Tm_n$ may be approximated as
\begin{equation*}
\P(\Tm_n<t)=\P(\Um=0)=e^{-\lambda_{0}}.
\end{equation*} 
The definition of a test lenght requires that $\lambda_{0}$ is bounded away from $0$ and $\infty$. In other words, this means the existence of a deterministic constant $c$ such that
\begin{equation*}
m=\log_{1/p}\big((n-1)(1-p)+1\big)+c.
\end{equation*} 
This assumption entails that the Poisson approximation in Theorem \ref{longest_head_run_th} is of order $1/n$ where the constant can be found by considering the $1/n$-order terms such as $np^{2m}$. We thus find the result of Arratia, Goldstein and Gordon in \cite{ArratiaGoldsteinGordon1990} who used the Chen-Stein method to deal with the local dependence structure of $\Um$.
\end{remark}	

\subsection{Number of occurrences of a word in a DNA sequence}

\noindent 
A DNA sequence can be represented by a finite series $\Xm_1\Xm_2\dots\Xm_n$ of characters taken from the alphabet $\mathcal A:=\{\A,\Cm,\Gm,\Tm\}$ where the four letters stand for the four bases \textit{adenine, cytosine, guanine and thymine.} The question of the identification of words $\Wm$ with unexpected frequencies is crucial in DNA sequence analysis, and in diagnostic issues in particular. In this example, we model the sequence $\Xm_1\Xm_2\dots\Xm_n$ with an homogeneous and stationary Markov chain of order $m$. The transition probability is given by the application $\theta$ defined on  $\mathcal A\times\mathcal A$, whereas the invariant probability measure is denoted by $\mu$. 
The aim is to compute the number of occurrences in the sequence of a given word $\Wm$ of size $h$ (with $h>m$) $\Wm=w_1w_2\cdots w_h$. Let $(\Zm_j)_{j\in\I}$ be the sequence defined by 
\begin{equation*}
\Zm_j=\car_{\{\Xm_j=w_j,\dots,\Xm_{j+h-1}=w_h\}},
\end{equation*}
where $\I=\{1,\dots,n-h+1\}$. Since the underlying Markov chain is homogeneous and stationary of invariant measure $\mu$, $\esp{\Zm_j}=\mu(\Wm)$ ($j\in\I$). The number of occurrences of the word $\Wm$ is then provided by the random variable 
\begin{equation*}
\mathfrak{T}(\Wm)=\sum_{j\in\I}\Zm_j,
\end{equation*}
whose asymptotic behaviour we want to analyse when $n$ goes to infinity and $h$ grows as $\log(n)$.
As explained in particular in Schbath \cite{Schbath1997}, the word $\Wm$ may appear in clumps. Indeed, if $\Wm$ has a periodic decomposition, its occurrences in the sequence can overlap. A $k$-clump is thus the occurrence of a concatenated word $\Cm$ composed of exactly $k$ overlapping occurrences of $\Wm$. For instance, if $\Wm=\A\Cm\Tm\A\A$, the sequence
\begin{equation*}
\Gm\underline{\A\Cm\Tm\A\A\Cm\Tm\A\A\A\Cm\Tm\A\A}\Tm\Gm\A\A\underline{\A\Cm\Tm\A\A}\Cm\Gm
\end{equation*}
has a $3$-clump at position $j=2$ and a $1$-clump at position $j=20$.
Especially when the word $\Wm$ can overlap, we must consider
$(\Ztm_j)_{j\in\I}$, the "declumped" sequence associated to $(\Zm_j)_{j\in\I}$, such that $\Ztm_j$ only counts occurrences that do not overlap the preceding one. Define for any $j\in\I$,
\begin{equation*}
\Ztm_j=\Zm_j(1-\Zm_{j-1})\cdots(1-\Zm_{j-h+1}).
\end{equation*}
\begin{remark}
In fact, as highlighted by Schbath (\cite{Schbath1997}, remark 2) it would be more rigorous from a practical point of view to consider the "observable" sequence $(\widehat\Zm_j)_{j\in\J}$ defined by $\widehat\Zm_1=\Zm_1$ and for any $i\in\{2,\dots,j-1\}$, $\widehat\Zm_j=\Zm_j\prod_{i=1}^j(1-\Zm_{i})$ and $\widehat\Zm_j=\Ztm_j$ otherwise, since $\Xm_0,\Xm_{-1},\dots,\Xm_{-h+2}$ may not be known. That being so, as the total variation distance between $\widehat{\mathfrak T}(\Wm)=\sum_{j\in\I} \widehat\Zm_j$ and $\widetilde{\mathfrak T}(\Wm)=\sum_{j\in\I} \Ztm_j$ is bounded by $2h\mu(\Wm)$, both distributions have the same asymptotic behaviour, so that the sequence $(\Ztm_j)_{j\in\I}$ can be used more conveniently.
\end{remark}
\noindent
Define for any $k\in\N$, the random variable $\overline {\mathfrak T}^{(k)}(\Wm)$ that gives the number of $k$-clumps, as well as for $(j,k)\in\I\times\N$, the random variable $\Zt_j^{(k)}$ that indicates if there is a $k$-clump at position $j$. In order to approximate $\mathfrak T(\Wm)$, write up to now $\mathfrak T(\Wm)=\sum_{k\in\N}k\overline {\mathfrak T}^{(k)}(\Wm)$ that can be well approximated (see for instance Barbour and Chryssaphinou \cite{BarbourChryssaphinou2001}, or Reinert and Schbath \cite{ReinertSchbath1998}) by the random variable 
\begin{equation*}
\overline{\mathfrak T}(\Wm):=\sum_{j\in\I}\sum_{k\in\N}k\Zt_j^{(k)},
\end{equation*}
and $\mathrm{dist}_{\mathrm{TV}}(\P_{\mathfrak T(\Wm)},\P_{\overline{\mathfrak T}(\Wm)})\pp 2h\mu(\Wm)$. Moreover, it appears (see Reinert and Schbath \cite{ReinertSchbath1998}) that for any $j\in\I$, $\Zt_j^{(k)}$ is a Bernoulli-distributed random variable of mean 
\begin{equation}\label{Proba_pk_DNA}
p_k=(1-\alpha)^2\alpha^{k-1}\mu(\Wm)
\end{equation}
where $\alpha$ can be written with respect to the principal periods of $\Wm$. The reader can find a explicit expression of $\alpha$ in the case of a first-order Markov chain in Schbath \cite{Schbath1997}, section 3.  
This last point suggests to approximate $\mathfrak T(\Wm)$ by $\overline{\mathfrak T}(\Wm)$ and, in order to get a marked binomial functional, by introducing the random variable
\begin{equation*}
\Hm=\sum_{j\in\I}\V_j\D\Nm_j,
\end{equation*}
where $(\V_j)_{j\in\N}$ is a sequence of independent and identically distributed random variables which the common geometric distribution $\mathbf V$ of parameter $(1-\alpha)$, where $\alpha$ appears in \eqref{Proba_pk_DNA}. In fact, $(1-\alpha)\alpha^{k-1}$ is the probability that the word $\Wm$ overlaps exactly $k$ times after having occurred at position $j$. The sequence $\V$ is also supposed to be independent of a Bernoulli process $(\D\Nm_j)_{j\in\I}$ of intensity  $(1-\alpha)\mu(\Wm)$ so that $\EuScript{P}\EuScript C(\lambda_0,\mathbf V)$ is exactly the P\'olya-Aeppli distribution of parameters $(\lambda_0,\alpha)$ where $\lambda_0=(n-h+1)(1-\alpha)\mu(\Wm)$. Some computations highlight that $\overline{\mathfrak T}(\Wm)$ and $\P_\Hm$ are identically distributed, so that $\mathrm{dist}_{\mathrm{TV}}(\P_{\overline{\mathfrak T}(\Wm)},\P_\Hm)=0$.
It remains to control $\mathrm{dist}_{\mathrm{TV}}\big(\P_{\Hm},\EuScript{P}\EuScript C(\lambda_0,\mathbf V)\big)$ by means of Theorem \ref{Approximation_compound_Poisson_th}. The following bound results from it.
\begin{proposition}\label{DNA_count_word_th}
Let $\lambda_0=(n-h+1)(1-\alpha)\mu(\Wm)$ and $\EuScript{P}\EuScript C(\lambda_0,\mathbf V)$ the random variable written as a Poisson of mean $\lambda_0$ compounded by the distribution $\mathbf V$. Then
\begin{align*}
\mathrm{dist}_{\mathrm{TV}}\big(\P_{\mathfrak T(\Wm)},\EuScript{P}\EuScript C(\lambda_0,\mathbf V)\big)
&\pp \mathrm{dist}_{\mathrm{TV}}(\P_{\mathfrak T(\Wm)},\P_{\overline{\mathfrak T}(\Wm)})+\mathrm{dist}_{\mathrm{TV}}(\P_{\overline{\mathfrak T}(\Wm)},\P_\Hm)\\
&\quad+\mathrm{dist}_{\mathrm{TV}}(\P_\Hm,\EuScript{P}\EuScript C(\lambda_0,\mathbf V))\\
&\pp 2h\mu(\Wm)+(n-h+1)\mathfrak d_{\EuScript P\EuScript C}\mu(\Wm)^2
\end{align*}
where $\mathfrak d_{\EuScript P\EuScript C}$ is defined by \eqref{Stein_class_bound_psi}.
\begin{remark}
The convergence occurs since the assumption on the order of the length $h$ (in $\log n$) entails that $n\mu(\Wm)=\mathrm O(1)$ (see Schbath \cite{Schbath1997}).
We retrieve the rate of convergence of this approximation in $\log n/n$ , without the additional assumptions made on the size of the "neighbourhood of dependence" (see Schbath \cite{Schbath1997}) or on the order of the magnitude of the maximal overlap (see Geske \textit{et al.} \cite{GeskeGodboleSchaffnerSkolnickWallstrom}). As noted in several works, in particular Robin and Schbath \cite{RobinSchbath}, the compound Poisson approximation is an excellent choice (especially with respect to the Gaussian approximation and to a lesser extent to the Poisson one) to describe the asymptotic behavior of a long and rare word in an "infinite" DNA sequence.
\end{remark}
\end{proposition}

\subsection{Portfolio optimisation in the trinomial model}

\noindent
We consider a simple financial market modelled by two assets i.e. a couple of $\R_+$-valued processes $(\A_t,\S_t)_{t\in\mathbf \T}$, defined on the same probability space $(\O,\F,\P)$ where $\F=(\f{t})_{t\in\T}$ is a filtration (generally that generated by the canonical process) and $\T=\Z_+\cap[0,T]$ ($T\in\N$) is called the trading interval. Denote also $\T^*=\T\setminus\{0\}$.
The riskless asset $(\A_t)_{t\in \T}$ is deterministic with initial value $\A_0=a_0$ and is defined for $r\in\R_+$ (generally smaller than $1$) by
\begin{equation}\label{risklessdisc}
\A_t=a_{0}(1+r)^t,
\end{equation}
whereas the stock price which models the risky asset, is the $\F$-adapted process $(\S_t)_{t\in\T}$ with (deterministic) initial value $\S_0=1$ and  such that for any $t\in\T^*$,
\begin{equation}\label{riskternary}
\Delta \S_t=\eta_t\,\S_{t-1}\,\Delta \Nm_t,
\end{equation}
where $\eta_t=b\car_{\{\Wm_t=1\}}+a\car_{\{\Wm_t=-1\}}$, $a$ and $b$ are real numbers such that $-1<a<r<b$ and $\{\Wm_t,\, t\in\T^*\}$ is a family of i.i.d. $\{-1,1\}$-valued random variables such that $\P(\Wm_t=1)=p$ ($p\in(0,1)$, $q=1-p$). The sequence of discounted prices $(\St_t)_{t\in\mathbf T}$ is defined by $\St_t=\A_t^{-1}\S_t$ ($t\in\mathbf T$).
\begin{remark}[Trinomial and ternary models: differences and equivalence in law]\label{Equi_CRR_ternary_rem}The price process defined in the ternary model has the same law as the one of a well-chosen trinomial model (for more details on classical trinomial model see for instance Delbaen \cite{Delbaen} or Runggaldier \cite{runggaldier2006}). As a reminder, the stock price $(\Tm_t)_{t\in\T}$ is defined in this latter model by $\Tm_0=1$ and verifies the recurrent relation
\begin{equation*}
\Tm_{t}=(1+b)\Tm_{t-1}\car_{\{\Xm_t=1\}}+(1+a)\Tm_{t-1}\car_{\{\Xm_t=-1\}}+\Tm_{t-1}\car_{\{\Xm_t=0\}},
\end{equation*}
where the process $(\Xm_t)_{t\in \T^*}$  is distributed according to the measure $\P$ such that
\begin{equation*}
\P(\Xm_t=1)=\bar p,\; \P(\Xm_t=-1)=\bar q\; \text{ and }\P(\Xm_t=0)=1-\bar p-\bar q,
\end{equation*}
and $(\bar p,\bar q)\in(0,1)^2$ such that $1-\bar p-\bar q\in(0,1)$. Let $\lambda\in(0,1)$, $\bar p=\lambda p$ and $\bar q=\lambda (1-p)$ such that $1-\bar p-\bar q=1-\lambda$. Then,
\begin{equation*}
\mathbf E\Big[s^{\frac{\S_t}{\S_{t-1}}}\Big]=\esp{s^{\eta_t\Delta \Nm_t+1}}
=s^{1+b}\,\bar p+s^{1+a}\,\bar q+s(1-\lambda)=\mathbf E\Big[s^{\frac{\Tm_t}{\Tm_{t-1}}}\Big],
\end{equation*}
and $\S_0=\Tm_0$. Thus the trinomial and the ternary models are equivalent in law. The introduction of the second one is motivated by the following remark. As explained in Halconruy's PhD dissertation (see \cite{Halconruy2020}, conclusion of chapter 4), it turned out to be impossible to derive a Karatzas-Ocone-type hedging formula for replicable claims in the trinomial model (underlying by a sequence of $\{-1,0,1\}$-valued independent variables)  by the Clark-Ocone formula stated in Decreusefond and Halconruy (see \cite{DecreusefondHalconruy}, Theorem 3.3).  Indeed, the $\f{k}$-measurability of the term $\Dm_k\esp{\Fm|\f{k}}$ appearing in this prevents from deriving the expected $\F$-predictable drift process. This observation was prone to replace the trinomial model by a ternary model, based on a jump process, and, as we will see, lends itself more easily to the statement of a hedging formula, directly derived from Clark's one \eqref{Clark_formula_eq}.
\end{remark}

\begin{remark}[Incompleteness of the ternary model]
As explained in Runggaldier \cite{runggaldier2006}, the trinomial tree model is an incomplete market; as expected, so does the ternary model. Indeed, the measure with respect to which the sequence of discounted prices is a $\F$-martingale, is not unique. Considering the process $(\S_t)_{t\in\T}$ defined by \eqref{riskternary} and that is identically distributed to the one of the trinomial model, we expect to reach the same incompleteness result.
\noindent
By writing  for any $t\in\id{1}{T}$,
\begin{equation*}
\Delta \St_t=\frac{\S_t-(1+r)\S_{t-1}}{(1+r)^t}=\frac{[b\car_{\{\Wm_t=1\}}+a\car_{\{\Wm_t=-1\}}]\,\,\Delta \Nm_t-r}{(1+r)^t}\times \S_{t-1}=(b\D\Zm_{(t,1)}+a\D\Zm_{(t,-1)})\St_{t-1},
\end{equation*}
it appears that the discounted price sequence is a $\F$-martingale within the condition
$\lambda(bp+aq)-r=0$. As expected, the system  
\begin{displaymath}
\left\{
\begin{array}{rcl}
\lambda(bp+aq)&=&r\\
p+q&=&1
\end{array}\right.
\end{displaymath}
admits infinitely many solutions $(\lambda,p,q)\in(0,1)^3$ such that any triplet $(\lambda,p,q)$ forms a convex $\mathscr M$ set (here a segment) characterized by its extremal points, i.e. the measures $\P^0=(1,(r-a)/(b-a),(b-r)/(b-a)) $ and $=\P^1=(r/b,1,0), $which are not equivalent to $\P$ but such that any convex combination
$\P^{\gamma}=\gamma \P^0+(1-\gamma)\P^1$ is. Any measure defined on $\O$ and with respect to which the sequence $\St$ is a $\F$-martingale is called a \textit{$\F$-martingale measure}.
\end{remark}

\noindent
The value of the portfolio at time $t\in \T$ is given by the random variable
\begin{equation*}
\V_t=\alpha_{t}\,\A_t+\vp_{t}\,\S_t,
\end{equation*}
where $(\alpha_t,\vp_t)_{t\in \T}$ is a couple of $\F$-predictable processes modelling respectively the amounts of riskless and risky assets held in the portfolio. Its discounted value at time is $\Vt_t:=\V_t/\A_t$.

\noindent
The aim of this subsection is to exhibit a hedging formula; this is, given a nonnegative $\f{T}$-measurable random variable  $\Fm$ (called \textit{claim}), to find an \textit{admissible} strategy $\psi=(\alpha,\varphi)$ that is \textit{self-financed} in the sense where for any $t\in \T\setminus\{T\}$,
\begin{equation}\label{Self_financing_cond}
\A_t\,(\alpha_{t+1}-\alpha_{t})+\S_t\,(\vp_{t+1}-\vp_{t})=0,
\end{equation}
and which corresponding portfolio value satisfies $
\V_0>0, \; \V_t\pg 0$ for all $t\in \T\setminus\{T\}$, and  $\V_{T}=\Fm$.
In an incomplete market, there is no systematic hedging formula, since all claims are not reachable; they have an \textit{intrisic risk}. Face to the impossibility to perform a perfect hedge in the general case, we can only hope to reduce the a priori risk to this minimal
component.  The question of hedging in an incomplete market has been widely investigated for years (see for instance Dalang \cite{dalang1990equivalent}, F\"ollmer and Schweizer \cite{FollmerSchweizer1991} in continuous time, Schweizer \cite{Schweizer1995} in discrete time). As explained in Remark \ref{Equi_CRR_ternary_rem} the ternary model is not complete; we choose to deal with the optimization problem in return:
\begin{equation}\label{Opt_hedging_pb_eq}
\underset{\psi\in\mathscr S}\min\,\esp{(\Fm-x-\Vt_T(\psi))^2},
\end{equation}
where the claim $\Fm$ and the initial capital $x\in\R_+^*$ are given, and  $\mathscr S$ is the set of $\F$-predictable admissible strategies.
\noindent
The \textit{mean-variance tradeoff} process $(\mathrm K_t)_{t\in\T}$ is defined by
\begin{equation*}
\mathrm K_t=\sum_{s=1}^t\frac{\big(\esp{\Delta\St_s\,|\,\f{s-1}}\big)^2}{\mathrm{var}[\Delta\St_s\,|\,\f{s-1}]}\; ; \;  t\in\T.
\end{equation*}
	
\noindent 
Introduce also the discrete analogue of the \textit{minimal martingale measure} (see Föllmer and Schweizer \cite{FollmerSchweizer1991}), i.e. the signed measure $\widehat\P$  defined on $(\O,\F)$ such that 
\begin{equation}\label{Mart_mesure_def}
\dfrac{\d\widehat{\P}}{\d\P}=\prod_{t=1}^{\T}\frac{1-\theta_t\D\St_t}{1-\theta_t\esp{\D\St_t|\f{t-1}}},
\end{equation}
where $(\theta_t)_{t\in\T^*}$ is the $\F$-predictable process such that $\theta_t=\esp{\D\St_t\,|\,\f{t-1}}/\esp{(\D\St_t)^2\,|\,\f{t-1}}$, for any $t\in\T^*$. Last, consider the Kunita-Watanabe decomposition of $\Fm$ (see Metivier \cite{Metivier1982} or Schweizer \cite{Schweizer1995}) i.e. the unique couple of processes $(\xi^{\Fm},\L^{\Fm})$ where $\xi^{\Fm}$ is a square-integrable admissible strategy  and $\L^{\Fm}$ is a $\F$-martingale, strongly orthogonal to $\S$, with null initial value and such that
 \begin{equation*}
 \Fm=\Fm_0+\sum_{t\in\T}\xi_t^{\Fm}\D\St_t+\L_T^{\Fm} \quad \P\text{-a.s.}
 \end{equation*}
Within previous notations, M. Schweizer gives in (\cite{Schweizer1992}, Proposition 4.3) an expression of the quadratic-loss minimizing strategy.
\begin{theorem}[Schweizer, 1992]\label{Schweizer_opt_th}
Provided $(\mathrm K_t)_{t\in\T}$ is deterministic, the solution of \eqref{Opt_hedging_pb_eq} is given by
\begin{equation}\label{Schweizer_opt_strat_eq}
 \varphi_t^*=\xi_t^{\Fm}+\frac{\esp{\Delta\St_t\,|\,\f{t-1}}}{\esp{(\Delta\St_t)^2\,|\,\f{t-1}}}(\widehat{\mathbf E}\big[\Fm|\f{t}\big]-x-\Vt_{t-1}(\varphi^*))
\end{equation}
where $\widehat{\mathbf E}$ denotes the expectation with respect to the measure $\widehat\P$ i.e., the minimal martingale measure defined by \eqref{Mart_mesure_def}. Moreover, the quota of the riskless asset $(\A_t)_{t\in\T}$ is given by  $\alpha_0=\esph{\Fm}/\S_0$ and for any $t\in\T^*$,
\begin{equation*}
\alpha_t=\alpha_{t-1}-(\varphi_t-\varphi_{t-1})\St_{t-1}.
\end{equation*}
\end{theorem}

\begin{remark}
If the contingent claim $\Fm$ is reachable, then $\varphi^*=\xi^{\Fm}$. The term $\xi^{\Fm}$ in \eqref{Schweizer_opt_strat_eq} can be interpreted as a pure hedging demand, whereas the second one can be viewed as a demand for mean-variance purposes (see Schweizer \cite{Schweizer1992}).
\end{remark}

\noindent 
These results are slot to our formalism to solve \eqref{Opt_hedging_pb_eq} in the ternary model.
\begin{lemma}\label{Mean_variance_deter_lem}
The mean-variance tradeoff process of the ternary model is deterministic.
\end{lemma}
\begin{proof}
For any $t\in\T$, 
\begin{equation*}
\frac{\big(\esp{\Delta\St_t\,|\,\f{t-1}}\big)^2}{\mathrm{var}[\Delta\St_t\,|\,\f{t-1}]}
=\frac{(\esp{\eta_t\D\Nm_t-r|\f{t-1}})^2}{\mathrm{var}[\eta_t\D\Nm_t-r|\f{t-1}]}=\frac{(\lambda(bp+aq)-r)^2}{\lambda p(1-\lambda p)b^2+a^2\lambda q(1-\lambda q)},
\end{equation*}
is a deterministic constant. Hence the result.
\end{proof}
\noindent
The family $\mathcal R$ is provided by Gram-Schmidt process \eqref{family_R_eq} such that
\begin{equation*}
\D\Rm_{(t,1)}=\D\Zm_{(t,1)} \quad \text{and} \quad \D\Rm_{(t,-1)}=\D\Zm_{(t,-1)}+\frac{\lambda^2pq}{\lambda p(1-\lambda p)}\D\Rm_{(t,1)}=\D\Zm_{(t,-1)}+\rho \D\Zm_{(t,1)},
\end{equation*}
where $\rho:=\lambda q/(1-\lambda p).$

\begin{lemma}[Kunita-Watanabe decomposition in the ternary model]\label{Kunita_lem} For any claim $\Fm\in\L^2(\P)$ there exist a square-integrable admissible strategy $\xi^{\Fm}$ and a $\F$-martingale $\L^{\Fm}$, strongly orthogonal to $\St$, with null intial value such that
\begin{equation*}
\Fm=\Fm_0+\sum_{t\in\T}\xi_t^{\Fm}\D\St_t+\L_T^{\Fm} \quad \P\text{-a.s.}
\end{equation*}
Moreover, for any $t\in\T^*$,
\begin{equation}\label{Kunita_ternary_eq}
\xi_t^{\Fm}=\frac{1}{\St_{t-1}}\Big(\displaystyle\sum_{k\in\E}w_{t,k}\widehat{\mathbf E}\big[\Dm_{(t,k)}\Fm|\f{t-1}\big]\Big)\quad \text{and}\quad \L_t^{\Fm}=\mathbf E\bigg[\Fm-\sum_{s\in\T}\xi_s^{\Fm}\D\St_s\,\Big|\,\f{t}\bigg]-\mathbf E\bigg[\Fm-\sum_{s\in\T}\xi_s^{\Fm}\D\St_s\bigg],
\end{equation}
where $\esp{\L_0^{\Fm}}=0$, the sequence $w=(w_{t,k})_{(t,k)\in\Xb}$ is defined by
\begin{equation*}
w_{t,1}=\frac{(b-a\rho)\kappa_1}{(b-a\rho)^2\kappa_1+a^2\kappa_{-1}},\quad w_{t,-1}=\frac{a\kappa_{-1}}{(b-a\rho)^2\kappa_1+a^2\kappa_{-1}}.
\end{equation*}
The minimal martingale measure $\widehat \P$, equivalent to $\P$ can be explicitly given by
\begin{equation}\label{Min_mart_mesure_bin_eq}
\dfrac{\d\widehat{\P}}{\d\P}=\prod_{t\in\T}\frac{1-\theta_t\D\St_t}{1-\theta_t\esp{\D\St_t|\f{t-1}}},
\end{equation}
with  
\begin{equation*}
\theta_t=\frac{\St_{t-1}(\lambda(bp+aq)-r)}{\St_{t-1}^2(\lambda^2(b^2p+a^2q)+r^2-2\lambda(bp+aq))}=\frac{\lambda(bp+aq)-r}{\St_{t-1}((b-a\rho)^2\kappa_1+a^2\kappa_{-1})}\;;\; t\in\T.
\end{equation*}
\end{lemma}
\begin{remark}
The expression \eqref{Kunita_ternary_eq} of $\xi^\Fm$ which is the replicating strategy when $\Fm$ is reachable, is not so dissimilar to that of the hedging strategy in the binomial model (see Privault \cite{Privault_stochastica}, proposition 1.14.4). 
\end{remark}

\begin{theorem}[Loss quadratic minimizing strategy in the ternary model]\label{Hedging_th}
Let $\widehat{\P}$ be the minimal martingale measure defined by \eqref{Min_mart_mesure_bin_eq} and let a claim $\Fm$.  The quadratic loss minimizing hedge $\varphi^*$ is given by
\begin{equation*}
 \varphi_t^*=\xi_t^{\Fm}+\theta_t\big(\widehat{\mathbf E}\big[\Fm|\f{t}\big]-x-\Vt_{t-1}(\varphi^*)\big),
\end{equation*}
where $\xi^\Fm\in\mathscr S$ is given by the Kunita-Watanabe decomposition.
\end{theorem}

\begin{proof}
Since the mean-variance process is deterministic by Lemma \ref{Mean_variance_deter_lem}, it suffices to incorporate the result of Lemma \ref{Kunita_lem} to Theorem \ref{Schweizer_opt_th}. The process $(\alpha_t)_{t\in\T}$ is defined by the self-financing condition \eqref{Self_financing_cond}.
\end{proof}

\section{Proofs}\label{Proofs_sec}
\subsection{Proofs of the section \ref{Binomial_marked_analysis_sec}}

\begin{proof}[Proof of Proposition \ref{Isometry_integral_prop}]
Let $u,v\in\mathcal P$; there exists $T\in\N$ such that $u$ and $v$ are of the form \eqref{simpleprocess}. For any $t\in\id{1}{T}$, \\
\begin{align*}
\mathbf E\Big[\Jt_1&(u\car_{[t,\infty)}\,;\mathcal R)\Jt_n(v\car_{[t,\infty)}\,;\mathcal R)\,\Big|\,\f{t-1}\Big]\\
&=\sum_{(s,k)\in\id{t}{T}\times\E}\sum_{(r,\l)\in\{t:T\}\times\E}\esp{\uk(\eta,(s,k))\vk(\eta,(r,\l))\D\Rm_{(s,k)}\D\Rm_{(r,\l)}\,\big|\,\f{t-1}}\\
&=\sum_{(s,k,\l)\in\id{t}{T}\times\E^2}\esp{\uk(\eta,(s,k))\vk(\eta,(s,\l))\esp{\D\Rm_{(s,k)}\D\Rm_{(s,\l)}\,|\,\f{s-1}}\,\big|\,\f{t-1}}\\
&\qquad+\sum_{(s,k)\in\{t:T\}\times\E}\sum_{\underset{r>s}{(r,\l)\in\id{t}{T}\times\E}}\esp{\uk(\eta,(s,k))\vk(\eta,(r,\l))\D\Rm_{(s,k)}\esp{\D\Rm_{(r,\l)}|\f{r-1}}\,\big|\,\f{t-1}}\\
&\qquad+\sum_{(r,\l)\in\id{t}{T}\times\E}\sum_{\underset{s>r}{(s,k)\in\id{t}{T}\times\E}}\esp{\uk(\eta,(s,k))\vk(\eta,(r,\l))\D\Rm_{(r,\l)}\esp{\D\Rm_{(s,k)}|\f{s-1}}\,\big|\,\f{t-1}}\\
&=\sum_{(s,k)\in\id{t}{T}\times\E}\kappa_k\esp{\uk(\eta,(s,k))\vk(\eta,(s,k))\,\big|\,\f{t-1}}=\esp{\textless u\car_{[t,\infty)}, v\car_{[t,\infty)}\textgreater_{\L^2(\Xb,\tilde\nu)}\,\big|\,\f{t-1}},
\end{align*}
where the second and third sums in the second equality vanish as $\esp{\D\Rm_{(t,k)}}=0$ for all $(t,k)\in\Xb_T$.
The extension of the stochastic integral to the set of square-integrable adapted processes comes from a Cauchy sequence argument. Define the sequence $(u^n)_{n\in\N}$ of simple predictable processes by
\begin{equation*}
u^n(\eta,(t,k))=u(\eta,(t,k))\car_{\{t\in\{1:n\}\}}\car_{\{|u(\eta,(t,k))|\pp n\}.}
\end{equation*}
Thus,  $(\Jt_1(u^n)\,;\mathcal R)_{n\in\N}$ is Cauchy and converges in $\L^2(\P)$. Let then
\begin{equation*}
\Jt_1(u\,;\mathcal R)=\underset{n\rightarrow \infty}\lim \Jt_1(u^n\,;\mathcal R).
\end{equation*} 
The limit is independent of the approximating sequence by applying the isometry property \eqref{isometry_proc_eq}  with $t=1$. Hence the result.
\end{proof}

\begin{proof}[Proof of Proposition \ref{Isometry_mul_prop}]
Assume with no loss of generality that $m>n$. Let $(f_n,g_m)\in\mathcal \L^2(\Xb,\tilde\nu)^{\circ n}\times\L^2(\Xb,\tilde\nu)^{\circ m}$. For any $(\mathbf t_n,\mathbf s_m)\in\N^{n,<}\times\N^{m,<}$, there exists $j_0\in\N$ such that $s_{j_0}\in\mathbf s_m\setminus\mathbf t_n$. By independence of the random variable $\D\Rm_{(s_{j_0},k_{j_0})}$ with respect to the  $\sigma$-algebra $\g{s_{j_0}}=\sigma\big\{\sum_{(s,k)}\eta(s,k),\, s\neq s_{j_0},\,k\in\E\big\}$, 
\begin{align*}
\esp{\Jt_n(f_n\,;\mathcal R)\,\Jt_m(g_m\,;\mathcal R)}
&=n!\,m!\sum_{(\mathbf t_n,\mathbf k_n)\in \Xb^{n,<}}\sum_{(\mathbf s_m,\mathbf l_m)\in \Xb^{m,<}}\,f_n(\mathbf t_n,\mathbf k_n)\,g_m(\mathbf s_m,\mathbf l_m)\\
&\qquad\qquad\qquad\qquad\qquad\qquad\qquad\qquad\times\esp{\prod_{i=1}^{n}\prod_{j=1}^m\D\Rm_{(t_i,k_i)}\D\Rm_{(s_j,\l_j)}}\\
&=n!\,m!\sum_{(\mathbf t_n,\mathbf k_n)\in \Xb^{n,<}}\sum_{(\mathbf s_m,\mathbf l_m)\in \Xb^{m,<}}\,f_n(\mathbf t_n,\mathbf k_n)\,g_m(\mathbf s_m,\mathbf l_m)\\
&\qquad\qquad\qquad\times\esp{\D\Rm_{(s_{j_0},k_{j_0})}}
\,\prod_{i=1}^{n}\prod_{\underset{j\neq j_0}{j=1}}^{m}\esp{\D\Rm_{(t_i,k_i)}\D\Rm_{(s_j,\l_j)}}=0,
\end{align*}
and for $m=n$,
\begin{align*}
\esp{\Jt_n(f_n\,;\mathcal R)\Jt_n(g_n\,;\mathcal R)}
&=(n!)^2\sum_{(\mathbf t_n,\mathbf k_n)\in \Xb^{n,<}}\sum_{(\mathbf s_n,\mathbf l_n)\in \Xb^{n,<}}\,f_n(\mathbf t_n,\mathbf k_n)\,g_n(\mathbf s_n,\mathbf l_n)\\
&\qquad\qquad\qquad\qquad\qquad\qquad\qquad\qquad\times\esp{\prod_{i,j=1}^{n}\D\Rm_{(t_i,k_i)}\D\Rm_{(s_j,\l_j)}}\\
&=(n!)^2\sum_{\underset{\mathbf l_n\in\E^{n,<}}{(\mathbf t_n,\mathbf k_n)\in\Xb^{n,<}}}\,f_n(\mathbf t_n,\mathbf k_n)\,g_n(\mathbf t_n,\mathbf l_n)\esp{\prod_{i,j=1}^{n}\D\Rm_{(t_i,k_i)}\D\Rm_{(t_i,\l_i)}}\\
&=n!\,\textless f_n,g_n\textgreater_{\L^2(\Xb,\tilde\nu)^{\circ n}},
\end{align*}
since $\esp{\D\Rm_{(t,k)}\,\D\Rm_{(t,\l)}}=\kappa_k\car_{\{k\}}(\l)$. Besides, for any $f_n\in\L^2(\Xb,\tilde\nu)^{\circ n}$,
\begin{align*}
\Jt_n(f_n\,;\mathcal R)
&=n\sum_{(t,k)\in \Xb}\Jt_{n-1\,}(\pi^n_{(t,k)}f_n;\mathcal R)\,\D\Rm_{(t,k)}\\
&=n!\sum_{(t,k)\in\Xb}\sum_{(\mathbf t_{n-1},\mathbf k_{n-1})\in \Xb^{n,<}}\,f_n\big((\mathbf t_{n-1},\mathbf k_{n-1}),(t,k)\big)\,\D\Rm_{(t,k)}\prod_{i=1}^{n-1}\D\Rm_{(t_i,k_i)}\\
&=n!\sum_{(\mathbf t_n,\mathbf k_n)\in \Xb^{n,<}}\,f_n(\mathbf t_n,\mathbf k_n)\,\prod_{i=1}^{n}\D\Rm_{(t_i,k_i)},
\end{align*}
that completes the proof.
\end{proof}

\begin{proof}[Proof of Lemma \ref{Int_tensor_lem}]
Let $g\in\L^2(\Xb_{T'})$ and $f_n\in\L^2(\Xb_{T''})^{\circ n}$ for some $T',T''\in\N$, the definition of the symmetric tensor product implies $g\circ f_n\in\L^2(\Xb_T)^{\circ n+1}$ where $T:=\max(T',T'')$, and
\begin{align*}
\Jt_{n+1}(g\circ f_n\,;\mathcal R)
&=n!\sum_{i=1}^{n+1}\sum_{(\mathbf t_{n+1},\mathbf k_{n+1})\in\Xb_T^{n+1,<}}g(t_i,k_{i})f_n^{\neg i}(\mathbf t_{n+1},\mathbf k_{n+1})\prod_{i=1}^{n+1}\D\Rm_{(t_i,k_i)}\\
&=n!\sum_{i=1}^{n}\sum_{(\mathbf t_{n+1},\mathbf k_{n+1})\in\Xb_T^{n+1,<}}g(t_i,k_{i})f_n^{\neg i}(\mathbf t_{n+1},\mathbf k_{n+1})\prod_{i=1}^{n+1}\D\Rm_{(t_i,k_i)}\\
&\qquad+ n!\sum_{(t,k)\in\Xb_T}\sum_{\underset{(t,k)\notin(\mathbf t_{n},\mathbf k_{n})}{(\mathbf t_{n},\mathbf k_{n})\in\Xb_T^{n,<}:}}g(t,k)f_n^{\neg i}(\mathbf t_{n},\mathbf k_{n})\,\D\Rm_{(t,k)}\prod_{i=1}^{n}\D\Rm_{(t_i,k_i)}\\
&=n!\sum_{(t,k)\in\Xb_T}\sum_{i=1}^{n}\sum_{\underset{(t,k)\notin(\mathbf t_{n},\mathbf k_{n})}{(\mathbf t_{n},\mathbf k_{n})\in\Xb_T^{n,<}:}}g(t_i,k_{i})f_n^{\neg i}((\mathbf t_{n},\mathbf k_{n}),(s,k))\D\Rm_{(s,k)}\prod_{i=1}^{n}\D\Rm_{(t_i,k_i)}\\
&\qquad+\sum_{(t,k)\in\Xb_T}g(t,k)\Jt_n(f_n\car_{\id{1}{t-1}^n}\,;\mathcal R)\D\Rm_{(t,k)}\\
&=n\sum_{(t,k)\in\Xb_T}\Jt_n\big(f_n(\star,(t,k))\circ g(\cdot)\car_{\id{1}{t-1}^n}(\star,\cdot)\,;\mathcal R\big)\D\Rm_{(t,k)}\\
&\qquad+ \sum_{(t,k)\in\Xb_T}g(t,k)\Jt_n(f_n\car_{\id{1}{t-1}^n}\,;\mathcal R)\D\Rm_{(t,k)}.
\end{align*}
The result is then extended to $g\in\L^2(\Xb)$ and $f_n\in\L^2(\Xb)^{\circ n}$ by a standard Cauchy argument.
\end{proof}

\begin{proof}[Proof of Lemma \ref{Cond_int_lem}]
Let $T\in\N$ and $f_n\in\L^2(\Xb_T)^{\circ n}$. For any $t\in\N$ such that $t<T$,
\begin{align*}
\esp{\Jt_n(f_n\,;\mathcal R)\,|\,\f{t}}
&=n!\,\sum_{(\mathbf t_n,\mathbf k_n)\in (\Xb_T)^{n,<}}\,f_n(\mathbf t_n,\mathbf k_n)\,\esp{\prod_{i=1}^{n}\D\Rm_{(t_i,k_i)}\,\bigg|\,\f{t}}\\
&=n!\,\sum_{(\mathbf t_n,\mathbf k_n)\in (\Xb_t)^{n,<}}\,f_n(\mathbf t_n,\mathbf k_n)\,\esp{\prod_{i=1}^{n}\D\Rm_{(t_i,k_i)}\,\bigg|\,\f{t}}=\Jt_n\big(f_n\car_{\id{1}{t}}\,;\mathcal R\big),
\end{align*}
since the independence of the centered variables $\{\D\Rm_{(t_i,k_i)},\; (t_i,k_i)\in\Xb,\,i\in\id{1}{n} \}$ implies that $\mathbf E\big[\prod_{i=1}^{n}\D\Rm_{(t_i,k_i)}\big|\f{t}\big]=0$ if there exists $i_0\in\id{1}{T}$ such that $t_{i_0}>t$. The result is extended to $\L^2(\Xb)^{\circ n}$ by a limit procedure.
\end{proof}

\begin{proof}[Proof of Lemma \ref{chaos_L0_lem}]
It suffices to note that $\mathcal H_s\cap \mathcal L^0(\P,\f{t})$ for some $(s,t)\in\N$, $s\pp t$, is generated by the orthogonal family
\begin{equation}\label{orthbasis}
\{1\}\cup\left\{\prod_{i=1}^{s}\,\D\Rm_{(t_i,k_i)}, \; 1\pp t_1<\cdots<t_s\pp t, \, (k_1,\dots,k_s)\in \E^s \right\}.
\end{equation}
Indeed any element of $\overline{\mathcal R}_t=\mathrm{Span}\{\D\Rm_{(s,k)},\, (s,k)\in\Xb_t\}$ can be expressed in terms of multiple integrals as
\begin{equation*}
\prod_{i=1}^{s}\,\D\Rm_{(t_i,k_i)}=\Jt_s\Big(\car_{\{(t_1,k_1),\dots,(t_s,k_s)\}}^{<}\car_{\id{0}{t}^s}\Big).
\end{equation*}
We conclude by noting that the dimensions of $\overline{\mathcal R}_t$ and $\mathcal L^0(\P,\f{t})$ in \eqref{chaos_L0_lem_eq} are both equal to
\begin{equation*}
1+\sum_{s=1}^t\,|\E|^s\times\binom{t}{s}=(1+|\E|)^t.
\end{equation*}
The proof is thus complete.
\end{proof}

\begin{proof}[Proof of Proposition \ref{Chaos_Z_prop}]
Let, for notation purposes, $\mathfrak m_{k^i,k^j}^{-1}$ designate the $(i,j)$-th entry of matrix $\mathfrak M^{-1}$, that is the inverse of matrix $\mathfrak M$ defined by \eqref{matrix_basis_eq}.
It suffices to state it for any random variable $\Fm\in\cyl$. Let $\E=\{k^1,\dots,k^{\overline{\mathfrak m}}\}$. By Theorem \eqref{Chaos_R_th}, The chaotic decomposition of $\Fm$ reads
\begin{equation*}
\Fm=\esp{\Fm}+\sum_{\mathbf t_n }\sum_{i_1=1}^{\overline{\mathfrak m}}\cdots\sum_{i_n=1}^{\overline{\mathfrak m}}f_n((t_1,k_1^{i_1}),\dots,(t_n,k_n^{i_n}))\prod_{j=1}^n\D\Rm_{(t_j,k_j^{i_j})}.
\end{equation*}
Since $\D\Rm_{(t_j,k_j^{\l})}=\sum_{p=1}^\l \mathfrak m_{k^\l k^p}^{-1}\,\D\Zm_{(t_j,k_j^{p})}$, we get 
\begin{align*}
\Fm-\esp{\Fm}
&=
\sum_{\mathbf t_n }\sum_{i_1=1}^{\overline{\mathfrak m}}\cdots\sum_{i_n=1}^{\overline{\mathfrak m}}f_n((t_1,k_1^{i_1}),\dots,(t_n,k_n^{i_n}))\prod_{j=1}^n\Big(\sum_{p=1}^{i_j} \mathfrak m_{k^{i_j}k^p}^{-1}\,\D\Zm_{(t_j,k_j^{p})}\Big)\\
&=
\sum_{\mathbf t_n }\sum_{i_1=1}^{\overline{\mathfrak m}}\cdots\sum_{i_n=1}^{\overline{\mathfrak m}}f_n((t_1,k_1^{i_1}),\dots,(t_n,k_n^{i_n}))\Big(\sum_{p_1=1}^{i_1}\cdots\sum_{p_n=1}^{i_n} \prod_{j=1}^n\mathfrak m_{k^{i_j}k^{p_j}}^{-1}\D\Zm_{(t_j,k_j^{p_j})}\Big)\\
&=
\sum_{\mathbf t_n }\sum_{p_1=1}^{\overline{\mathfrak m}}\cdots\sum_{p_n=1}^{\overline{\mathfrak m}}\sum_{i_1=p_1}^{\overline{\mathfrak m}}\cdots\sum_{i_n=p_n}^{\overline{\mathfrak m}} f_n((t_1,k_1^{i_1}),\dots,(t_n,k_n^{i_n}))\prod_{j=1}^n\Big(\mathfrak m_{k^{i_j}k^{p_j}}^{-1}\D\Zm_{(t_j,k_j^{p_j})}\Big)\\
&=
\sum_{\mathbf t_n }\sum_{p_1=1}^{\overline{\mathfrak m}}\cdots\sum_{p_n=1}^{\overline{\mathfrak m}}\Big(\sum_{i_1=p_1}^{\overline{\mathfrak m}}\cdots\sum_{i_n=p_n}^{\overline{\mathfrak m}}\prod_{j=1}^n\mathfrak m_{k^{i_j}k^{p_j}}^{-1} f_n((t_1,k_1^{i_1}),\dots,(t_n,k_n^{i_n}))\Big)\prod_{j=1}^n\D\Zm_{(t_j,k_j^{p_j})},
\end{align*}
where we summed over the set of $\{\mathbf t_n\in\Xb_T^{n,<}\,:\,\mathbf t_n=(t_1,\cdots,t_n)\}$. The result is extended to $\L^2(\P)$ by density of $\cyl$.
\end{proof}	

\begin{proof}[Proof of Proposition \ref{Doleans_prop}]
For any $T\in\N,t\in \id{1}{T}$ define
\begin{equation*}
\zeta_t^{T}=1+\sum_{n=1}^{T}\frac{1}{n!}\Jt_{n}(h^{\otimes n}\car_{\id{1}{t}^{n}}).
\end{equation*}
where we assume with no loss of generality that $\esp{\zeta_t^{T}}=1$.
Consider $T$ large enough such that $T>t$; then
\begin{align*}
1&+\sum_{s=1}^{t}\sum_{k\in \E}h(s,k)\zeta_{s-1}^{T}\D\Rm_{(s,k)}\\
&=1+\sum_{s=1}^{t}\sum_{k\in \E}h(s,k)\bigg(1+\sum_{n=1}^{T}\frac{1}{n!}\Jt_n\big(h^{\otimes n}\car_{\id{1}{s-1}^n}\big)\bigg)\D\Rm_{(s,k)}\\
&=1+\sum_{s=1}^{t}\sum_{k\in \E}h(s,k)\D\Rm_{(s,k)}+\sum_{n=1}^{T}\frac{1}{n!}\Jt_{n+1}\big(h^{\otimes n+1}\car_{\id{1}{t}^{n+1}}\big)\\
&\qquad-\sum_{n=1}^{T}\sum_{k\in \E}\frac{n}{n!}\sum_{s=1}^{t}\Jt_n\big(h^{\otimes n}(\star,(s,k))\circ h(\cdot)\car_{\id{1}{s-1}^n}(\star,\cdot)\big)\D\Rm_{(s,k)}\\
&=1+\Jt_1(h\car_{\id{1}{t}})+\sum_{n=1}^{T}\frac{1}{n!}\Jt_{n+1}\big(h^{\otimes n+1}\car_{\id{1}{t}^{n+1}}\big)-\sum_{n=1}^{T}\sum_{k\in \E}\frac{n}{n!}\Jt_{n}\big(h^{\otimes n+1}(\star,(s,k))\car_{\id{1}{t}^{n+1}}(\star)\big)\D\Rm_{(s,k)}\\
&=1+\Jt_1(h\car_{\id{1}{t}})+\sum_{n=1}^{T}\frac{1}{n!}\Jt_{n+1}\big(h^{\otimes n+1}\car_{\id{1}{t}^{n+1}}\big)-\sum_{n=1}^{T}\frac{n}{(n+1)!}\Jt_{n+1}(h^{\otimes n+1}\car_{\id{1}{t}^{n+1}})\\
&=1+\Jt_1(h\car_{\id{1}{t}})+\sum_{n=2}^{T+1}\frac{1}{n!}\Jt_{n}\big(h^{\otimes n}\car_{\id{1}{t}^{n}}\big)=\zeta_t^{T+1},
\end{align*}
where we used Lemma \ref{Int_tensor_lem} in the second line and the definition of the multiple integral \eqref{Jn_mult_int_rec_eq} in the penultimate one.
\noindent
Since by the very definition of Doléans exponential \eqref{Doleans_def_eq}, for all $t\in\N$ $\zeta_t^{T}$ tends to $\xi_t(h)$ almost surely when $T$ goes ton infinity, we get
\begin{equation*}
\xi_t(h)=1+\sum_{s=1}^{t}\Big(\sum_{k\in \E}h(s,k)\D\Rm_{(s,k)}\Big)\xi_{s-1}(h).
\end{equation*}
Besides, the sequence $(\zeta_t)_{t\in\N}$ satisfies the equation in differences
\begin{equation*}
\xi_t(h)-\xi_{t-1}(h)=\xi_{t-1}(h)\sum_{k\in\E}g(t,k)\big(\car_{(t,k)}-\lambda\Q(\{k\})\big)
\end{equation*}
where $\Jt_1(h)=\J_1(g\,;\mathcal Z)$.
On the other hand, provided the product converges, define the sequence of exponential products  $(\xi_t^{\mathcal Z}(g))_{t\in\N}$, that stand for the Doléan exponentials with respect to the family $\mathcal Z$, by
\begin{equation*}
\xi_t(h)=\xi_t^{\mathcal Z}(g)=\prod_{t\in\N}\Big(1+\sum_{k\in\E}g(t,k)\big(\car_{(t,k)}-\lambda\Q(\{k\})\big)\Big)
\end{equation*}
and so that for all $t\in\N$,
\begin{equation*}
\xi_t^{\mathcal Z}(g)=1+\sum_{s=1}^{t}\Big(\sum_{k\in \E}h(s,k)\D\Rm_{(s,k)}\Big)\xi_{s-1}^{\mathcal Z}(g).
\end{equation*}
By uniqueness of the decomposition, provided the series and product converge, $\xi_t^{\mathcal Z}(g)=\xi_t(h)$ for any $t\in\N$; that leads to the conclusion.
\end{proof}

\subsection{Proofs of Section 3}
\subsubsection{Proofs of Subsection 3.1}

\begin{proof}[Proof of Lemma \ref{Mecke_formula_th}]
Let $\mathbf K$ be a probability measure on $\N$, $\V_1,\V_2,\dots$ independent random elements in $\E$ with distribution $\Q$, and $\mathrm K$ a random variable with distribution $\mathbf K$ supposed to be independent of $(\V_n,\, n\in\N)$. Recall that 
\begin{equation}\label{Mixing_bin_def_eq}
\varpi=\sum_{j=1}^{\mathrm K}\delta_{\V_j}
\end{equation}
is called a mixed binomial process with mixing distribution $\mathbf K$ and sampling distribution $\Q$. Let $\eta\in\widehat{\mathfrak N_{\Xb}}$; there exists $T\in\N$ such that $\eta$ is a marked binomial process on $\Xb_T$. By its very definition any marked binomial process on $\Xb_T$ of intensity measure $\nu$ is a mixed binomial process with mixing distribution $\mathcal B\mathrm{in}(T,\lambda)$ and sampling distribution $\Q$. Moreover, for any $n\in\id{1}{T}$, $\varpi_{|\Km=n}$ is a binomial process of intensity measure $n\Q$. Then, as a special case of the Georgii-Nguyen-Zessin formula (see \cite{DaleyVere-Jones2007}, Proposition 15.5.II with $\mathbf x=\varpi_{|\Km=n}$ and $\rho=n\nu$), for $\uk$ measurable application from $\mathfrak N_\Xb\times\Xb$ into $[0,+\infty]$,
\begin{align*}
\mathbf E\bigg[\sum_{(t,k)\in\eta}\uk\big(\eta,(t,k)\big)\bigg]
&=\sum_{n=1}^{T}\esp{\sum_{k\in\E}\esp{\uk(\varpi_{|\mathrm K},k)\,\varpi(k)\,\Big|\,\mathrm K=n}}\\
&=\sum_{n=1}^{T}n\sum_{k\in\E}\esp{\uk(\varpi_{|\mathrm K-1}+\delta_{k},k)}\Q(\{k\})\\
&=\sum_{n=1}^{T}\sum_{j=1}^n\sum_{k\in\E}\mathbf E\bigg[\big(\sum_{i\neq j}\delta_{\V_i}+\delta_k,\V_j\big)\bigg]\bigg|_{\V_j=k}\\
&=\esp{\int_{\Xb_T}\uk\big(\pi_{t}(\eta)+\delta_{(t,k)},(t,k)\big)\d\nu(t,k)},
\end{align*}
where we have used the mixed binomial representation of $\eta$ in the second line and in the last one. By replacing where necessary $\eta$ by $\eta-\delta_{(t,k)}$, we can state 
\begin{equation*}
\mathbf E\bigg[\sum_{(t,k)\in\eta}\uk\big(\eta-\delta_{(t,k)},(t,k)\big)\bigg]=\esp{\int_{\Xb_T}\uk\big(\pi_t(\eta),(t,k)\big)\d\nu(t,k)}
\end{equation*}
Hence the result. 
\end{proof}

\begin{proof}[Proof of Proposition \ref{IPP_L1_prop}]
Under the hypotheses of the proposition, by noting that $\Dm_{(t,k)}^+\Fm-\overline{\Dm}_t\Fm=\fk(\pi_t(\eta)+\delta_{(t,k)})-\fk(\eta)$, we have
\begin{align*}
\mathbf E\Big[\int_\Xb&\Big(\Dm_{(t,k)}^+\Fm-\overline{\Dm}_t\Fm\Big)\,\d\nu(t,k)\Big]\\
&=\mathbf E\Big[\int_\Xb[\fk(\pi_t(\eta)+\delta_{(t,k)})-\fk(\eta)]\,\uk(\eta,(t,k))\,\d\nu(t,k)\Big]\\
&=\mathbf E\Big[\int_\Xb\fk(\pi_t(\eta)+\delta_{(t,k)})\,\uk(\eta,(t,k))\,\d\nu(t,k)\Big]-\mathbf E\Big[\int_\Xb[\fk(\eta)\,\uk(\eta,(t,k))\,\d\nu(t,k)\Big]\\
&=\mathbf E\Big[\int_\Xb\fk(\eta)\,\uk(\pi_t(\eta),(t,k))\,\d\eta(t,k)\Big]-\mathbf E\Big[\int_\Xb[\fk(\eta)\,\uk(\eta,(t,k))\,\d\nu(t,k)\Big]\\
&=\mathbf E\Big[\int_\Xb\fk(\eta)\,\uk(\eta,(t,k))\,\d\eta(t,k)\Big]-\mathbf E\Big[\int_\Xb[\fk(\eta)\,\uk(\eta,(t,k))\,\d\nu(t,k)\Big]=\mathbf E\big[\Fm\widetilde{\delta}(u)\big],
\end{align*}
where we have used the Mecke formula in the third line and that $u$ is $\F$-predictable to get the last one. Hence the result. 
\end{proof}

\subsubsection{Proofs of Subsection 3.2}

\begin{proof}[Proof of Theorem \ref{IPP_S_prop}] The proof is identical to that of Theorem 1.8.2 in Privault \cite{Privault_stochastica}. Let $\Fm=\Jt_n(f_n)$ for some $f_n\in\L^2(\Xb)^{\circ n}$ and $u_{\cdot}=\Jt_m(g_{m+1}(\star,(\cdot)))\in\mathcal U$ for some $g_{m+1}\in\L^2(\Xb)^{\circ m}\otimes \L^2(\Xb)$. Then,  
\begin{align*}
&\mathbf E\bigg[\textless\Dm_{\cdot}\Jt_n(f_n),\Jt_m(g_{m+1}(\star,\cdot))\textgreater_{\L^2(\Xb,\tilde\nu)}\bigg]\\
&=n\esp{\textless\Jt_{n-1}(f_n(\star,\cdot))\car_{\Delta_n(\star,\cdot)},\Jt_m(g_{m+1}(\star,\cdot))\textgreater_{\L^2(\Xb,\tilde\nu)}}\\
&=n!\car_{\{n-1\}}(m)\int_\Xb\esp{\Jt_{n-1}(f_n(\star,(t,k)))\car_{\Delta_n(\star,(t,k))},\Jt_m(g_{m+1}(\star,(t,k)))}\,\d\tilde\nu(t,k)\\
&=n!\car_{\{n\}}(m+1)\textless\car_{\Delta_n(\star,(t,k))}f_n(\star,(t,k)),\tilde g_{m+1}(\star,(t,k))\textgreater_{\L^2(\Xb,\tilde\nu)}\\
&=\esp{\Jt_n(f_n)\Jt_m(\tilde g_{m+1})}=\esp{\Fm\,\delta u}.
\end{align*}
Hence the result.
\end{proof}

\begin{proof}[Proof of Corollary \ref{Closability_corollary}]
Let $(\Fm_n)_{n\in\Z_+}$ be a sequence of random variables defined on
$\cyl$ such that $\Fm_n$ converges to 0 in $\L^2(\P)$ and the sequence
$(\Dm\Fm_n)_{n\in\Z_+}$ converges to $\Lambda$ in $\L^2(\P\otimes\tilde\nu)$. Let $u$ be a simple process.
From the integration by parts formula (\ref{IPP_S_eq}),
\begin{align*}
\mathbf E\bigg[\sum_{(t,k)\in\Xb}\Dm_{(t,k)}\Fm_n\ u_{(t,k)}\bigg]
&=\mathbf E\bigg[\Fm_n\sum_{(t,k)\in\Xb}u_{(t,k)}\D\Rm_{(t,k)}\bigg],
\end{align*}
where $\sum_{(t,k)\in\Xb}u_{(t,k)}\D\Rm_{(t,k)}\in \L^2(\P)$. Indeed, the process $(\D\Rm_{(t,k)}u_{(t,k)})_{(t,k)\in\Xb_T}$ belongs to $\L^2(\O\times\Xb,\P\otimes\tilde\nu)$ since, by the Cauchy-Schwarz inequality,
\begin{equation*}
\mathbf E\bigg[\sum_{(t,k)\in \Xb_T}\big|u_{(t,k)}\D\Rm_{(t,k)}\big|^2\bigg]\pp \sum_{(t,k)\in \Xb_T}\kappa_k\mathbf E\bigg[u_{(t,k)}^2\delta_{(t,k)}\bigg]<\infty.
\end{equation*}
 Then,
\begin{equation*}
\langle \Lambda, u\rangle_{\L^2(\P\otimes\tilde\nu)}
=\underset{n\rightarrow \infty}\lim\mathbf E\bigg[\Fm_n\sum_{(t,k)\in\Xb}u_{(t,k)}\D\Rm_{(t,k)}\bigg]=0,
\end{equation*}
for  simple process $u$. It follows that $\Lambda=0$ and then the operator $\Dm$
is closable from $\L^2(\P)$ to $\L^2(\O\times\Xb,\P\otimes\tilde\nu)$. By equivalence of the norms $\|\cdot\|_{\L^2(\Xb,\tilde\nu)}$ and $\|\cdot\|_{\L^2(\Xb,\nu)}$, this result can be extended to $\L^2(\O\times\mathscr \Xb,\P\otimes\nu)$.
\end{proof}

\subsubsection{Proofs of Subsection 3.3}

\begin{proof}[Proof of Proposition \ref{Grad_difference_prop} ]
The application of $\Dm^+$ to $\Fm=\Jt_n(f_n)\in\cyl$, and $(t,k)\in\Xb_T$, gives
\begin{align*}
\Dm_{(t,k)}^+\Jt_n(f_n)
&=\,n!\sum_{(\mathbf t_n,\mathbf k_n)\in (\Xb_T)^{n,<}}\,f_n\big((t_1,k_1),\dots,(t_n,k_n)\big)\,\prod_{i=1}^{n}\Dm_{(t,k)}^+\, \D\Rm_{(t_i,k_i)}\\
&=n!\sum_{(\mathbf t_n,\mathbf k_{n}^{\neg k})\in (\Xb_T^{\neg t})^{n-1,<}}\,f_n\big((t_1,k_1),\dots,(t,k),\dots,(t_{n},k_{n})\big)\,\prod_{\underset{t_i\neq t}{i=1}}^{n}\D\Rm_{(t_i,k_i)}\\
&=n!\sum_{(\mathbf t_{n-1},\mathbf k_{n-1})\in(\Xb_T^{\neg t})^{n,<}}\,f_n\big((\mathbf t_{n-1},\mathbf k_{n-1}),(t,k)\big)\,\prod_{\underset{t_i\neq t}{i=1}}^{n}\D\Rm_{(t_i,k_i)}\\
&=n\Jt_{n-1}\big(f_n(\star,(t,k))\car_{\Delta_n^{<}}\big)=\Dm_{(t,k)}\Jt_n(f_n),
\end{align*}
where $\Xb_T^{\neg t}=\Xb_T\setminus\cup \{(t,k),\,k\in\E\}$.
Thus, for any $\Fm\in\cyl$, $\Dm_{(t,k)}\Fm=\big[\Fm(\pi_t(\eta)+\delta_{(t,k)})-\Fm(\pi_t(\eta))\big]$. The result is then extended to $\DD$ by a density argument relying on the closability of $\Dm$ (see Corollary \ref{Closability_corollary}).
\end{proof}

\begin{proof}[Proof of Lemma \ref{Mult_D_esp_lem}]
It suffices to state the result for $\Fm=\xi(h)$, with $h\in\L^2(\Xb)$. By \eqref{Doleans_def_eq}, 
\begin{equation*} \xi(h)=\esp{\xi(h)}+\sum_{m\in\N}\sum_{\underset{|\J|=m}{\J\subset\N}}\prod_{i\in\J}h(t_i,k_i)\D\Rm_{(t_i,k_i)}.
\end{equation*}
Follows from the definition of $\Dm^{(n)}$ \eqref{Mult_D_eq} that for any $(\mathbf t_n,\mathbf k_n)\in\Xb^n$ and any set $\big\{(t_i,k_i),\,i\in\J,|\J|=m\big\}$ with $m>n$, there exists $i_0\in\J$ such that $(t_{i_0},k_{i_0})\notin\{(t_i,k_i),\,i\in\id{1}{n}\}$. Then, by independence of the $\Delta\Rm_{(t,k)}$, 
\begin{align*}
\esp{\Dm_{(\mathbf t_n,\mathbf k_n)}^{(n)}\Big(\prod_{i\in\J}h(t_i,k_i)\D\Rm_{(t_i,k_i)}\Big)}=\esp{\D\Rm_{(t_{i_0},k_{i_0})}}\mathbf E\bigg[\Dm_{(\mathbf t_n,\mathbf k_n)}^{(n)}\Big(\prod_{i\in\J\setminus\{i_0\}}h(t_i,k_i)\D\Rm_{(t_i,k_i)}\Big)\bigg].
\end{align*}
For any $(t,k)\in\Xb$ let $\rk_{(t,k)}$ be the representative of  $\D\Rm_{(t,k)}$. 
With a similar argument we can prove the same result for any set $\big\{(t_i,k_i),\,i\in\J,|\J|=n\big\}$ different of $(\mathbf t_n,\mathbf k_n)$ so that
\begin{align*}
\mathbf E\Big[\Dm_{(\mathbf t_n,\mathbf k_n)}^{(n)}\Big(\prod_{i\in\J}h(t_i,k_i)\D\Rm_{(t_i,k_i)}\Big)\Big]
&=\mathbf E\Big[\sum_{\J\subset\N}\car_{\{\J=\id{1}{n}\}}\Big(\prod_{i\in\J}h(t_i,k_i)\D\Rm_{(t_i,k_i)}\big(\pi_{(\mathbf t_n,\mathbf k_n)}(\eta)+\delta_{(\mathbf t_n,\mathbf k_n)}\big)\\
&\qquad\qquad\qquad\qquad\qquad-\prod_{i\in\J}h(t_i,k_i)\D\Rm_{(t_i,k_i)}\big(\pi_{(\mathbf t_n,\mathbf k_n)}(\eta)\big)\Big)\Big]\\
&=\sum_{\J\subset\N}\car_{\{\J=\id{1}{n}\}}\prod_{i\in\J}h(t_i,k_i)=n! \prod_{i=1}^nh(t_i,k_i)=\esp{\Dm_{(\mathbf t_n,\mathbf k_n)}^{(n)}\Fm}.
\end{align*}
On the other hand, by using the alternative characterization of $\Fm=\xi(h)$, and the orthogonality of the centred variables $\D\Rm$,
\begin{multline*}
\mathbf E\Big[\Fm\prod_{i=1}^n\frac{1}{\kappa_i}\D\Rm_{(t_i,k_i)}\Big]=\mathbf E\Big[\esp{\xi(h)}\prod_{i=1}^n\frac{1}{\kappa_i}\D\Rm_{(t_i,k_i)}\Big]+\mathbf E\Big[\prod_{s\in\N}\Big(1+\sum_{k\in\E}h(s,k)\D\Rm_{(s,k)}\Big)\prod_{i=1}^n\frac{1}{\kappa_i}\D\Rm_{(t_i,k_i)}\Big]\\
=\esp{\prod_{i=1}^n\Big(1+\sum_{k\in\E}h(t_i,k)\D\Rm_{(t_i,k)}\Big)\frac{1}{\kappa_i}\D\Rm_{(t_i,k_i)}}=\prod_{i=1}^nh(t_i,k_i).
\end{multline*}
The result is extended to $\L^2(\P)$ by density of the Doléans exponential family.
\end{proof}

\begin{proof}[Proof of Lemma \ref{Stroock_formula_esp_lem}]
It suffices to state the equality for $\Fm=\xi(f)$, $\Gm=\xi(g)$, where $f,g\in\L^2(\Xb)$. On the one hand, there exists $T\in\N$ such that
\begin{align*}
\esp{\Fm\Gm}&-\esp{\Fm}\esp{\Gm}
=\prod_{t\in\id{1}{T}}\prod_{s\in\id{1}{T}}\esp{\Big(1+\sum_{k\in\E}f(t,k)\D\Rm_{(t,k)}\Big)\Big(1+\sum_{\l\in\E}g(s,\l)\D\Rm_{(s,\l)}\Big)}\\
&=\prod_{t\in\id{1}{T}}\Big(1+\sum_{k\in\E}\kappa_kf(t,k)g(t,k)\Big)\\
&=\sum_{n\in\id{1}{T}}\sum_{\underset{|\J|=n}{\J\subset\N}}\prod_{j\in\J}\Big(\sum_{k\in\E}\kappa_kf(t_j,k)g(t_j,k)\Big)=\sum_{n\in\id{1}{T}}\sum_{\underset{|\J|=n}{\J\subset\id{1}{T}}}\prod_{j\in\J}\textless f(t_j,\cdot),g(t_j,\cdot)\textgreater_{\L^2(\Xb,\tilde\nu)^{\otimes n}}.
\end{align*}
On the other hand, for any $n\in\id{1}{T}$ and $\mathrm I_n\subset\id{1}{T}$ of cardinality $n$, denoted $\mathrm I_n=\{(t_j^{\mathrm I_n},k_j^{\mathrm I_n}),\, j\in\id{1}{n}\}$ ,
\begin{align*}
\esp{\Dm_{\mathrm I_n}^{(n)}\Fm}
=\prod_{j\in\mathrm I_n}f(t_j^{\mathrm I_n},k_j^{\mathrm I_n})\mathbf E\bigg[\prod_{j\in\id{1}{T}\setminus\id{1}{n}}\Big(1+\sum_{k\in\E}f(t,k)\D\Rm_{(t,k)}\Big)\bigg]=\prod_{j\in\mathrm I_n}f(t_j,k_j).
\end{align*}
Then, by denoting by $\mathrm I_n^{<}$ the ordered sets $\mathrm I_n$ with respect to the jump times $t_j's$,
\begin{align*}
\sum_{n\in\id{1}{T}}\frac{1}{n!}\textless\mathbf E[\Dm^{(n)}\Fm]\,,\,\mathbf E[\Dm^{(n)}\Gm]\textgreater_{\L^2(\Xb,\tilde\nu)^{\otimes n}}
&=\sum_{n\in\id{1}{T}}\sum_{\underset{|\mathrm I_n^<|=n}{\mathrm I_n^<\subset\id{1}{T}}}\prod_{j\in\mathrm I_n^<}\textless f(t_j,\cdot),g(t_j,\cdot)\textgreater_{\L^2(\Xb,\tilde\nu)^{\otimes n}}.
\end{align*}
The result is extended to $\L^2(\P)$ by density of the class of Doléans exponentials.
\end{proof}

\begin{proof}[Proof of Lemma \ref{Stroock_formula_lem}] The proof follows closely that of Last and Penrose (\cite{LastPenrose2011}, Theorem 1.3).
 Define for any $\Fm\in\L^2(\P)$ and $n\in\Z_+$ the application $\theta_n^\Fm$ by
\begin{equation*}
\theta^\Fm_n(\mathbf s_n,\mathbf l_n)=\esp{\Dm^{(n)}_{(\mathbf s_n,\mathbf l_n)}\Fm}\;;\;\forall (\mathbf s_n,\mathbf l_n)\in\Xb^n.
\end{equation*}
Let $\Fm\in\L^2(\P)$.
The idea is to state the identity for any random variable of the form $\Gm=\xi(g)$ with $g\in\L^2(\Xb)$, well chosen to approximate $\Fm$. Indeed follows from the isometry property \eqref{isometry_eq} that for $(\mathbf s_m,\mathbf l_m)\in\Xb^m$, 
\begin{align*}
\mathbf E\bigg[\Gm\prod_{i=1}^m\frac{\D\Rm_{(s_j,\l_j)}}{\kappa_i}\bigg]
&=\mathbf E\bigg[\Big(\esp{\Gm}+\sum_{n\in\N}\sum_{(\mathbf t_n,\mathbf k_n)\in \Xb^{n}}g_n(\mathbf t_n,\mathbf k_n)\prod_{i=1}^{n}\D\Rm_{(t_i,k_i)}\Big)\prod_{j=1}^m\frac{\D\Rm_{(s_j,\l_j)}}{\kappa_j}\bigg]
\end{align*}
which is equal to $g_m(\mathbf s_m,\mathbf l_m)$, whereas by Lemma \ref{Mult_D_esp_lem}, the right member is also equal to $(m!)^{-1}\theta_m^\Gm((\mathbf s_m,\mathbf l_m))$.  Now, from Lemma \ref{Stroock_formula_esp_lem} together with the isometry identity \eqref{isometry_proc_eq}, follows 
\begin{equation*}
\sum_{n=0}^{\infty}\esp{\frac{1}{n!} \Jt_n(f_n)}^2=\sum_{n=0}^{\infty}\frac{1}{n!}\|f_n\|_{\L^2(\Xb)^{\circ n}}^2=\esp{\Fm^2}<\infty.
\end{equation*}
Hence the infinite series of orthogonal terms
$\S:=\sum_{n\in\Z_+}\frac{1}{n!}\Jt_n(\theta_n^\Fm )$
converges in $\L^2(\P)$. Then,
\begin{equation*}
\esp{(\S-\xi(g))^2}=\sum_{n\in\Z_+}\frac{1}{n!}\|\theta_n^\Fm-\theta_n^\Gm\|_{\L^2(\Xb)^{n}}^2=\esp{(\Fm-\xi(g))^2},
\end{equation*}
so that, since the set of Doléans exponentials is dense in $\L^2(\P)$ and $\S$ converges in $\L^2(\P)$, the equality $\Fm=\S$ stands $\P$-almost surely. To prove uniqueness, assume there exists $\Hm\in\L^2(\P)$ which decomposition satisfies \eqref{Dom_D_condition_chaos_eq} and such that $\S':=\sum_{n\in\Z_+}(n!)^{-1}\Jt_n(h_n)$ converges in $\L^2(\P)$ to $\Fm$. Taking the expectation entails $h_0=\esp{\Fm}=\theta_0^\Fm$. For $n\in\N$, follows  from Lemma \ref{Stroock_formula_esp_lem} that $\esp{\Fm\Jt_n(g)}=n!\textless \theta_n^\Fm,g\textgreater_{\L^2(\Xb,\tilde\nu)^{\otimes n}}$ and by replacing $\theta_n^\Fm$ by $h_n$ (since the convergence of $\S'$ holds in $\L^2(\P)$) that $\esp{\Fm\Jt_n(g)}=n!\textless h_n,g\textgreater_{\L^2(\Xb,\tilde\nu)^{\otimes n}}$. Then $\|\theta_n^\Fm-h_n\|_{\L^2(\Xb,\tilde\nu)^{\otimes n}}^2=\textless\theta_n^\Fm-h_n,g\textgreater_{\L^2(\Xb,\tilde\nu)^{\otimes n}}$ is equal to zero by taking $g=\theta_n^\Fm-h_n$. The proof is thus complete.   
\end{proof}

\begin{proof}[Proof of Proposition \ref{Mehler_formula_prop}]
It suffices to prove it for $\Fm=\xi(h)=\fk(\eta)$ where, as $\eta$ is finite, there exists $T\in\N$ such that 
\begin{equation*}
\fk(\eta)=\prod_{s\in\id{1}{T}}\Big(1+\sum_{k\in \E}g(s,k)(\car_{\{(s,k)\in\eta\}}-\lambda\Q(\{k\}))\Big).
\end{equation*}
where $g\in\L^2(\Xb_T)$ is such that $\J_1(h)=\J_1(g\,;\,\mathcal Z)$.
On the one hand, by action of the semi-group $\Pm$ on the quasi-chaotic decomposition \eqref{Chaos_R_eq}, 
\begin{equation*}
\Pm_\tk\Fm=\xi(e^{-\tk}u)
=\prod_{s\in\N}\Big(1+e^{-\tk}\sum_{k\in \E}g(s,k)(\car_{\{(s,k)\in\eta\}}-\lambda\Q(\{k\}))\Big)
\end{equation*}
On the other hand, by definition of $\eta^{\tk,0}$ \eqref{wtk_def_eq} and $\tilde\eta$, which law given $\eta$ is  provided by \eqref{Law_tilde_eta}, 
\begin{align*}
\esp{\fk(\eta^{\tk,0}+\varepsilon\tilde\eta)\big|\eta}
&=\prod_{s\in\id{1}{T}}\esp{1+\sum_{k\in \E}g(s,k)\big(\car_{\{(s,k)\in(\eta^{\tk,0}+\varepsilon_s^\tk\tilde\eta)\}}-\lambda\Q(\{k\})\big)\,\bigg|\eta}\\
&=\prod_{s\in\id{1}{T}}\Big(1+\sum_{k\in \E}g(s,k)((1-e^{-\tk})\lambda\Q(\{k\}))+e^{-\tk}\car_{\{(s,k)\in\eta\}}-\lambda\Q(\{k\})\Big),\\
&=\prod_{s\in\id{1}{T}}\Big(1+e^{-\tk}\sum_{k\in \E}g(s,k)\big(\car_{\{(s,k)\in\eta\}}-\lambda\Q(\{k\})\big)\Big)=\Pm_\tk\Fm.
\end{align*}
Since $\eta$ is finite, the result holds in $\L^2(\P)$. The proof is complete.
\end{proof}

\begin{proof}[Proof of Proposition \ref{Commutation_PtF_prop}] Let $\Fm=\xi(h)=\fk(\eta)$ such that $\esp{\Fm}=1$ and
\begin{equation*}
\fk(\eta)=\prod_{s\in\N}\Big(1+\sum_{k\in \E}g(s,k)\big(\car_{\{(s,k)\in\eta\}}-\lambda\Q(\{k\})\big)\Big),
\end{equation*}
where $g\in\L^2(\Xb)$ is such that $\J_1(h)=\J_1(g\,;\,\mathcal Z)$.
Then, from Mehler's formula \eqref{Mehler_formula_prop_eq},
\begin{equation*}
\Pm_\tk\fk(\eta)=\xi(e^{-\tk}h)
=\prod_{s\in\N}\Big(1+e^{-\tk}\sum_{k\in \E}g(s,k)\big(\car_{\{(s,k)\in\eta\}}-\lambda\Q(\{k\})\big)\Big).
\end{equation*}
On the one hand, for any $(s,k)\in\Xb$,
\begin{multline*}
\Dm_{(s,k)}\Pm_\tk\fk(\eta)=\prod_{r\in\N}\Big(1+e^{-\tk}\sum_{\l\in \E}g(r,\l)\big(\car_{\{(r,\l)\in(\pi_s(\eta)+\delta_{(s,k)})\}}-\lambda\Q(\{k\})\big)\Big)\\
-\prod_{r\in\N}\Big(1+e^{-\tk}\sum_{\l\in \E}g(r,\l)\big(\car_{\{(r,\l)\in\pi_s(\eta)\}}-\lambda\Q(\{k\})\big)\Big)=e^{-\tk}g(s,k)\,\Pm_\tk\fk(\pi_s(\eta)).
\end{multline*}
\noindent
On the other hand, follows from
\begin{equation*}
\Dm_{(s,k)}\fk(\eta)=g(s,k)\prod_{r\in\N\setminus{\{s\}}}\Big(1+\sum_{k\in \E}g(r,k)\big(\car_{\{(r,k)\in\pi_s(\eta)\}}-\lambda\Q(\{k\})\big)\Big)=g(s,k)\fk(\pi_s(\eta)),
\end{equation*}
that for any $(s,k)\in\Xb$,
\begin{align*}
\Pm_\tk(\Dm_{(s,k)}\fk(\pi_s(\eta)))
&=g(s,k)\prod_{r\in\N\setminus{\{s\}}}\Big(1+e^{-\tk}\sum_{k\in \E}g(r,k)\big(\car_{\{(r,k)\in\pi_s(\eta)\}}-\lambda\Q(\{k\})\big)\Big)\\
&=g(s,k)\,\Pm_\tk\fk(\pi_s(\eta)).
\end{align*}
Hence the result.
\end{proof}

\begin{proof}[Proof of Proposition \ref{Inverse_L_cor}]
Let $\Fm\in\L^2(\P)$ such that $\esp{\Fm}=0$. For any $m\in\N$,
\begin{equation}\label{inverseLtheoremeq1}
\L^{-1}\left(\sum_{n=1}^m\Jt_n(f_n)\right)
=-\sum_{n=1}^m\frac{1}{n}\Jt_n(f_n)=-\int_0^{\infty}\sum_{n=1}^me^{-n\tk}\Jt_n(f_n)\,\d\tk.
\end{equation}
Moreover, the random variable $\Rm_m$ defined by
\begin{equation*}
\Rm_m:=\int_0^{\infty}\Big(\Pm_\tk\Fm-\sum_{n=1}^me^{-n\tk}\Jt_n(f_n)\Big)\d \tk=\int_0^{\infty}\Big(\sum_{n=m+1}^{\infty}e^{-n\tk}\Jt_n(f_n)\Big)\d\tk
\end{equation*}
converges to zero in $\L^2(\P)$ by noting that $\Jt_0(f_0)=\esp{\Fm}=0$ and 
\begin{equation*}
\esp{\Rm_m^2}\pp \int_0^{\infty}\esp{\Big(\sum_{n=m+1}^{\infty}e^{-n\tk}\Jt_n(f_n\Big)^2}\d\tk=\sum_{n=m+1}^{\infty} n!\|f_n\|_{\L^2(\Xb,\tilde\nu)^{\otimes n}}^2\int_0^{\infty}e^{-2n\tk}\,\d\tk.
\end{equation*}
Checking that
\begin{equation*}
\esp{\Big(\int_{0}^{\infty}\Pm_\tk\Fm\,\d \tk\Big)^2}\pp \esp{\int_{0}^{\infty}|\Pm_\tk\Fm|^2\,\d \tk}=\sum_{n=1}^{\infty} n!\|f_n\|_{\L^2(\Xb,\tilde\nu)^{\otimes n}}^2\int_0^{\infty}e^{-2n\tk}\,\d\tk
\end{equation*}
is finite, the proof of the first point is complete by letting $n$ go to infinity in \eqref{inverseLtheoremeq1}. As for \eqref{Inverse_L2_eq}, the commutation \eqref{Commutation_PtF_eq} and contractivity \eqref{Contractivity_PtF_eq}  properties satisfied by $(\Pm_\tk)_{\tk\in\R_+}$ ensure that 
\begin{equation*}
\esp{\int_{0}^{\infty}|\Dm_{(s,\l)}\Pm_\tk\Fm|\,\d \tk}=\esp{\int_{0}^{\infty}e^{-\tk}|\Pm_\tk\Dm_{(s,\l)}\Fm|\,\d \tk}\pp \esp{|\Dm_{(s,\l)}\Fm|}
\end{equation*}
is finite for $\nu$-a.e. $(s,\l)\in\Xb$. The result follows by applying the operator $\Dm$ to each side of equality \eqref{Inverse_L_eq} and using of the commutation property \eqref{Commutation_PtF_eq}. 
\end{proof}

\subsection{Proofs of section \ref{Functional_identities_sec}}

\begin{proof}[Proof of Theorem \ref{Girsanov_th}] Let $\Q$ be an equivalent measure to $\P$ on $\F$. Assume that $\E=\{k^i,\, i\in\Z\}$. Then there exist a real number $\beta$ and a collection of real numbers $\{\alpha_k,\,k\in\E\}$ all in $(0,1)$, satisfying  $\beta+\sum_{k\in\E}\alpha_k=1$, such that
\begin{equation*}
\Q_1
=\beta\delta_{\mathbf 0}+\sum_{k\in\E}\alpha_k\delta_{(1,k)}
=\frac{\beta}{1-\lambda}(1-\lambda)\delta_{\mathbf 0}+\sum_{k\in\E}\frac{\alpha_k}{\lambda \Q(\{k\})}\lambda \Q(\{k\})\delta_{(1,k)}
=\esp{\L_1\car_{\cdot}},
\end{equation*}
where $\L_1$ is the random variable defined by
\begin{equation*}
\L_1=\frac{1-\tilde\lambda}{1-\lambda}\car_{\{\eta(\Xb)= 0\}}+\sum_{k\in\E}\frac{\tilde\lambda\widetilde \Q(\{k\})}{\lambda \Q(\{k\})}\car_{\{(1,k)\in\eta\}}=\frac{1-\tilde\lambda}{1-\lambda}+\sum_{k\in\E}\Big(\frac{\tilde\lambda\widetilde \Q(\{k\})}{\lambda \Q(\{k\})}-\frac{1-\tilde\lambda}{1-\lambda}\Big)\car_{\{(1,k)\in\eta\}},
\end{equation*}
with $\tilde\lambda=1-\beta$ and $\widetilde \Q(\{k\})=\alpha_k/\tilde\lambda$, to get the $\Q(\{k\})$ summoned to $1$.
Let now $\L$ be the $(\f{t})_{t\in\id{1}{T}}$-martingale such that 
\begin{equation*}
\frac{\d\widetilde\P}{\d\P}\Big|_{\f{t}}=\L_t \;\;; \; t\in\id{1}{T}.
\end{equation*}
Since the increments of the jump process $(\Nm_t)_{t\in\id{1}{T}}$ are independent and identically distributed, we can show by induction that $\L$ is defined for any $t\in\id{1}{T}$ by
\begin{equation*}
\L_t=\prod_{s=1}^t\L_s=\prod_{s=1}^t\left(\frac{1-\tilde\lambda}{1-\lambda}+\sum_{i\in\Z}\Big(\frac{\tilde\lambda\widetilde \Q(\{k^i\})}{\lambda \Q(\{k^i\})}-\frac{1-\tilde\lambda}{1-\lambda}\Big)\car_{\{(s,k^i)\in\eta\}}\right).
\end{equation*}
Let $u$ be the process defined by \eqref{Girsanov_drift_eq}. For any $t\in\id{1}{T}$, by using $\car_{\{(s,k^i)\in\eta\}}=\D\Zm_{(s,k^i)}+\lambda \Q(\{k^i\})$ and noting that 
\begin{equation*}
\sum_{i\in\Z}\Big(\frac{\tilde\lambda\widetilde \Q(\{k^i\})}{\lambda \Q(\{k^i\})}-\frac{1-\tilde\lambda}{1-\lambda}\Big)\lambda \Q(\{k^i\})=\tilde\lambda-\frac{\lambda(1-\tilde\lambda)}{1-\lambda}=1-\frac{1-\tilde\lambda}{1-\lambda},
\end{equation*}
we can check that, by the definition of $h$ given in the theorem, that
\begin{equation*}
\xi_t(h)
=\prod_{s=1}^t\Big(1+\sum_{i\in\Z}g(s,k^i)\D\Zm_{(s,k^i)}\Big)=\prod_{s=1}^t\left(\frac{1-\tilde\lambda}{1-\lambda}+\sum_{i\in\Z}\Big(\frac{\tilde\lambda\widetilde \Q(\{k^i\})}{\lambda \Q(\{k^i\})}-\frac{1-\tilde\lambda}{1-\lambda}\Big)\car_{\{(s,k^i)\in\eta\}}\right)=\L_t.
\end{equation*}
Hence the result.
\end{proof}	

\begin{proof}[Proof of Corollary \ref{Girsanov_cor}]
Let $\varphi$ and $\widetilde\P$ as defined in the theorem. Let $\widetilde{\mathbf E}$ denote the expectation taken under the probability measure $\widetilde\P$. For any $s\in\R^*$, $t\in\N$,
\begin{align*}
\widetilde{\mathbf E}[s^{\Y_t}]
&=\left(\frac{1-\tilde\lambda}{1-\lambda}\right)^t\sum_{n=0}^t\esp{s^{\Y_t}\prod_{r=1}^{n}(1+\varphi(\V_r))\Big| \Nm_t=n}\P(\Nm_t=n)\\
&=\left(\frac{1-\tilde\lambda}{1-\lambda}\right)^t\sum_{n=0}^t\binom{t}{n}\lambda^n(1-\lambda)^{t-n}\esp{\prod_{r=1}^{n}(1+\varphi(\V_r))s^{\V_r}\Big| \Nm_t=n}\\
&=(1-\tilde\lambda)^t\sum_{n=0}^t\binom{t}{n}\left(\frac{\lambda}{1-\lambda}\right)^n\left(\frac{\tilde\lambda(1-\lambda)}{\lambda(1-\tilde\lambda)}\right)^n\prod_{r=1}^{n}\left(\sum_{k\in\E}\frac{\widetilde \Q(\{k\})}{\Q(\{k\})}\cdot \Q(\{k\})s^k\right)\\
&=\sum_{n=0}^t\binom{t}{n}\widetilde\lambda^n(1-\tilde\lambda)^{t-n}
\bigg(\sum_{k\in\E}\widetilde \Q(\{k\})s^k\bigg)^n=\Big(1-\tilde\lambda+\tilde\lambda\widetilde{\mathbf E}_{\Q}[s^{\V}]\Big)^t,
\end{align*}
where $\V$ is a $\E$-valued random variable and $\widetilde{\mathbf E}_{\Q}$ is the expectation taken under the probability measure $\widetilde\Q$.
Hence the result.
\end{proof}

\begin{proof}[Proof of Theorem \ref{Clark_formula_prop}]
Let $\Fm\in\cyl$; follows from both its chaotic decomposition \eqref{Chaos_R_eq} and the definition of the  gradient operator \eqref{Gradient_def_eq} that for some $T\in\N$,
\begin{align*}
\Fm
&=\esp{\Fm}+\sum_{n\in\N}\Jt_n(f_n\car_{\id{1}{T}^n})\\
&=\esp{\Fm}+\sum_{n\in\N}n\sum_{(t,k)\in\Xb_T}\Jt_{n-1}\big(f_n(\star,(t,k))\car_{\id{1}{t-1}^{n-1,<}}\big)\,\D\Rm_{(t,k)}\\
&=\esp{\Fm}+\sum_{(t,k)\in\Xb_T}\sum_{n\in\N}\esp{n\Jt_{n-1}\big(f_n(\star,(t,k))\big|\f{t-1}}\,\D\Rm_{(t,k)}\\
&=\esp{\Fm}+\sum_{(t,k)\in\Xb_T}\esp{\Dm_{(t,k)}\Fm\,|\,\f{t-1}}\,\D\Rm_{(t,k)},
\end{align*}
where we have used lemma \ref{Cond_int_lem} to get the third line. As noticed in Remark \ref{Boundness_D_rem}, the operator $\Fm\in\L^2(\P)\mapsto\big(\mathbf E[\Dm_{(t,k)}\Fm],\,(t,k)\in\Xb\big)$ is bounded with norm equal to $1$; the result can be thus extended to any random variable $\Fm\in \L^2(\P)$  using a standard Cauchy argument. 
\end{proof}

\begin{proof}[Proof of Corollary \ref{Poincare}]
According to~\eqref{Clark_formula_eq}, we have
\begin{align*}
\mathrm{var} (\Fm)
&=\mathbf E\Bigg[\bigg|\sum_{(t,k)\in\Xb}\esp{\Dm_{(t,k)}\Fm\,|\,\f{t-1}}\,\D\Rm_{(t,k)}\bigg|^2\Bigg]\\
& =\mathbf E\Bigg[\sum_{(t,k)\in\Xb}\kappa_k\Big|\esp{\Dm_{(t,k)}\Fm\,|\,\f{t-1}}\Big|^2\Bigg]\\
&\pp \mathbf E\Bigg[\sum_{(t,k)\in\Xb}\kappa_k\esp{|\Dm_{(t,k)}\Fm|^2|\f{t-1}}\Bigg]
=\mathbf E\bigg[\int_{\Xb}|\Dm_{(t,k)}\Fm|^2\,\d\tilde\nu(t,k)\bigg],
\end{align*}
where the inequality follows from Jensen's.
\end{proof}

\subsection{Proofs of section \ref{Applications_sec}}

\subsubsection{Proofs of Subsection 5.1}

\begin{proof}[Proof of Theorem \ref{Approximation_Poisson_th}]
Let $\Fm\in\DD$ be a $\Z_+$-valued random variable such that $\esp{\Fm}=\lambda_0$. Let $\A\subset\Z_+$ and $\varphi_\A$ be the solution of the Stein equation \eqref{Stein_Poisson_eq}. The uniform boundedness of $\nabla\varphi_\A$ ensures that $\varphi_\A(\Fm)\in\DD$, whereas
\begin{equation*}
\big|\varphi_\A(\Fm)\Dm^+\L^{-1}(\Fm-\esp{\Fm})\big|\pp\|\varphi\|_{\infty}|\Dm^+\L^{-1}(\Fm-\esp{\Fm}))|
\end{equation*}
implies that $\varphi_\A(\Fm)\Dm^+\L^{-1}(\Fm-\esp{\Fm})\in\L^1(\P\otimes\nu)$. Under this condition together with the hypotheses of the theorem, the combination of $\L^1$ and $\L^2$ theories performs; in particular, since $\Fm\in\DD$, $\Dm\Fm=\Dm^+\Fm$ for all $\Fm\in\L^2(\P)$, almost surely. Moreover, as $\E=\{1\}$, follows from remark \ref{E_singleton_remark} that $\widetilde{\L}\Fm=\L\Fm$ $\P$-almost surely and that the integration by parts formula \eqref{IPP_gen_eq} holds. Note besides that here $\pi_{t}(\eta)+\delta_{t}=\eta+\delta_{t}$ so that $\widetilde\Dm_{t}\Fm=\fk(\eta+\delta_{t})-\fk(\eta)$, $\P$-almost surely. 
By definition of the operators $\L$ and $\L^{-1}$ and using \eqref{IPP_tildeD_eq}, we get
\begin{align*}
\esp{\Fm\varphi_\A(\Fm)-\lambda_0\varphi_\A(\Fm+1)}
&=\esp{(\Fm-\esp{\Fm})\varphi_\A(\Fm)-\lambda_0\nabla\varphi_\A(\Fm)}\\
&=\esp{(\widetilde \L\widetilde{\L}^{-1}(\Fm-\esp{\Fm}))\varphi_\A(\Fm)}-\esp{\lambda_0\nabla\varphi_\A(\Fm)}\\
&=-\esp{\textless\widetilde\Dm(\varphi_\A(\Fm)),\Dm{\L}^{-1}(\Fm-\esp{\Fm})\textgreater_{\L^2(\Xb,\nu)}}-\esp{\lambda_0\nabla\varphi_\A(\Fm)}\\
&=-\esp{\nabla\varphi_\A(\Fm)\textless\widetilde\Dm\Fm,\Dm{\L}^{-1}(\Fm-\esp{\Fm})\textgreater_{\L^2(\Xb,\nu)}+\rem}-\esp{\lambda_0\nabla\varphi_\A(\Fm)},
\end{align*}
where we have used to get the third line that
\begin{multline*}
\textless\widetilde\Dm(\varphi_\A(\Fm)),\Dm{\L}^{-1}(\Fm-\esp{\Fm})\textgreater_{\L^2(\Xb,\nu)}=\nabla\varphi_\A(\Fm)\int_{\N}\big(\widetilde\Dm_{t}\Fm\big)\big(\Dm_{t}{\L}^{-1}(\Fm-\esp{\Fm})\big)\,\nu(\d t)\\+\int_{\N}\mathfrak R_t\big(\Dm_{t}{\L}^{-1}\Fm\big)\,\nu(\d t),
\end{multline*}
where $\mathfrak R_t$ is a residual random function such that $\mathfrak R_t\pp\|\nabla^2\varphi_\A\|_{\infty}\big|(\widetilde\Dm_{t}\Fm)(\widetilde\Dm_{t}\Fm-1)\big|/2$.
By using inequality \eqref{Peccati_ineq},  with $k=\fk(\eta)+\widetilde\Dm_{t}\Fm$ and $a=\fk(\eta)$,  we get
\begin{multline*}
\esp{|\rem|}=\esp{\Big|\int_{\N}\mathfrak R_t(\Dm_{t}{\L}^{-1}(\Fm-\esp{\Fm}))\,\nu(\d t)\Big|}
\\\pp\frac{\|\nabla^2\varphi_\A\|_{\infty}}{2}\esp{\int_{\N}|(\widetilde\Dm_{t}\Fm)(\widetilde\Dm_{t}\Fm-1)\big|\big|\Dm_{t}\L^{-1}(\Fm-\esp{\Fm})\big|\,\nu(\d t)}.
\end{multline*}
Then,
\begin{align*}
\big|\esp{\Fm\varphi_\A(\Fm)-\lambda_0\varphi_\A(\Fm+1)}\big|
&\pp\|\nabla\varphi_\A\|_{\infty}\esp{\big|\lambda_0-\textless\widetilde\Dm\Fm,-\Dm\L^{-1}(\Fm-\esp{\Fm})\textgreater_{\L^2(\Xb,\nu)}\big)\big|}\\
&+\frac{\|\nabla^2\varphi_\A\|_{\infty}}{2}\esp{\int_{\N}\big|(\widetilde\Dm_{t}\Fm)(\widetilde\Dm_{t}\Fm-1)\big|\big|\Dm_{t}\L^{-1}(\Fm-\esp{\Fm})\big|\,\nu(\d t)}.
\end{align*}
The result is then stated by taking the supremum over the set $\{\A\subset\Z_+\}$ and using the uniform bounds (over the class ) estimates on $\nabla_\cdot\varphi$ and $\nabla_\cdot^2\varphi$.
\end{proof}

\begin{proof}[Proof of Proposition \ref{Approximation_compound_Poisson_th}]
Let $\Fm\in\DD$ be a $\Z_+$-valued random variable such that $\esp{\Fm}=\lambda_0\esp{\V_1}$. Via Stein's method and the definition of the Stein operator for Poisson compound approximation, we are led to control
\begin{equation*}
\mathbf E\Big[\Fm\psi_\A(\Fm)-\int_{\Xb}k\psi_{\A}(\Fm+k)\d\nu(t,k)\Big]
=\mathbf E\Big[(\Fm-\esp{\Fm})\psi_\A(\Fm)-\int_{\Xb}k(\psi_\A(\Fm+k)-\psi_\A(\Fm))\d\nu(t,k)\Big],
\end{equation*} 
where, by definition of the operators $\widetilde{\L}$ and $\widetilde{\L}^{-1}$, we have
\begin{equation*}
\esp{(\Fm-\esp{\Fm})\psi_\A(\Fm)}=\mathbf E\big[(\widetilde\L\widetilde\L^{-1})(\Fm-\esp{\Fm})\psi_\A(\Fm)\big]=\mathbf E\Big[\textless \Dm\widetilde\L^{-1}(\Fm-\esp{\Fm}),\widetilde\Dm\psi_\A(\Fm) \textgreater_{\L^2(\Xb,\nu)}\Big].
\end{equation*}
On the other hand,
\begin{multline*}
\textless \Dm^+\widetilde\L^{-1}(\Fm-\esp{\Fm}),\widetilde\Dm\psi_\A(\Fm) \textgreater_{\L^2(\Xb,\nu)}-\int_{\Xb}k(\psi_\A(\Fm+k)-\psi_\A(\Fm))\,\d\nu(t,k)\\
=\int_{\Xb}\,\Dm^+\widetilde\L^{-1}\big(\fk(\eta)-\esp{\fk(\eta)}\big)\big[\psi_\A\big(\fk(\pi_t(\eta)+\delta_{(t,k)})\big)-\psi_\A(\fk(\eta))\big]\,\d\nu(t,k)\\
\qquad\qquad\qquad\qquad\qquad\qquad\qquad\qquad\qquad\qquad-\int_{\Xb}k\big[\psi_\A(\fk(\eta)+k)-\psi_\A(\fk(\eta))\big]\,\d\nu(t,k)\\
=\int_{\Xb}\,\big[\Dm^+\widetilde\L^{-1}\big(\fk(\eta)-\esp{\fk(\eta)}\big)\psi_\A\big(\fk(\pi_t(\eta)+\delta_{(t,k)})\big)-k\psi_\A(\fk(\eta)+k)\big]\,\d\nu(t,k)\\
\quad-\int_{\Xb}\psi_\A(\fk(\eta))\big[\Dm^+\widetilde\L^{-1}\big(\fk(\eta)-\esp{\fk(\eta)}\big)-k\big]\,\d\nu(t,k).
\end{multline*}
The conclusion follows by taking the expectation and then the supremum over the class $\{\psi_\A,\;\A\subset \Z_+\}$. 
\end{proof}

\begin{proof}[Proof of Theorem \ref{longest_head_run_th}]
Since $\Um$ is a simple binomial functional (since the space mark boils down to a singleton), $\L^1$ and $\L^2$ theories combine perfectly, $\Dm\Um=\Dm^+\Um$, $\widetilde\Dm_t\Um=\car_{\{t\notin\eta\}}\Dm^+\Um$, and we get $\widetilde\L\Um=\L\Um$ $\P$-almost surely. Note that $\lambda_0^2=p^{2m}+2(n-1)qp^{2m}+(n-1)^2q^2p^{2m}$. Then,
\begin{align}\label{Bound_Gamma_head_run}
\nonumber
\mathbf E\big[\textless\widetilde\Dm\Um,-\Dm{\L}^{-1}(\Um-\esp{\Um})\textgreater_{\L^2(\Xb,\nu)}\big]
&=\mathbf E\big[(\L\L^{-1}(\Um-\esp{\Um}))\Um\big]=\var[\Um]\\
\nonumber
&=\lambda_0+2\sum_{i=m+1}^{n-1}\esp{\Um_0\Um_i}+2\sum_{1\pp i<j}\esp{\Um_i\Um_j}-\lambda_0^2\\
\nonumber
&=\lambda_0+2(n-m-1)qp^{2m}+(n-1)(n-2m-2)q^2p^{2m}-\lambda_0^2\\
&=\lambda_0-2mqp^{2m}-(2m-1)q^2p^{2m}-p^{2m}.
\end{align}
Then,
\begin{equation*}
\big|\lambda_0-\mathbf E\big[\textless\widetilde\Dm\Um,-\Dm{\L}^{-1}(\Um-\esp{\Um})\textgreater_{\L^2(\Xb,\nu)}\big]\big|=p^{2m}[2(m-1)q^2+2mq+1].
\end{equation*}
\noindent
On the other hand, using Corollary \ref{Inverse_L_cor} 
\begin{align*}
\mathbf E \bigg[\int_{\N}\big|(\widetilde\Dm_{s}\Um)(\widetilde\Dm_{s}\Um-1)\big|\big|&\Dm_{s}\L^{-1}(\Um-\esp{\Um})\big|\nu(\d s)\bigg]\\
&\pp\esp{ \int_{\id{1}{n}}\Big(\car_{\{t\notin\eta\}}\big|(\Dm_s\Um)^2-(\Dm_s\Um)\big|\Big|\int_0^{\infty}\Dm_s\Pm_{\tk}\Um\,\d \tk\Big|\Big)\nu(\d s)}\\
&=\esp{ \int_{\id{1}{n}}\Big(\big|(\Dm_s\Um)^2-(\Dm_s\Um)\big|\Big|e^{-s}\Dm_s\Um\,\d \tk\Big|\Big)\nu(\d s)},
\end{align*}
where we have used that $\Dm_t\Um$ is $\g{t}$-measurable. Moreover,
\begin{equation*}
(\Dm_t\Um)^3
=\bigg(\prod_{i=1,i\neq t}^{m}\D\Nm_i+\sum_{i=t-m}^{t-1}(1-\D\Nm_{i})\prod_{\l=1,i+\l\neq t}^{m}\D\Nm_{i+\l}-\prod_{\l=1}^m\D\Nm_{t+\l}\bigg)^3=:(\A+\B-\Cm)^3,
\end{equation*}
\\
\noindent
Note first that for $t>m$, $\Dm_t\big(\prod_{i=1}^m\D\Nm_i\big)=0$ so that 
$\Dm_t\Um=\B-\Cm$.
Assume there exists $i_0\in\{t-m,t-1\}$ such that $\B_{i_0}=1$ where $\B_{i}:=(1-\D\Nm_{i})\prod_{\l=1,i+\l\neq t}^{m}\D\Nm_{i+\l}$. Then, $\D\Nm_{i_0}=0$  implies $\B_i=0$ for any $i\in\{\max(0,i_0-m),i_0-1\}$ whereas  $\D\Nm_{i_0+\l}=1$ leads to $\B_i=0$ for any $i\in\id{i_0+1}{\min(i_0+m,t-1)}$; this entails $\B=\sum_{i=t-m}^{t-1}\V_i=0+\B_{i_0}=1$. Thus $\B\in\{0,1\}$ and $\P(\{\B\pp1\})=1$.
Besides,
\begin{align*}
(\Dm_t\Um)^3
&=\B^3-3\B^2\Cm+3\B\Cm^2-\Cm^3\\
&=(\B-2\B\Cm+\Cm)+2\B\Cm-2\Cm=(\Dm_t\Um)^2+2\Cm(\B-1).
\end{align*}
Since $\B,\Cm\in\{0,1\}$, this proves that  $(\Dm_{t}\Um)^3\pp(\Dm_t\Um)^2$ $\P$-p.s.
On the event $\{\Cm=1\}$, $\prod_{\l=1}^{m}\D\Nm_{t+\l}=1$ so that $\B_{t-1}$ is equal to $0$  provided $\D\Nm_{t-1}=1$. This entails $\prod_{\l=1,\l\neq 2}^{m}\D\Nm_{t-2+\l}$ is equal to $1$ and then $\B_{t-2}$ is equal to $0$ if and only if $\D\Nm_{t-2}=1$. By induction, we can prove that, on $\{\Cm=1\}$, $\B=0$ if and only if $\D\Nm_i=1$ for any $i\in\{t-m,\dots,t-1\}$ (with probability $p^m$). Then 
\begin{equation*}
\esp{\Cm(1-\B)}=\esp{\car_{\{\B=0\}}\big|\{\Cm=1\}}\P(\{\Cm=1\})=p^{2m},
\end{equation*}
so that we get for $t>m$,
\begin{equation*}
\esp{(\Dm_t\Um)^2-(\Dm_t\Um)^3}
= p^{2m},
\end{equation*}
On the other hand for $t\pp m$, $\Dm_t\Um=(\A+\B+\Cm)$ and $\A\B=0$. In that case, since $\A,\Cm\in\{0,1\}$,
\begin{equation*}
\esp{((\Dm_t\Um)^3-(\Dm_t\Um)^2)\car_{\{\A=1\}}}=\esp{(1-\Cm)^2\Cm\car_{\{\A=1\}}}=0.
\end{equation*}
On $\{\A=0\}$, $\Dm_t\Um=\B-\Cm$ and we can reason as above to prove that for $t\pp m$, \\ $\esp{((\Dm_t\Um)^3-(\Dm_t\Um)^2)}=(1-p^{m-1})\esp{((\Dm_t\Um)^3-(\Dm_t\Um)^2)\car_{\{\A=0\}}}=(1-p^{m-1})p^{2m}\pp p^{2m}$. Finally,
\begin{multline*}
\displaystyle\int_{\id{1}{n-m-1}}\mathbf E \Big[\big|(\widetilde\Dm_{t}\Um)(\widetilde\Dm_{t}\Um-1)\big|\big|\Dm_{t}\L^{-1}(\Um-\esp{\Um})\big|\Big]\,\nu(\d t)
\\=\displaystyle\int_{\id{1}{n-m-1}}\mathbf E \Big[\car_{\{t\notin\eta\}}\big|(\Dm_{t}\Um)(\Dm_{t}\Um-1)\big|\big|\Dm_{t}\L^{-1}(\Um-\esp{\Um})\big|\Big]\,\nu(\d t)\pp (n-m-1)(1-p)p^{2m+1}.
\end{multline*}
This together with \eqref{Bound_Gamma_head_run} provides the bound in Theorem  \ref{Approximation_compound_Poisson_th}. 
\end{proof}

\begin{proof}[Proof of Theorem \ref{DNA_count_word_th}]
Consider the random variable $\Hm=\sum_{j\in\I}\V_j\D\Nm_j$, as defined in the theorem. Since $\Hm$ belongs to $\mathcal  H_1$, then $\widetilde \L\Hm=\Hm$ so that $\Dm\widetilde\L^{-1}(\Hm-\esp{\Hm})=\Dm^+\widetilde\L^{-1}(\Hm-\esp{\Hm})=\Dm^+\Hm$.  Moreover  for any $(t,k)\in\I\times\N$, $\Dm_{(t,k)}^+\Hm=k$ and $\wDm_{(t,k)}\Hm=(k-\l)\car_{\{(t,\l)\in\eta\}}$ $\P$-almost surely. Then the second term in \eqref{Approximation_compound_Poisson_eq} vanishes and it remains to control
\begin{equation*}
\int_{\Xb}\,k\esp{\psi_\A(\hk(\pi_t(\eta)+\delta_{(t,k)}))-\psi_\A(\hk(\eta)+k)}\d\nu(t,k).
\end{equation*} 
On the other hand, by denoting $\Hm^{\neg t}=\sum_{j\in\I\setminus\{t\}}\V_j\D\Nm_j$,
\begin{align*}
\mathbf E\big[\psi_\A(\hk(\pi_t(\eta)+&\delta_{(t,k)}))-\psi_\A(\hk(\eta)+k)\big]\\
&=\sum_{\l\in\N}\esp{\Big(\psi_\A\big(\hk(\pi_t(\eta)+\delta_{(t,k)})\big)-\psi_\A\big(\hk(\eta)+k\big)\Big)\car_{\{(t,\l)\in\eta\}}}\\
&=\sum_{\l\in\N}\esp{\Big(\psi_\A\big(\hk(\pi_t(\eta)+\delta_{(t,k)})\big)-\psi_\A\big(\hk(\pi_t(\eta)+\delta_{(t,\l)})+k\big)\Big)\car_{\{(t,\l)\in\eta\}}}\\
&\pp \|\nabla\psi_\A\|_{\infty}\sum_{\l\in\N}\esp{\Big(\Hm^{\neg t}+k-\big(\Hm^{\neg t}+k+\l\big)\Big)\car_{\{(t,\l)\in\eta\}}}\\
&=\|\nabla\psi_A\|_{\infty}\sum_{\l\in\N}\l\mu(\Wm)(1-\alpha)^2\alpha^{\l-1}=\|\nabla\psi_A\|_{\infty}\mu(\Wm).
\end{align*}
Then,
\begin{equation*}
\Big|\int_{\Xb}\,k\esp{\psi_\A[\hk(\pi_t(\eta)+\delta_{(t,k)})]-\psi_\A[\hk(\eta)+k]}\,\d\nu(t,k)\Big|\pp (n-h+1)\mathfrak d_{\EuScript P\EuScript C}\mu(\Wm)^2.
\end{equation*}
This provides a bound for $\mathrm{dist}_{\mathrm{TV}}(\P_\Hm,\EuScript{P}\EuScript C(\lambda_0,\mathbf V))$ and the conclusion follows by using the triangular inequality together with the previous bounds.
\end{proof}

\subsubsection{Proofs of subsection 5.2}

\begin{proof}[Proof of Lemma \ref{Kunita_lem}]
From Schweizer (\cite{Schweizer1995}, proof of Lemma 2.7), $\xi_t^\Fm$ can be simply written 
\begin{equation*}
\xi_t^\Fm=\frac{\esp{\D\esph{\Fm|\f{t}}\D\St_t\big|\F_{t-1}}}{{\esp{(\D\St_t)^2\,|\,\f{t-1}}}} \; ; \; t\in\T.
\end{equation*}
The application of the Clark decomposition to $\esph{\Fm|\f{t}}-\esph{\Fm|\f{t-1}}$ yields
\begin{align*}
\xi_t^{\Fm}
&=\frac{1}{{\esp{(\D\St_t)^2\,|\,\f{t-1}}}}\bigg(\sum_{k\in\E}\mathbf E\Big[\widehat{\mathbf E}\big[\Dm_{(t,k)}\Fm|\f{t-1}\big]\D\Rm_{(t,k)}\St_{t-1}\big((b-a\rho)\D\Rm_{(t,1)}+a\D\Rm_{(t,-1)}\big)\,\Big|\,\f{t-1}\Big]\bigg)\\
&=:\frac{1}{\St_{t-1}v_{t}}\displaystyle\sum_{k\in\E}u_{t,k}\widehat{\mathbf E}\big[\Dm_{(t,k)}\Fm|\f{t-1}\big],
\end{align*}
where we have used that $\esp{\D\Rm_{(t,\l)}\D\Rm_{(t,k)}|\f{t-1}}=0$ for $\ell\neq k$ due to the orthogonality of the family $\mathcal R$. The sequence $v=(v_{t})_{t\in\T}$ is defined by $v_{t}=\lambda^2(b^2p+a^2q)+r^2-2\lambda(bp+aq)=(b-a\rho)^2\kappa_{1}+a^2\kappa_{-1}$, by using that $b\D\Zm_{(t,1)}+a\D\Zm_{(t,-1)}=(b-a\rho)\D\Rm_{(t,1)}+b\D\Rm_{(t,-1)}$. The sequence $u=(u_{t,k})_{(t,k)\in\Xb}$ is given by
\begin{equation*}
u_{t,1}=(b-a\rho)\kappa_1\quad\text{and}\quad u_{t,-1}=a\kappa_{-1}.
\end{equation*}
Last, let $w_{t,k}=u_{t,k}/v_t$ for $(t,k)\in\Xb$. The proof is complete.
\end{proof}

\noindent
\textbf{Acknowledgements.} This project has received funding from the European Union’s Horizon 2020 research and innovation programme under grant agreement $\Nm^\circ 811017$. I am also grateful to Giovanni Peccati for motivating discussions.


\begin{thebibliography}{10}

\bibitem{ArratiaGoldsteinGordon1989Twomoments}
R.~Arratia, L.~Goldstein, and L.~Gordon.
\newblock Two moments suffice for {P}oisson approximations: the {C}hen-{S}tein
  method.
\newblock {\em The Annals of Probability}, 17(1):9--25, 1989.

\bibitem{ArratiaGoldsteinGordon1990}
R.~Arratia, L.~Goldstein, and L.~Gordon.
\newblock Poisson approximation and the {C}hen-{S}tein method.
\newblock {\em Statistical Science}, pages 403--424, 1990.

\bibitem{AzmoodehCampesePoly}
E.~Azmoodeh, S.~Campese, and G.~Poly.
\newblock Fourth moment theorems for {M}arkov diffusion generators.
\newblock {\em Journal of Functional Analysis}, 266(4):2341–2359, Feb 2014.

\bibitem{AzmoodehMalicetMijoulePoly2016}
E.~Azmoodeh, D.~Malicet, G.~Mijoule, and G.~Poly.
\newblock Generalization of the {N}ualart-{P}eccati criterion.
\newblock {\em Annals of Probability}, 44(2):924--954, 2016.

\bibitem{barbour_introduction}
A.~D. Barbour and L.~H.~Y. Chen.
\newblock {\em An introduction to {S}tein's method}, volume~4 of {\em Lecture
  {Notes} {Series}}.
\newblock National University of Singapore, 2005.

\bibitem{BarbourChenLoh1992}
A.~D. Barbour, L.~H.~Y. Chen, and W.-L. Loh.
\newblock Compound {P}oisson approximation for nonnegative random variables via
  {S}tein's method.
\newblock {\em The Annals of Probability}, pages 1843--1866, 1992.

\bibitem{BarbourChryssaphinou2001}
A.~D. Barbour and O.~Chryssaphinou.
\newblock Compound {P}oisson approximation: a user's guide.
\newblock {\em Annals of Applied Probability}, pages 964--1002, 2001.

\bibitem{BarbourLarsJanson1992}
A.~D. Barbour, L.~Holst, and S.~Janson.
\newblock {\em Poisson approximation}, volume~2.
\newblock The Clarendon Press Oxford University Press, 1992.

\bibitem{BarbourUtev1998}
A.~D. Barbour and S.~Utev.
\newblock Solving the {S}tein equation in compound {P}oisson approximation.
\newblock {\em Advances in Applied Probability}, pages 449--475, 1998.

\bibitem{bichteler1987malliavin}
K.~Bichteler.
\newblock Malliavin calculus for processes with jumps.
\newblock {\em Stochastics Monographs}, 1987.

\bibitem{BourguinPeccati2016}
S.~Bourguin and G.~Peccati.
\newblock The {M}alliavin-{S}tein method on the {P}oisson space.
\newblock In {\em Stochastic analysis for Poisson point processes}, pages
  185--228. Springer, 2016.

\bibitem{BoyleTan}
P.~Boyle and T.~Wang.
\newblock Pricing of new securities in an incomplete market: the catch 22 of
  no-arbitrage pricing.
\newblock {\em Mathematical Finance}, 11(3):267--284, 2001.

\bibitem{Chen1974}
L.~H.~Y. Chen.
\newblock On the convergence of {P}oisson binomial to {P}oisson distributions.
\newblock {\em Ann. Probability}, 2(1):178--180, 1974.

\bibitem{Chernoff1981note}
H.~Chernoff.
\newblock A note on an inequality involving the {N}ormal distribution.
\newblock {\em The Annals of Probability}, pages 533--535, 1981.

\bibitem{cox1979option}
J.~Cox, S.~Ross, A, and M.~Rubinstein.
\newblock Option pricing: A simplified approach.
\newblock {\em Journal of financial Economics}, 7(3):229--263, 1979.

\bibitem{dalang1990equivalent}
R.~Dalang, A.~Morton, and W.~Willinger.
\newblock Equivalent martingale measures and no-arbitrage in stochastic
  securities market models.
\newblock {\em Stochastics: An International Journal of Probability and
  Stochastic Processes}, 29(2):185--201, 1990.

\bibitem{DaleyVere-Jones2007}
D.~J. Daley and D.~Vere-Jones.
\newblock {\em An introduction to the theory of point processes: volume II:
  general theory and structure}.
\newblock Springer Science \& Business Media, 2007.

\bibitem{DecreusefondSDM}
L.~Decreusefond.
\newblock The {S}tein-{D}irichlet-{M}alliavin method.
\newblock {\em ESAIM: Proc.}, 51:49--59, 2015.

\bibitem{Decreusefondflint}
L.~Decreusefond and I.~Flint.
\newblock Moment formulae for general point processes.
\newblock {\em Comptes Rendus Mathematique}, 352(4):357 -- 361, 2014.

\bibitem{DecreusefondHalconruy}
L.~Decreusefond and H.~Halconruy.
\newblock Malliavin and {D}irichlet structures for independent random
  variables.
\newblock {\em Stochastic Processes and their Applications},
  129(8):2611–2653, Aug 2019.

\bibitem{DecreusefondMoyal2012}
L.~Decreusefond and P.~Moyal.
\newblock {\em Stochastic modeling and analysis of telecom networks}.
\newblock John Wiley \& Sons, 2012.

\bibitem{Delbaen}
F.~Delbaen and W.~Schachermayer.
\newblock {\em The Mathematics of Arbitrage}.
\newblock Springer-Verlag Berlin Heidelberg, 2006.

\bibitem{DoblerPeccati2018}
C.~D{\"o}bler and G.~Peccati.
\newblock The fourth moment theorem on the {P}oisson space.
\newblock {\em Ann. Probab.}, 46(4):1878--1916, 07 2018.

\bibitem{DuerinckxGloriaOtto2020}
M.~Duerinckx, A.~Gloria, and F.~Otto.
\newblock The structure of fluctuations in stochastic homogenization.
\newblock {\em Communications in Mathematical Physics}, pages 1--48, 2020.

\bibitem{Erhardsson2005}
T.~Erhardsson.
\newblock Stein's method for {P}oisson and compound {P}oisson.
\newblock {\em An introduction to Stein’s method}, 4.

\bibitem{FollmerSchweizer1991}
H.~F{\"ollmer} and M.~Schweizer.
\newblock Hedging of contingent claims under incomplete information.
\newblock {\em Applied stochastic analysis}, 5(389-414):19--31, 1991.

\bibitem{FollmerSondermann}
H.~F{\"o}llmer and D.~Sondermann.
\newblock Hedging of non-redundant contingent claims.
\newblock 1985.

\bibitem{Frittelli2000}
M.~Frittelli.
\newblock The minimal entropy martingale measure and the valuation problem in
  incomplete markets.
\newblock {\em Mathematical finance}, 10(1):39--52, 2000.

\bibitem{GeskeGodboleSchaffnerSkolnickWallstrom}
M.~Geske, A.~Godbole, A.~Schaffner, A.~Skolnick, and G.~Wallstrom.
\newblock Compound {P}oisson approximations for word patterns under markovian
  hypotheses.
\newblock {\em Journal of applied probability}, pages 877--892, 1995.

\bibitem{Hakansson1979}
N.~Hakansson.
\newblock The fantastic world of finance: Progress and the free lunch.
\newblock {\em Journal of Financial and Quantitative Analysis}, pages 717--734,
  1979.

\bibitem{Halconruy2020}
H.~Halconruy.
\newblock {\em Calcul de Malliavin et structures de Dirichlet pour des
  variables al{\'e}atoires ind{\'e}pendantes}.
\newblock PhD thesis, Institut polytechnique de Paris, 2020.

\bibitem{jacod1998local}
J.~Jacod and A.~Shiryaev.
\newblock Local martingales and the fundamental asset pricing theorems in the
  discrete-time case.
\newblock {\em Finance and stochastics}, 2(3):259--273, 1998.

\bibitem{Krokowski2015}
K.~Krokowski.
\newblock Poisson approximation of {R}ademacher functionals by the
  {C}hen-{S}tein method and {M}alliavin calculus.
\newblock {\em arXiv preprint arXiv:1505.01417}, 2015.

\bibitem{lachiezereypeccati}
R.~Lachi{\`e}ze-Rey and G.~Peccati.
\newblock {Fine {G}aussian fluctuations on the {P}oisson space II: rescaled
  kernels, marked processes and geometric {U}-statistics}.
\newblock {\em {Stochastic Processes and their Applications}},
  123(12):4186--4218, December 2013.

\bibitem{Last2016}
G.~Last.
\newblock Stochastic analysis for {P}oisson processes.
\newblock In {\em Stochastic analysis for {P}oisson point processes}, pages
  1--36. Springer, 2016.

\bibitem{LastPeccatiSchulte2016}
G.~Last, G.~Peccati, and M.~Schulte.
\newblock Normal approximation on {P}oisson spaces: {M}ehler{\textquoteright}s
  formula, second order {P}oincar{\'e} inequalities and stabilization.
\newblock August 2016.

\bibitem{LastPenrose2011}
G.~Last and M.~Penrose.
\newblock Martingale representation for {P}oisson processes with applications
  to minimal variance hedging.
\newblock {\em Stochastic Processes and their Applications}, 121(7):1588 --
  1606, 2011.

\bibitem{LastPenrose2017}
G.~Last and M.~Penrose.
\newblock {\em Lectures on the {P}oisson Process}.
\newblock Institute of Mathematical Statistics Textbooks. Cambridge University
  Press, 2017.

\bibitem{Ledoux2012}
M.~Ledoux.
\newblock Chaos of a {M}arkov operator and the fourth moment condition.
\newblock {\em The Annals of Probability}, 40(6):2439--2459, 2012.

\bibitem{HarrisonKreps1979}
J~Michael M.~Harrison and D.~Kreps.
\newblock Martingales and arbitrage in multiperiod securities markets.
\newblock {\em Journal of Economic theory}, 20(3):381--408, 1979.

\bibitem{Metivier1982}
M.~M{\'e}tivier.
\newblock Semimartingales, volume 2 of de gruyter studies in mathematics, 1982.

\bibitem{Nash1958continuity}
J.~Nash.
\newblock Continuity of solutions of parabolic and elliptic equations.
\newblock {\em American Journal of Mathematics}, 80(4):931--954, 1958.

\bibitem{Nourdin_2008}
I.~Nourdin and G.~Peccati.
\newblock Stein’s method on {W}iener chaos.
\newblock {\em Probability Theory and Related Fields}, 145(1-2):75–118, June
  2008.

\bibitem{NourdinPeccati2009}
I.~Nourdin and G.~Peccati.
\newblock Noncentral convergence of multiple integrals.
\newblock {\em The Annals of Probability}, 37(4):1412--1426, 2009.

\bibitem{Nourdin:2012fk}
I.~Nourdin and G.~Peccati.
\newblock {\em Normal Approximations with Malliavin Calculus: From Stein's
  Method to Universality}.
\newblock Cambridge University Press, 2012.

\bibitem{NourdinPeccatiReinert2009}
I.~Nourdin, G.~Peccati, and G.~Reinert.
\newblock Invariance principles for homogeneous sums: universality of
  {G}aussian {W}iener chaos.
\newblock {\em The Annals of Probability}, 38(5):1947--1985, 2010.

\bibitem{NourdinPeccatiReinert}
I.~Nourdin and G~Peccati, G.and~Reinert.
\newblock Stein's method and stochastic analysis of {R}ademacher functionals.
\newblock {\em Electronic Journal of Probability}, 15:1703--1742, 2010.

\bibitem{nualart.book}
D.~Nualart.
\newblock {\em The {M}alliavin Calculus and Related Topics}.
\newblock Springer-Verlag, 1995.

\bibitem{Nualart2014finitechaos}
D.~Nualart.
\newblock Normal approximation on a finite {W}iener chaos.
\newblock In Dan Crisan, Ben Hambly, and Thaleia Zariphopoulou, editors, {\em
  Stochastic Analysis and Applications 2014}, pages 377--395, Cham, 2014.
  Springer International Publishing.

\bibitem{NualartSchoutens}
D.~Nualart and W.~Schoutens.
\newblock Chaotic and predictable representations for {L}{\'e}vy processes.
\newblock {\em Stochastic Processes and their Applications}, 90(1):109 -- 122,
  2000.

\bibitem{nualart88_1}
D.~Nualart and J.~Vives.
\newblock Anticipative calculus for the {P}oisson process based on the {F}ock
  space.
\newblock In {\em S\'eminaire de probabilit\'es XXIV}, pages 154--165.
  Springer-Verlag, 1988.

\bibitem{DiNunnoOksendalProske2004}
G.~Di Nunno, B.~Oksendal, and F.~Proske.
\newblock White noise analysis for {L}{\'e}vy processes.
\newblock {\em Journal of Functional Analysis}, 206:109--148, 2004.

\bibitem{Peccati2011chen}
G.~Peccati.
\newblock The {C}hen-{S}tein method for {P}oisson functionals.
\newblock {\em arXiv preprint arXiv:1112.5051}, 2011.

\bibitem{peccati2010}
G.~Peccati, J.~L. Solé, M.~S. Taqqu, and F.~Utzet.
\newblock Stein’s method and {N}ormal approximation of {P}oisson functionals.
\newblock {\em Ann. Probab.}, 38(2):443--478, 03 2010.

\bibitem{privault93}
N.~Privault.
\newblock Chaotic and variational calculus in discrete and continuous time for
  the {P}oisson process.
\newblock {\em Stoch. \& Stoch. Rep.}, 51:83--109, 1994.

\bibitem{Privault_stochastica}
N.~Privault.
\newblock {\em Stochastic {A}nalysis in {D}iscrete and {C}ontinuous Settings}.
\newblock 1982. Springer-Verlag Berlin Heidelberg, 2009.

\bibitem{Privault2013market}
N.~Privault.
\newblock {\em Stochastic finance: An introduction with market examples}.
\newblock CRC Press, 2013.

\bibitem{Privault2018jump}
N.~Privault.
\newblock Chapter 19 {S}tochastic calculus for jump processes.
\newblock 2018.

\bibitem{PrivaultTorrisi2015}
N.~Privault and G.~Torrisi.
\newblock The {S}tein and {C}hen-{S}tein methods for functionals of
  non-symmetric {B}ernoulli processes.
\newblock {\em ALEA Lat. Am. J. Probab. Math. Stat}, 12(1):309--356, 2015.

\bibitem{ReinertSchbath1998}
G.~Reinert and S.~Schbath.
\newblock Compound {P}oisson and {P}oisson process approximations for
  occurrences of multiple words in markov chains.
\newblock {\em Journal of Computational Biology}, 5(2):223--253, 1998.

\bibitem{RendlemanBartter1979}
R.~Rendleman and B.~Bartter.
\newblock Two-state option pricing.
\newblock {\em The Journal of Finance}, 34(5):1093--1110, 1979.

\bibitem{RobinSchbath}
S.~Robin and S.~Schbath.
\newblock Numerical comparison of several approximations of the word count
  distribution in random sequences.
\newblock {\em Journal of Computational Biology}, 8(4):349--359, 2001.

\bibitem{runggaldier2006}
W.~Runggaldier.
\newblock {\em Portfolio optimization in discrete time}.
\newblock Accademia delle Scienze dell'Istituto di Bologna, 2006.

\bibitem{Schachermayer1992hilbert}
W.~Schachermayer.
\newblock A {H}ilbert space proof of the fundamental theorem of asset pricing
  in finite discrete time.
\newblock {\em Insurance: Mathematics and Economics}, 11(4):249--257, 1992.

\bibitem{Schbath1997}
S.~Schbath.
\newblock Compound poisson approximation of word counts in {D}{N}{A} sequences.
\newblock {\em ESAIM: probability and statistics}, 1:1--16, 1997.

\bibitem{Schulte2016Kolomogorov}
M.~Schulte.
\newblock {N}ormal approximation of {P}oisson functionals in {K}olmogorov
  distance.
\newblock {\em Journal of theoretical probability}, 29(1):96--117, 2016.

\bibitem{Schweizer1991semimartingales}
M.~Schweizer.
\newblock Option hedging for semimartingales.
\newblock {\em Stochastic Processes and their Applications}, 37(2):339 -- 363,
  1991.

\bibitem{Schweizer1992}
M.~Schweizer.
\newblock Mean-variance hedging for general claims.
\newblock {\em The annals of applied probability}, pages 171--179, 1992.

\bibitem{Schweizer1995}
M.~Schweizer.
\newblock Variance-optimal hedging in discrete time.
\newblock {\em Mathematics of Operations Research}, 20(1):1--32, 1995.

\bibitem{stein1972}
C.~Stein.
\newblock A bound for the error in the normal approximation to the distribution
  of a sum of dependent random variables.
\newblock In {\em Proceedings of the Sixth Berkeley Symposium on Mathematical
  Statistics and Probability, Volume 2: Probability Theory}, pages 583--602,
  Berkeley, Calif., 1972. University of California Press.

\bibitem{Torrisi2017Poisson}
G.~Torrisi.
\newblock Poisson approximation of point processes with stochastic intensity,
  and application to nonlinear {H}awkes processes.
\newblock In {\em Annales de l'Institut Henri Poincar{\'e}, Probabilit{\'e}s et
  Statistiques}, volume~53, pages 679--700. Institut Henri Poincar{\'e}, 2017.

\bibitem{Viquez2018normal}
J.~J. V{\'\i}quez~R.
\newblock Normal convergence using {M}alliavin calculus with applications and
  examples.
\newblock {\em Stochastic Analysis and Applications}, 36(2):341--372, 2018.

\bibitem{Wu:2000lr}
L.~Wu.
\newblock A new modified logarithmic {S}obolev inequality for {P}oisson point
  processes and several applications.
\newblock {\em Probability Theory and Related Fields}, 118(3):427--438, 2000.

\bibitem{Zheng2017NormalAA}
G.~Zheng.
\newblock Normal approximation and almost sure central limit theorem for
  non-symmetric rademacher functionals.
\newblock 2017.

\end{thebibliography}
\end{document}